\pdfoutput=1
\RequirePackage{ifpdf}
\ifpdf % We are running pdfTeX in pdf mode
\documentclass[pdftex]{sigma}
\else
\documentclass{sigma}
\fi

\usepackage{mathrsfs}

\usepackage{tikz,combelow,mathtools,wasysym,enumerate}

\usetikzlibrary{arrows,positioning,cd,calc}

\numberwithin{equation}{section}

\newtheorem{Theorem}{Theorem}[section]
\newtheorem{Proposition}[Theorem]{Proposition}

\newtheorem{Corollary}[Theorem]{Corollary}
\newtheorem{Lemma}[Theorem]{Lemma}

\theoremstyle{definition}
\newtheorem{Remark}[Theorem]{Remark}
\newtheorem{Example}[Theorem]{Example}
\newtheorem{Definition}[Theorem]{Definition}

\DeclarePairedDelimiter\ceil{\lceil}{\rceil}

\newcommand{\half}{\frac{1}{2}}
\newcommand{\qd}[1]{ \mathfrak{Diff}_{q^{#1}} }
\newcommand{\valpha}{| \alpha \rangle}

\newcommand{\vac}{|0 \rangle}
\newcommand{\n}{|l \rangle}
\newcommand{\qaD}{\mathrm{Diff}_q^{\, \mathrm{A}}}
\newcommand{\qD}{\mathfrak{Diff}_q}
\newcommand{\qDn}{\mathfrak{Diff}_{q^n}}

\newcommand{\Hem}{\mathfrak{Heis}}

\newcommand{\psia}{\psi_{({a})}}
\newcommand{\psiaa}{\psi^*_{( \! a \! )}}
\newcommand{\psib}{\psi^*_{( b )}}
\newcommand{\psibb}{\psi_{( b )}}
\newcommand{\phib}{\phi_{b}}

\newcommand{\T}{\tilde{T}}
\newcommand{ \W }{ \mathcal{W}_{q} ( \mathfrak{sl}_n ) }
\newcommand{ \Wgl }{ \mathcal{W}_{q} ( \mathfrak{gl}_n ) }
\newcommand{ \twW }{ \mathcal{W}_{q} ( \mathfrak{sl}_n, n_{tw} ) }
\newcommand{ \twWgl }{ \mathcal{W}_{q} ( \mathfrak{gl}_n, n_{tw} ) }
\newcommand{\tT}{T^{tw}}
\newcommand{\Q}{\mathbf{Q}_{(n)}}
\renewcommand{\P}{\mathbf{P}_{(n)}}
\newcommand{\dQ}{\mathbf{Q}^d_{(n)}}
\newcommand{\Torgln}{\qD({\mathfrak{gl}_n})}
\newcommand{\bE}{\bar{E}}
\newcommand{\tE}{\tilde{E}}
\newcommand{\tF}{\tilde{F}}
\newcommand{\tH}{\tilde{H}}

\newcommand{\dWFock}{\mathcal{F}^{\mathcal{W}_{q} ( \mathfrak{sl}_{nd}, n'd )}_{u_1, \dots, u_d}}
\newcommand{\nWFock}{\mathcal{F}^{\W}_{u_1, \dots, u_d}}
\newcommand{\WFock}{\mathcal{F}^{\twW}_{u_1, \dots, u_d}}

\newcommand{\tphi}{\tilde{\varphi}}
\newcommand{\Ver}{ \mathcal{V}^{\twW}_{\lambda_1, \dots, \lambda_d} }
\newcommand{\Vergl}{\mathcal{V}^{\twWgl}_{\lambda_1, \dots, \lambda_d} }
\newcommand{\VerglNT}{\mathcal{V}^{\Wgl}_{\lambda_1, \dots, \lambda_d} }
\newcommand{\vacgl}{\left|\bar{\lambda} \right\rangle_{\mathfrak{gl}} }
\newcommand{\vacsl}{ \left| \bar{\lambda} \right\rangle_{\mathfrak{sl}} }
\newcommand{\vacF}{ \left| \bar{u} \right\rangle}

\newcommand{\mw}{\underline{w}}
\newcommand{\WqD}{\mathcal{S}}
\newcommand{\qDW}{\mathcal{P}}
\newcommand{\tbE}{\tilde{\mathbb{E}}^{(2)}}
\newcommand{\siL}{ \Lambda^{^{\! \infty/2}} \,}
\newcommand{\UH}{\otimes U(\mathfrak{Heis})}
\newcommand{\WH}{\twW \UH}
\newcommand{\WHnd}{\mathcal{W}(\mathfrak{sl}_{nd}, n'd) \UH}
\newcommand{\e}{\varepsilon}
\newcommand{\se}{\mathtt{e}}
\newcommand{\ha}{\tilde{H}}
\DeclareMathOperator{\id}{id}
\DeclareMathOperator{\ch}{ch}
\DeclareMathOperator{\Span}{span}
\DeclareMathOperator{\End}{End}
\newcommand{\I}{\mathrm{I}}
\newcommand{\fact}{\frac{ns}{d}! \,}

\begin{document}
\allowdisplaybreaks

\newcommand{\arXivNumber}{1906.00600}

\renewcommand{\PaperNumber}{077}

\FirstPageHeading

\ShortArticleName{Twisted Representations of Algebra of $q$-Difference Operators}

\ArticleName{Twisted Representations of Algebra\\ of $\boldsymbol{q}$-Difference Operators, Twisted $\boldsymbol{q}$-$\boldsymbol{W}$ Algebras\\ and Conformal Blocks}

\Author{Mikhail BERSHTEIN~$^{\dag^1\dag^2\dag^3\dag^4\dag^5}$ and Roman GONIN~$^{\dag^2\dag^3}$}

\AuthorNameForHeading{M.~Bershtein and R.~Gonin}

\Address{$^{\dag^1}$~Landau Institute for Theoretical Physics, Chernogolovka, Russia}
\EmailDD{\href{mailto:mbersht@gmail.com}{mbersht@gmail.com}}

\Address{$^{\dag^2}$~Center for Advanced Studies, Skolkovo Institute of Science and Technology, Moscow, Russia}
\EmailDD{\href{mailto:roma-gonin@yandex.ru}{roma-gonin@yandex.ru}}

\Address{$^{\dag^3}$~National Research University Higher School of Economics, Moscow, Russia}

\Address{$^{\dag^4}$~Institute for Information Transmission Problems, Moscow, Russia}
\Address{$^{\dag^5}$~Independent University of Moscow, Moscow, Russia}

\ArticleDates{Received November 22, 2019, in final form August 01, 2020; Published online August 16, 2020}

\Abstract{We study certain representations of quantum toroidal $\mathfrak{gl}_1$ algebra for $q=t$. We construct explicit bosonization of the Fock modules $\mathcal{F}_u^{(n',n)}$ with a nontrivial slope $n'/n$. As a vector space, it is naturally identified with the basic level 1 representation of affine $\mathfrak{gl}_n$. We also study twisted $W$-algebras of $\mathfrak{sl}_n$ acting on these Fock modules. As an application, we prove the relation on $q$-deformed conformal blocks which was conjectured in the study of $q$-deformation of isomonodromy/CFT correspondence.}

\Keywords{quantum algebras; toroidal algebras; $W$-algebras; conformal blocks; Nekrasov partition function; Whittaker vector}

\Classification{17B67; 17B69; 81R10}

\section{Introduction}
{\bf Toroidal algebra.} Representation theory of quantum toroidal algebras has been actively de\-veloped in recent years. This theory has numerous applications, including geometric representation theory and AGT relation~\cite{N16}, topological strings~\cite{AFS}, integrable systems, knot theory~\cite{GN13}, and combinatorics~\cite{CM15}.

In this paper we consider only the quantum toroidal $\mathfrak{gl}_1$ algebra; we denote it by $U_{q,t}\big(\ddot{\mathfrak{gl}}_1\big)$. The algebra depends on two parameters $q$, $t$ and has PBW generators $E_{k,l}$, $(k,l)\in \mathbb{Z}^2$ and central generators~$c'$,~$c$~\cite{Burban}. In the main part of the text we consider only the case $q=t$, where toroidal algebra becomes the universal enveloping of the Lie algebra with these generators~$E_{k,l}$,~$c'$,~$c$ and the relation
\begin{gather*}
[E_{k,l} , E_{r,s} ] = \big( q^{(sk-lr)/2}-q^{(lr-sk)/2} \big) E_{k+r, l+s} + \delta_{k, -r} \delta_{l, -s} ( c' k + c l ).
\end{gather*}
We denote this Lie algebra by $\qD$, since there is a homomorphism from this algebra to the algebra of $q$-difference operators generated by $D$, $x$ with the relation $Dx=qxD$; namely $E_{k,l} \mapsto q^{kl/2} x^l D^k$.

There is another presentation of the algebra $\qD$ (and more generally $U_{q,t}\big(\ddot{\mathfrak{gl}}_1\big)$) using the Chevalley generators $E(z) = \sum\limits_{k \in \mathbb{Z}} E_{1,k} z^{-k}$, $F(z) = \sum\limits_{k \in \mathbb{Z}} E_{-1, k} z^{-k}$, $H(z) = \sum\limits_{k \neq 0} E_{0, k} z^{-k}$, see, e.g.,~\cite{Ts}.

In this paper we deal with the Fock representations of $\qD$; to be more precise there is a~family $\mathcal{F}_u$ of Fock modules, depending on the parameter $u$ (see Proposition~\ref{prop:sec3 Boson Fock} for a construction of~$\mathcal{F}_u$). They are just Fock representations of the Heisenberg algebra generated by~$E_{0,k}$. The images of~$E(z)$ and~$F(z)$ are vertex operators. A construction of this type is usually called \emph{bosonization}.

It was shown in \cite{FHHSY,N16} that the image of toroidal algebra $U_{q,t}\big(\ddot{\mathfrak{gl}}_1\big)$ in the endomorphisms of the tensor product of $n$ Fock modules is the deformed $W$-algebra for $\mathfrak{gl}_n$. There is the so-called conformal limit $q,t\rightarrow 1$, in which deformed $W$-algebras go to vertex algebras. These vertex algebras are tensor products of the Heisenberg algebra and the $W$-algebras of $\mathfrak{sl}_n$. In the case $q=t$, the central charge of the corresponding $W$-algebra of $\mathfrak{sl}_n$ is equal to $n-1$. These $W$-algebras appear in the study of isomonodromy/CFT correspondence (see \cite{GIL12,GM16}). This is one of the motivations of our paper.

The $q$-deformation of the isomonodromy/CFT correspondence was proposed in \cite{BS:2016:1,BGT17,JNS17}. The main statement is an explicit formula for the $q$-isomonodromic tau function as an infinite sum of conformal blocks for certain deformed $W$-algebras with $q=t$. In general, these tau functions are complicated, but there are special cases (corresponding to algebraic solutions) where these tau functions are very simple \cite{BGM, BS:2016:1}. These cases should correspond to special representations of $q$-deformed $W$-algebras. The construction of such representation is one of the purposes of this paper.

{\bf Twisted Fock modules.}
There is a natural action of ${\rm SL}(2,\mathbb{Z})$ on $\qD$. We will parametrize $\sigma\in {\rm SL}(2,\mathbb{Z})$ by
\begin{gather*}%\label{sigma sss}
\sigma = \begin{pmatrix}
m' & m \\
n' & n
\end{pmatrix}.
\end{gather*}
Then $\sigma$ acts as
\begin{gather*}
\sigma( E_{k,l} ) = E_{ m' k + m l, n' k + n l },\qquad
\sigma (c') = m' c' + n' c,\qquad
\sigma (c) = m c' + n c.
\end{gather*}
For any $\qD$ module $M$ and $\sigma \in {\rm SL}(2,\mathbb{Z})$, we denote by $M^\sigma$ the module twisted by the automorphism $\sigma$ (see Definition \ref{twisted}). The twisted Fock modules depend only on $n$ and $n'$ (up to isomorphism). These numbers are the values of the central generators $c$ and $c'$, correspondingly, acting on $\mathcal{F}_u^\sigma$. Therefore we will also use the notation $\mathcal{F}_u^{(n',n)}$ for $\mathcal{F}_u^\sigma$. Twisted Fock modules $\mathcal{F}_u^\sigma$ (for generic $q, t$) were used, for example, in \cite{AFS} and \cite{GN15}.

In Section \ref{Section with answers} we construct explicit bosonization of the twisted Fock modules $\mathcal{F}_u^\sigma$ for $q=t$. Actually, we give three constructions: the first one in terms of $n$-fermions (see Theorem \ref{TheFermionicTh}), the second one in terms of $n$-bosons (see Theorem \ref{Th:qD Boson}) and the third one in terms of one twisted boson (see Theorem \ref{Th:StrangeBoson}) (here, for simplicity, we assume that $n>0$). In other words, any twisted Fock module will be identified with the basic module for $\widehat{\mathfrak{gl}}_n$; these two bosonizations correspond to homogeneous~\cite{FK} and principal~\cite{KKLW,LW} constructions.

The construction of the bosonization is nontrivial, because it is given in terms of Chevalley generators (note that the ${\rm SL}(2,\mathbb{Z})$ action is not easy to describe in terms of Chevalley generators). The appearance of affine $\mathfrak{gl}_n$ is in agreement with the Gorsky--Negu\c{t} conjecture~\cite{GN15}. More specifically, it was conjectured in \cite{GN15} that there exists an action (with certain properties) of $U_{p^{1/2}} \big(\widehat{\mathfrak{gl}}_n\big)$ on $\mathcal{F}_u^{\sigma}$ for $p=q/t\neq1$; we expect this to be $p$-deformation of the $\widehat{\mathfrak{gl}}_n$-action constructed in this paper.

It is instructive to look at the formulas in the simplest examples. For simplicity, we give here only formulas for $E(z)$. Here we introduce the notation in a sloppy way (for details see Sections~\ref{level1} and~\ref{Section with answers}).

\begin{Example}
In the standard case $n=1$, $n'=0$ we have
\begin{gather*}
E(z) =u q^{-1/2} z \psi\big( q^{-1/2} z\big) \psi^* \big( q^{1/2} z\big)= \frac{u}{1-q}: \! \exp \big( \phi\big(q^{1/2} z\big) - \phi\big(q^{-1/2} z\big) \big) \! :,
\end{gather*}
where $\psi(z)$, $\psi^*(z)$ are complex conjugate fermions (see Section~\ref{fermi}), $\phi(z)=\sum\limits_{j \neq 0} a [j] z^{-j}/j$ is a~boson and $a[j]$ are generators of the Heisenberg algebra with relation $[a[j],a[j']]=j\delta_{j+j',0}$ (see Section~\ref{boson}).
\end{Example}

\begin{Example}
The first nontrivial case is given by $n=2$, $n'=1$. We have three formulas (corresponding to Theorems~\ref{TheFermionicTh}, \ref{Th:qD Boson} and \ref{Th:StrangeBoson}):
\begin{gather}
E(z) = u^{\frac{1}{2}}q^{-\frac14} \left( z^2\psi_{(0)} \big( q^{-1/2} z\big) \psi_{(1)}^* \big(q^{1/2} z\big) + z\psi_{(1)} \big( q^{-1/2} z\big) \psi_{(0)}^* \big(q^{1/2} z\big) \right), \label{eq:Etwferm}
\\
E(z) = u^{\frac{1}{2}}q^{-\frac14}\left( z^2 :\!\exp \left( \phi_1\big(q^{1/2} z\big) - \phi_0\big(q^{-1/2} z\big) \right)\!: \right.\nonumber\\
\left. \hphantom{E(z) =}{} +z^{} :\!\exp \left( \phi_0\big(q^{1/2} z\big) - \phi_1\big(q^{-1/2} z\big) \right)\!: \right) (-1)^{a_{0}[0]},\label{eq:Etwboson} \\
E(z) = \frac{z^{\frac12} u^{\frac{1}{2}}}{2\big(1-q^{\frac12}\big)}\left( :\!\exp \! \left( \sum_{k\neq 0} \frac{q^{-k/4}-q^{k/4}}{k} a_{k} z^{-k/2} \right)\!: \right.\nonumber\\
\left.\hphantom{E(z) =}{} - :\!\exp\!\left( \sum_{k \neq 0} (-1)^{k}\frac{q^{-k/4}-q^{k/4}}{k} a_{k} z^{-k/2} \right)\!:\right). \label{eq:Etwbosontw}
\end{gather}
Here $\psi_{(0)}(z)$, $\psi_{(0)}^*(z)$ and $\psi_{(1)}(z)$, $\psi_{(1)}^*(z)$ are anticommuting pairs of complex conjugate fermions (see Section~\ref{section: The Fermionic Th}), $\phi_{b}(z)=\sum\limits_{j \neq 0} a_{b} [j] z^{-j}/j + Q + a_b[0] \log z$ are commuting bosons, and $a_{b}[j]$ are generators of the Heisenberg algebra with the relation $[a_{b}[j],a_{b'}[j']]=j\delta_{j+j',0}\delta_{b,b'}$ (see Section~\ref{subsection: TW qDboson}). The generators $a_k$ in \eqref{eq:Etwbosontw} satisfy $[a_k,a_{k'}]=k\delta_{k+k',0}$.

The relation between \eqref{eq:Etwferm} and \eqref{eq:Etwboson} is a standard boson-fermion correspondence. In the right-hand side of formula \eqref{eq:Etwbosontw} we have only one Heisenberg algebra with generators~$a_{k}$, but since we have both integer and half-integer powers of $z$, one can think that we have a boson with a nontrivial monodromy. This is the reason for the term `twisted boson'; we will also call this construction \emph{strange bosonization}. Note that half-integer powers of~$z$ cancel in the right-side of~\eqref{eq:Etwbosontw}.
\end{Example}

We present two different proofs of Theorems \ref{TheFermionicTh}, \ref{Th:qD Boson} and \ref{Th:StrangeBoson}. The first one is given in Section~\ref{S:sublattice} and is based on the following idea. For any
full rank sublattice $\Lambda\in \mathbb{Z}^2$ of index $n$, we have a subalgebra $\qd{1/n}^{\Lambda} \subset \qd{1/n}$, which is spanned by $E_{a,b}$ for $(a,b) \in \Lambda$ and central elements $c$, $c'$. The algebra $\qd{1/n}^{\Lambda}$ is isomorphic to $\qD$; the isomorphism depends on the choice of a positively oriented basis $v_1$, $v_2$ in $\Lambda$. Denote this isomorphism by $\phi_{v_1,v_2}$.

If the basis $v_1$, $v_2$ is such that $v_1=(N,0)$, $v_2=(R,d)$, then the restriction of the Fock module~$\mathcal{F}_u$ on $\phi_{v_1,v_2}(\qD)$ is isomorphic to the sum of tensor products of the Fock modules
\begin{align} \label{eq:Fu decomp intro}
\left. \mathcal{F}_{u^{1/N}} \right|_{\phi_{v_1,v_2}(\qD)} \cong \bigoplus_{ l \in \mathbf{Q}_{(d)}} \ \mathcal{F}_{u q^{ rl_0}} \otimes \cdots \otimes \mathcal{F}_{u q^{r\left(\frac{\alpha}{n}+ l_{\alpha} \right)
}} \otimes \cdots \otimes \mathcal{F}_{u q^{r \left(\frac{d-1}{d}+ l_{d-1} \right)}}
\end{align}
where $r=\gcd(N,R)$ and $\mathbf{Q}_{(d)}=\{(l_0,\dots,l_{d-1})\in \mathbb{Z}^d\,|\,\sum l_i=0\}$. If we choose basis $w_1$, $w_2$ in $\Lambda$ which differs from~$v_1$,~$v_2$ by $\sigma\in {\rm SL}(2,\mathbb{Z})$, we get an analogue of decomposition \eqref{eq:Fu decomp intro} with right-hand side given by a sum of tensor products of the twisted Fock modules. For the basis $w_1=(r,n_{tw})$, $w_2=(0,n)$, we write formulas for Chevalley generators of $\qD = \qd{1/n}^{\Lambda}$ using either initial fermion or initial boson for~$\mathcal{F}_u$. Applying this for the lattices with $d=1$, we get Theorems~\ref{TheFermionicTh},~\ref{Th:qD Boson} and~\ref{Th:StrangeBoson}.

The secondproof of these theorems is based on the semi-infinite construction. Let $V_{u}$ denote the representation of the algebra $\qD$ in a vector space with basis $x^{k-\alpha}$ for $k \in \mathbb{Z}$, where $\qD$ acts as $q$-difference operators (see Definition \ref{eq: sec3 eval def}). This representation is called \emph{vector} (or \emph{evaluation}) representation; the parameter $u$ is equal to $q^{-\alpha}$. The Fock module $\mathcal{F}_u$ is isomorphic to $\Lambda^{\infty/2+0}(V_u)\subset \siL (V_u)$. After the twist, we get a semi-infinite construction of $\mathcal{F}_u^{\sigma} \subset \big(\siL V_u\big)^{\sigma}= \siL (V_u^{\sigma})$. Note that conjecturally the semi-infinite construction of $\mathcal{F}_u^\sigma$ can be generalized for $q\neq t$ (cf.~\cite{FFJMMa}).

{\bf Twisted $\boldsymbol{W}$-algebras.} Denote by $\qD^{\geqslant0}$ the subalgebra of $\qD$ generated by $c$ and $E_{a,b}$, for $a\geqslant 0$. There is an another set of generators $E^k[j]$ of the completion of the $U\big(\qD^{\geqslant 0}\big)$, defined by the formula $\sum\limits_{j \in \mathbb{Z}} E^k[j]z^{-j}= ( E (z) )^k$ (see Appendix~\ref{nonorm} for the definition of the power of~$E(z)$). The currents $H(z)$ and $E^k(z)$ for $k\in \mathbb{Z}_{>0}$ satisfy relations of the $q$-deformed $W$-algebra of $\mathfrak{gl}_\infty$ (see~\cite{N16}). We denote this algebra by $\mathcal{W}_q(\mathfrak{gl}_\infty)$.

There is an ideal $J_{\mu,d}^{\geqslant0}$ in $U\big(\qD^{\geqslant0}\big)=\mathcal{W}_q(\mathfrak{gl}_\infty)$ which acts by zero on any tensor product $\mathcal{F}_{u_1}\otimes \dots\otimes \mathcal{F}_{u_d}$, here $\mu= \frac{1}{1-q} (u_1 \cdots u_d)^{{1}/{n}}$. This ideal is generated by relations $c=d$ and
\begin{gather*}
E^d(z)= {\mu^d}{d!} \exp ( \varphi_- (z) ) \exp ( \varphi_+ (z) ),
\end{gather*}
where \begin{gather*}\varphi_{-} (z) = \sum_{j>0} \frac{q^{-j/2}-q^{j/2}}{j} E_{0,-j} z^j,\qquad \varphi_{+} (z) = - \sum_{j>0} \frac{q^{j/2}-q^{-j/2}}{j} E_{0,j} z^{-j}.\end{gather*}
The quotient of $\mathcal{W}_q(\mathfrak{gl}_{\infty})/J_{\mu,d}^{\geqslant0}$ is the $q$-deformed $W$-algebra of $\mathfrak{gl}_d$. We denote this algebra by $\mathcal{W}_q(\mathfrak{gl}_d)$; it does no depend on $\mu$ (up to isomorphism) and acts on any tensor product $\mathcal{F}_{u_1}\otimes \dots\otimes \mathcal{F}_{u_d}$ (see~\cite{FHSSY,N16}).

In Section \ref{section: q-W} we study a tensor product of the twisted Fock modules ${\mathcal{F}^{\sigma}_{u_1} \otimes \cdots \otimes \mathcal{F}_{u_d}^{\sigma}}$. We prove that the ideal $J_{\mu,nd,n'd}^{\geqslant0}$ generated by relations $c=nd$ and
\begin{gather*}E^{nd}(z)= z^{n'd} {\mu^{nd}}{(nd)!} \exp ( \varphi_- (z) ) \exp ( \varphi_+ (z) )\end{gather*} acts by zero for
$\mu = (-1)^{{1}/{n}} \frac{q^{-{1}/{2n}}}{q^{1/2}-q^{-1/2}} (u_1 \cdots u_d)^{{1}/{nd}}$.We denote the quotient $\mathcal{W}_q(\mathfrak{gl}_{\infty})/J_{\mu,nd,n'd}^{\geqslant0}$ by $\mathcal{W}_q(\mathfrak{gl}_{nd},n'd)$ and call it the \emph{twisted $q$-deformed $W$-algebra of $\mathfrak{gl}_{nd}$}.

There exists another description of the above using the $q$-deformed $W$-algebra of $\mathfrak{sl}_n$ introduced in \cite{FF}. Define $T_k[j]$ by the formula
\begin{gather*}
T_k (z) =\sum T_k[j] z^{-j} = \frac{\mu^{-k}}{k!} \exp \left(-\frac{k}{c} \varphi_- (z) \right) E^k (z) \exp \left( - \frac{k}{c} \varphi_+ (z) \right).
\end{gather*}
The generators $T_k[j]$ are elements of a localization of the completion of $U\big(\qD^{\geqslant0}\big)$. These generators commute with $H_i$ and satisfy certain quadratic relations. The algebra generated by~$T_k[j]$ is denoted by $\mathcal{W}_q(\mathfrak{sl}_{\infty})$.

There is an ideal in $\mathcal{W}_q(\mathfrak{sl}_{\infty})$ which acts by zero on any tensor product $\mathcal{F}_{u_1}\otimes \dots\otimes \mathcal{F}_{u_d}$. This ideal contains relations $c=d$, $T_d(z)=1$, and $T_{d+k}(z)=0$ for $k>0$. The quotient is a standard $W$-algebra $\mathcal{W}_q(\mathfrak{sl}_d)$ \cite{FF} (see also Definition \ref{def: non-twisted W}). We have a relation $\mathcal{W}_q(\mathfrak{gl}_d)=\mathcal{W}_q(\mathfrak{sl}_d) \UH$, where $\mathfrak{Heis}$ is the Heisenberg algebra generated by $E_{0,j}$.

In the case of a product of the twisted Fock modules ${\mathcal{F}^{\sigma}_{u_1} \otimes \cdots \otimes \mathcal{F}_{u_d}^{\sigma}}$ the situation is similar. The corresponding ideal contains the relations $T_{nd}(z)=z^{n'd}$, $T_{nd+k}(z)=0$ for $k>0$. We present the quotient in terms of the generators $T_1(z),\dots,T_{nd}(z)$ and relations (this is Theorem~\ref{Th:Wiso}). We call the algebra with such generators and relations by twisted $W$-algebra $\mathcal{W}_q(\mathfrak{sl}_{nd},n'd)$; see Definition~\ref{def: twisted W}.\footnote{One can find a definition of $\mathcal{W}_{q,p}(\mathfrak{sl}_2, 1)$ in \cite[equations~(37)--(38)]{Sh}.} The quadratic relations in the algebra $\mathcal{W}_q(\mathfrak{sl}_{nd},n'd)$ are the same as in the untwisted case (see equation~\eqref{relation:modesW1}--\eqref{relation:modesWn-1}), the only difference lies in the relation $T_{nd}(z)=z^{n'd}$.

The algebra $\mathcal{W}_q(\mathfrak{sl}_{nd},n'd)$ is graded, with $\deg T_k[j]=j+\frac{n'k}{n}$. Let us rename the generators by $T^{tw}_k[r]=T_k\big[r-\frac{n'k}{n}\big]$, for $r \in \frac{n'k}{n}+\mathbb{Z}$. The presentations of the algebra $\mathcal{W}_q(\mathfrak{sl}_{nd},n'd)$ in terms of generators~$T^{tw}_k[r]$ and the presentations of the algebra $\mathcal{W}_q(\mathfrak{sl}_{nd})$ is terms of generators~$T_k[r]$ are given by the same formulas; the only difference is the region of $r$. Heuristically, one can think that $\mathcal{W}_q(\mathfrak{sl}_{nd},n'd)$ is the same algebra as $\mathcal{W}_q(\mathfrak{sl}_{nd})$ but with currents having nontrivial monodromy around zero.

In order to explain these results in more details, consider an example of $\mathfrak{sl}_2$.
\begin{Example} \label{examp: Int Fock odinary}
As a warm-up, consider the untwisted case $n'=0$. The algebra $\mathcal{W}_q(\mathfrak{sl}_2)$ is $q$-deformed Virasoro algebra \cite{SKAO}. It has one generating current $T(z)=T_1(z)$ and the relation reads
\begin{gather}\label{eq:Int TT}
\sum_{l=0}^{\infty} f[l] \big( T[r{-}l] T[s{+}l] - T[s{-}l] T[r{+}l] \big) = -2r\big(q^{\half}{-}q^{-\half}\big)^2 \delta_{r+s,0} ,
\end{gather}
where $f[l]$ are coefficients of a series $\sum\limits_{l=0}^{\infty} f[l] x^l=\sqrt{(1-qx)\big(1-q^{-1}x\big)}/(1-x)$. This algebra has a standard bosonization \cite{SKAO}
\begin{gather}
T (z) = - \big(q^{\half}-q^{-\half}\big) z \nonumber\\
\hphantom{T (z) =}{}\times \left[
u : \! \exp \big( \eta(q^{1/2} z\big) - \eta\big(q^{-1/2} z\big)\big) \! : +u^{-1} : \! \exp \big( \eta\big(q^{-1/2} z\big) - \eta\big(q^{1/2} z\big) \big) \! : \right],\label{eq:Int T st}
\end{gather}
where $\eta(z) = \sum\limits_{k \neq 0} \eta[k] z^{-k}/k$ and $\eta[k]$ are the generators of the Heisenberg algebra $[\eta[k_1], \eta[k_2]] = \frac12 k_1 \delta_{k_1+k_2,0}$; one can also add $\eta[0]$ related to the parameter $u$. In terms of the toroidal algebra $\qD$ this formula corresponds to the tensor product of two Fock modules $\mathcal{F}_{u_1}\otimes \mathcal{F}_{u_2}$, here $u^2=u_1/u_2$.
\end{Example}

\begin{Example}Now, consider the twisted case $n'=1$. The algebra $\mathcal{W}_q(\mathfrak{sl}_2,1)$ is generated by one current $T^{tw}(z)=T_1^{tw}(z)=\sum\limits_{r \in \mathbb{Z}+1/2} T_1^{tw}[r] z^{-r}$. The generators $T^{tw}[r]=T_1^{tw}[r]$ satisfy relation~\eqref{eq:Int TT}. The algebra $\mathcal{W}_q(\mathfrak{sl}_2,1)$ is called \emph{twisted $q$-deformed Virasoro algebra}.

As was explained above, the representations of $\mathcal{W}_q(\mathfrak{sl}_2,1)$ come from the twisted Fock modu\-les~$\mathcal{F}_u^{(1,2)}.\!$ The bosonization of the twisted Fock module leads to the bosonization of the~$\mathcal{W}_q(\mathfrak{sl}_2,1)$. Using formula~\eqref{eq:Etwboson} we get a bosonization
\begin{gather*}
T^{tw} (z) = \big(q^{\half}-q^{-\half}\big) \\
\hphantom{T^{tw} (z) =}{} \times \left[ z^{1/2} :\!\exp \big( \eta\big(q^{1/2} z\big) + \eta\big(q^{-1/2} z\big) \big)\!: + z^{3/2}:\! \exp \big( {-}\eta\big(q^{1/2} z\big) - \eta\big(q^{-1/2} z\big) \big)\!: \right].
\end{gather*}
Using formula \eqref{eq:Etwbosontw} we get a strange bosonization
\begin{gather*}
T^{tw} (z) = (-1)^{\frac{1}{2}} \frac{q^{\half}-q^{-\half}}{2\big(q^{ \frac{1}{4}} - q^{-\frac{1}{4}}\big)} z^{\frac{1}{2}} \\
\hphantom{T^{tw} (z) =}{}\times \left[ : \! \exp \left( \sum_{2 \nmid r} \frac{q^{-\frac{r}{4}}-q^{\frac{r}{4}}}{r} J_r z^{-\frac{r}{2}} \right) - : \! \exp \left( \sum_{2 \nmid r} \frac{q^{\frac{r}{4}}-q^{-\frac{r}{4}}}{r} J_r z^{-\frac{r}{2}} \right) \! \!:\right].
\end{gather*}
Here $\eta(z) = \sum\limits_{k \neq 0} \eta[k] z^{-k}/k +Q + \eta[0] \log z$, and $J_r$ are modes of the odd Heisenberg algebra, $[J_r,J_s]=r\delta_{r+s,0}$. These formulas for bosonization are probably new.
\end{Example}

\begin{Example}
One can also use embedding $\qd{1/n}^{\Lambda} \subset \qd{1/n}$ in order to construct a bosonization of the $W$-algebras. Namely one can take a representation of $\qd{1/n}$ with known bosonization and then express the $W$-algebra related to $\qD = \qd{1/n}^{\Lambda}$ in terms of these bosons.

For example, consider $\Lambda$ generated by $v_1=e_1$, $v_2=2e_2$ and the Fock representations $\mathcal{F}_{u^{1/2}}$ of $\qd{1/2}$. One can show (for example, using \eqref{eq:Fu decomp intro}) that $\mathcal{W}_q(\mathfrak{gl}_\infty)$ algebra related to $\qD \cong \qd{1/2}^\Lambda$ acts on $\mathcal{F}_{u^{1/2}}$ through the quotient $\mathcal{W}_q(\mathfrak{gl}_2)$. Therefore, we get an odd bosonization of non-twisted $q$-deformed Virasoro algebra $\mathcal{W}_q(\mathfrak{sl}_2)$
\begin{gather}\label{eq:Int T odd}
T (z) = \frac{q^{\frac{1}{4}}+q^{-\frac{1}{4}}}{2} \left[ : \! \exp \left( \sum_{2 \nmid r} \frac{q^{-\frac{r}{4}}-q^{\frac{r}{4}}}{r} J_r z^{-\frac{r}{2}} \right) \! : + : \! \exp \left( \sum_{2 \nmid r} \frac{q^{\frac{r}{4}}-q^{-\frac{r}{4}}}{r} J_r z^{-\frac{r}{2}} \right) \! : \right].\!\!\!
\end{gather}
Here $J_r$ are the odd modes of the initial boson for $\mathcal{F}_u$. The even modes of the boson disappear in the formula since it belongs to $\mathfrak{Heis}\subset \qd{1/2}^\Lambda$.

It follows from the decomposition \eqref{eq:Fu decomp intro} that formula \eqref{eq:Int T odd} gives bosonization of certain special representation $\mathcal{W}_q(\mathfrak{sl}_2)$, to be more specific, a direct sum of Fock modules (defined by~\eqref{eq:Int T st}) with particular parameters $u=q^{l-1/4}$ for $l \in \mathbb{Z}$.

In the conformal limit $q\rightarrow 1$ formula \eqref{eq:Int T odd} goes to the odd bosonization of the Virasoro algebra $L_k = \frac{1}{4} \sum\limits_{\frac12 (r+s)=k} : \! J_r J_s \! : + \frac{1}{16} \delta_{k,0}$, see, e.g., \cite{Z}.
\end{Example}

{\bf Whittaker vectors and relations on conformal blocks.}
As an application, in Section~\ref{section: conformal} we prove the following identity
\begin{gather}
z^{\frac12 \sum \frac{i^2}{n^2}} \prod_{i \neq j} \frac{1}{\big(q^{1+\frac{i-j}{n}}; q, q\big)_{\infty}} \big( q^{\frac{1}{n}} z^{\frac{1}{n}}; q^{\frac{1}{n}} ,q^{\frac{1}{n}} \big)_{\infty} \nonumber\\
\qquad{} = \sum_{(l_0, \dots, l_{n-1}) \in \Q} \mathcal{Z} \big( q^{l_0}, q^{\frac{1}{n}+l_1}, \dots, q^{\frac{n-1}{n} + l_{n-1}}; z \big).\label{eq:Intr relation}
\end{gather}
Here the lattice $\Q$ is as above, $(u; q, q)_{\infty} = \prod\limits_{i,j=0}^{\infty} \big(1-q^{i+j} u\big)$. The function $\mathcal{Z}(u_1,\dots,u_n;z)$ is a Whittaker limit of conformal block. By AGT relation it equals to the Nekrasov partition function. We recall the definition of $\mathcal{Z}(u_1,\dots,u_n;z)$ below.

The relation \eqref{eq:Intr relation} was conjectured in \cite{BGM} in the framework of $q$-isomonodromy/CFT correspondence. As we discussed in the first part of the introduction the main statement of this correspondence is an explicit formula for the $q$-isomonodromic tau function as an infinite sum of conformal blocks. The left-hand side of~\eqref{eq:Intr relation} is a tau function corresponding to the algebraic solution of deautonomized discrete flow in Toda system, see~\cite[equation~(3.11)]{BGM}.
The right-hand side of~\eqref{eq:Intr relation} is a specialization of conjectural formula \cite[equation~(3.6)]{BGM} for the generic tau function of these flows. In differential case the isomonodromy/CFT correspondanse is proven in many cases, see \cite{BS:2014,GIL:2018,GL:2016, ILT:2015}, but in the $q$-difference case the main statements are still conjectures. The generic formula for tau function of deautonomized discrete flow in Toda system is proven only for particular case $n=2$ \cite{BS:2018,MN:2018}. Here we prove formula for arbitrary $n$ but for special solution.

Let us recall the definition of $\mathcal{Z}(u_1,\dots,u_n;z)$. The Whittaker vector $W(z|u_1,\dots,u_N)$ is a~vector in a~completion of $ \mathcal{F}_{u_1}\otimes\dots \otimes\mathcal{F}_{u_n}$, which is an eigenvector of $E_{a,b}$ for $Nb\geqslant a \geqslant 0$ with certain eigenvalues depending on $z$, see Definition \ref{def: Whit eigen}. Such vector exists and unique for generic values of $u_1,\dots,u_n$. This property looks to be a part of folklore, we give a proof of this in Appendix \ref{Appendix:Whittaker}. The proof is essentially based on the results of \cite{N17, N16}. The function $\mathcal{Z}$ is proportional to a~Shapovalov pairing of two Whittaker vectors
\begin{gather*}
\mathcal{Z} (u_1, \dots, u_n ; z)
 = z^{\frac{\sum (\log u_i)^2}{2 (\log q)^2}} \!\prod_{i \neq j}\! \frac{1}{\big(qu_i u_j^{-1}; q, q\big)_{\infty}} \left\langle W_u\big(1|qu_n^{-1}, \dots, qu_1^{-1}\big) , W(z| u_1, \dots, u_n) \right\rangle.%\label{eq:Int def q-der conf block}
\end{gather*}

We give a proof of \eqref{eq:Intr relation} using decomposition \eqref{eq:Fu decomp intro}. We consider the Whittaker vector $W(z|1)$ for the algebra $\qd{1/n}$. Its Shapovalov pairing gives the left-hand side of the relation~\eqref{eq:Intr relation}. On the other hand, we prove that its restriction to summands $\mathcal{F}_{q^{ l_0}} \otimes \cdots \otimes \mathcal{F}_{ q^{\frac{n-1}{n}+ l_{n-1} }}$ is the Whittaker vector for the algebra $\qd{}$. So taking the Shapovalov pairing we get the right-hand side of the relation~\eqref{eq:Intr relation}.

In the conformal limit $q\rightarrow 1$ the analogue of the relation \eqref{eq:Intr relation} in case $n=2$ was proven in~\cite{BS:2016:2} by a similar method. The conformal limit of the decomposition~\eqref{eq:Fu decomp intro} was studied in~\cite{BGMtw}.

{\bf Discussion of $\boldsymbol{q\neq t}$ case.} As we mentioned above, $\qD$ is a specialization of quantum toroidal algebra $U_{q,t}\big(\ddot{\mathfrak{gl}}_1\big)$ for $q=t$. It is much more interesting to study the algebra without the constrain. Let us discuss our expectations on generalizations of the results from this paper.

It is likely that fermionic construction (see Theorem~\ref{TheFermionicTh}) will be generalized after the replacement of the fermions by vertex operators of quantum affine $\mathfrak{gl}_n$. Hence we have bosonization, expressing the currents in terms of exponents dressed by screenings. We also expect that representations of twisted and non-twisted $W_n$-algebras can be realized via these vertex operators (see \cite{BG} for the $n=2$ case). It is not clear how one can generalize strange bosonization and connection with isomonodromy/CFT correspondence for $q \neq t$.

{\bf Plan of the paper.} The paper is organized as follows.

In Section \ref{q-dif} we recall basic definitions and properties on the algebra $\qD$.

In Section \ref{level1} we recall basic constructions of the Fock module $\mathcal{F}_u$.

In Section \ref{Section with answers} we present three constructions of the twisted Fock module $\mathcal{F}_u^{\sigma}$: the fermionic construction in Theorem \ref{TheFermionicTh}, the bosonic construction in Theorem \ref{Th:qD Boson}, and the strange bosonic construction in Theorem \ref{Th:StrangeBoson}.

In Section \ref{S:sublattice} we study restriction of the Fock module to a subalgebra $\qD^{\Lambda}$. Using these restrictions we prove Theorems \ref{TheFermionicTh}, \ref{Th:qD Boson} and \ref{Th:StrangeBoson}.

In Section \ref{S:Semi-infinite} we give an independent proof of Theorem \ref{TheFermionicTh} using the semi-infinite construction.

In Section \ref{section: q-W} we study twisted $q$-deformed $W$-algebras. We define $\mathcal{W}_q(\mathfrak{sl}_n,n_{tw})$ by generators and relations. Then we show in Theorem \ref{Th:Wiso} that the tensor product $\mathcal{W}_q(\mathfrak{sl}_n,n_{tw})\UH$ is isomorphic to the certain quotient of $U(\qD)$; we denote this quotient by $\mathcal{W}_q(\mathfrak{gl}_n,n_{tw})$. We show that $\mathcal{W}_q(\mathfrak{sl}_{nd},n'd)$ acts on the tensor product of twisted Fock modules ${\mathcal{F}^{\sigma}_{u_1} \otimes \cdots \otimes \mathcal{F}_{u_d}^{\sigma}}$. At the end of the section we study relation between these modules and the Verma modules for $\mathcal{W}_q(\mathfrak{gl}_{nd},n'd)$ and $\mathcal{W}_q(\mathfrak{sl}_{nd},n'd)$.

In Section \ref{Section: general sublattice} we prove decomposition \eqref{eq:Fu decomp intro}. Then we study the strange bosonization of $W$-algebra modules arising from the restriction of Fock module on $\qD^{\Lambda}$.

In Section \ref{section: conformal} we recall definitions and properties of Whittaker vector, Shapovalov pairing, and conformal blocks. Then we prove \eqref{eq:Intr relation}, see Theorem \ref{Theorem: the relation}.

In Appendix \ref{nonorm} we give a definition and study necessary properties of regular product of currents $A(z) B(az)$ for $a\in \mathbb{C}$.

Appendices \ref{Appendix:Serre} and \ref{Appendix:HomWD} consist of calculations which are used in Section \ref{section: q-W}.

In Appendix \ref{Appendix:Whittaker} we study the Whittaker vector for $\qD$ in the completion of the tensor product $ \mathcal{F}_{u_1}\otimes\dots \otimes\mathcal{F}_{u_n}$. We prove its existence and uniqueness (we use this in Section \ref{section: conformal}). To prove existence we present a construction of Whittaker vector via an intertwiner operator from~\cite{AFS}. We also relate this Whittaker vector to the Whittaker vector of $\mathcal{W}_q(\mathfrak{sl}_n)$ introduced in~\cite{T}.

\section[$q$-difference operators]{$\boldsymbol{q}$-difference operators} \label{q-dif}

In this section we introduce notation and recall basic facts about algebra $\qD$, see \cite{FFZ, GL,KR93}.
\begin{Definition}
The associative algebra of $q$-difference operators $\qaD$ is an associative algebra generated by $D^{\pm 1}$ and $x^{\pm 1}$ with the relation $Dx=qxD$.
\end{Definition}

\begin{Definition}
The algebra of $q$-difference operators $\qD$ is a Lie algebra with a basis $E_{k,l}$ {\upshape(}where $(k, l) \in \mathbb{Z}^2 \backslash \{(0,0) \}${\upshape)}, $c$ and $c'$. The elements $c$ and $c'$ are central. All other commutators are given by
\begin{gather}
[ E_{k,l} , E_{r,s} ] = \big( q^{(sk-lr)/2}-q^{(lr-sk)/2} \big) E_{k+r, l+s} + \delta_{k, -r} \, \delta_{l, -s} ( c' k + c l ). \label{qDc}
\end{gather}
\end{Definition}

\begin{Remark}Note that the vector subspace of $\qaD$ spanned by $x^l D^k$ (for $(l,k) \neq (0,0)$) is closed under commutation, i.e., has a natural structure of Lie algebra (denote this Lie algebra by $\mathrm{Diff}_q^{\mathrm{L}}$). Consider a basis of this Lie algebra $E_{k,l} := q^{kl/2} x^l D^k$. Finally, $\qD$ is a central extension of $\mathrm{Diff}_q^{\mathrm{L}}$ by two-dimensional abelian Lie algebra spanned by $c$ and $c'$.
\end{Remark}

\subsection[${\rm SL}_2(\mathbb{Z} )$ action]{$\boldsymbol{{\rm SL}_2(\mathbb{Z} )}$ action}

In this section we will define action ${\rm SL}_2 ( \mathbb{Z} )$ on $\qD$. Let $\sigma$ be an element of ${\rm SL}_2 ( \mathbb{Z} )$ corresponding to a matrix
\begin{gather*}
\sigma = \begin{pmatrix}
m' & m \\
n' & n
\end{pmatrix}. %\label{notation:sigma}
\end{gather*}
Then $\sigma$ acts as follows
\begin{gather} \label{action1}
\sigma( E_{k,l} ) = E_{ m' k + m l, n' k + n l }, \qquad
\sigma (c') = m' c' + n' c, \qquad \sigma (c) = m c' + n c.
\end{gather}

\begin{Proposition} Formula \eqref{action1} defines ${\rm SL}_2 ( \mathbb{Z} )$ action on $\qD$ by Lie algebra automorphisms.
\end{Proposition}
\begin{proof} Note that \eqref{qDc} is ${\rm SL}_2 ( \mathbb{Z} )$ covariant. \end{proof}

For any $\qD$-module $M$ denote by $\rho_{M} \colon \qD \rightarrow \mathfrak{gl} (M)$ the corresponding homomorphism.

\begin{Definition} \label{twisted} For any $\qD$-module $M$ and $\sigma \in {\rm SL}(2,\mathbb{Z})$ let us define the representation $M^\sigma$ as follows. $M$ and $M^{\sigma}$ are the same vector space with different actions, namely $\rho_{M^{\sigma}} = \rho_M \circ \sigma$.
\end{Definition}

We will refer to $M^{\sigma}$ as a \emph{twisted representation}. More precisely, $M^{\sigma}$ is the representation $M$, twisted by~$\sigma$.

\subsection{Chevalley generators and relations}

The Lie algebra $\qD$ is generated by $E_k := E_{1, k}$, $F_k:=E_{-1, k}$ and $H_k := E_{0, k}$. We will call them the \emph{Chevalley generators} of $\qD$.
Define the following \emph{currents} (i.e., formal power series with coefficients in $\qD$)
\begin{gather*}
E(z) = \sum_{k \in \mathbb{Z}} E_{1,k} z^{-k} = \sum_{k \in \mathbb{Z}} E_k z^{-k},\\
F(z) = \sum_{k \in \mathbb{Z}} E_{-1, k} z^{-k} = \sum_{n \in \mathbb{Z}} F_k z^{-k},\\
H(z) = \sum_{k \neq 0} E_{0, k} z^{-k} = \sum_{k \neq 0} H_k z^{-k}.
\end{gather*}

Let us also define the formal delta function
\begin{gather*}
\delta(x) = \sum_{k \in \mathbb{Z}} x^k.
\end{gather*}
\begin{Proposition}\label{relation} Lie algebra $\qD$ is presented by the generators $E_k$, $F_k$ $($for all $k \in \mathbb{Z})$, $H_l$ {\upshape(}for $l \in \mathbb{Z} \backslash \{0 \}${\upshape)}, $c$, $c'$ and the following relations
\begin{gather}
[H_k, H_l] = k c \delta_{k+l, 0}, \label{RqDHH} \\
[H_k , E(z)] = \big(q^{-k/2}-q^{k/2}\big)z^k E(z), \qquad
[H_k, F(z)]= \big(q^{k/2}-q^{-k/2}\big) z^k F(z), \label{RqDHE-F} \\
(z-qw)\big(z-q^{-1} w\big) [E(z), E(w)]= 0, \qquad
(z-qw)\big(z-q^{-1} w\big) [F(z), F(w)]= 0 , \label{RqDE2} \\
[ E (z) , F (w)] = \big( H\big( q^{-1/2} w \big) - H\big(q^{1/2} w\big) + c'\big) \delta( w / z ) + c \frac{w}{z} \delta'(w/ z),\label{RqDEF} \\
z_2 z_3^{-1} [E(z_1), [E(z_2), E(z_3)]] + \mathrm{cyclic} =0, \label{RqDEEE} \\
z_2 z_3^{-1} [F(z_1), [F(z_2), F(z_3)]] + \mathrm{cyclic} =0. \label{RqDFFF}
\end{gather}
\end{Proposition}
One can find a proof of Proposition \ref{relation} in \cite[Theorem~2.1]{M07} or \cite[Theorem~5.5]{Ts}.

\section{Fock module} \label{level1}
In this section we review basic constructions of representations of $\qD$ with $c=1$ and $c'=0$. These construction were studied in~\cite{GL}.

\subsection{Free boson realization} \label{boson}

Introduce the Heisenberg algebra generated by $a_k$ (for $k \in \mathbb{Z}$) with relation $[a_k, a_l] = k \delta_{k+l, 0}$. Consider the Fock module $F_{\alpha}^a$ generated by $| \alpha \rangle$ such that $a_k | \alpha \rangle = 0$ for $k > 0$, $a_0 | \alpha \rangle = \alpha | \alpha \rangle$.
\begin{Proposition} \label{prop:sec3 Boson Fock}
The following formulas
determine an action of $\qD$ on $F_{\alpha}^a$:
\begin{gather}
c \mapsto 1, \qquad c' \mapsto 0, \qquad H_k \mapsto a_k, \label{boson1}\\
E(z) \mapsto \frac{u}{1-q} \exp \left( \sum_{k>0} \frac{q^{-k/2}-q^{k/2}}{k} a_{-k} z^k \right) \exp \left( \sum_{k<0} \frac{q^{-k/2}-q^{k/2}}{k} a_{-k} z^k \right) \label{eq:E},\\
F(z) \mapsto \frac{u^{-1}}{1-q^{-1}} \exp \left( \sum_{k>0} \frac{q^{k/2}-q^{-k/2}}{k} a_{-k} z^k \right) \exp \left( \sum_{k<0}\frac{q^{k/2}-q^{-k/2}}{k} a_{-k} z^k \right). \label{eq:F}
\end{gather}
\end{Proposition}

We will denote this representation by $\mathcal{F}_u$.

\begin{Remark} \label{remark:NOalpha}Note that $\alpha$ does not appear in formulas \eqref{boson1}--\eqref{eq:F}. But we will need operator $a_0$ later (see the proof of Proposition \ref{Boson-Fermion}) for the boson-fermion correspondence. Heuristically, one can think that $u = q^{-\alpha}$.
\end{Remark}

\begin{Remark}[on our notation]In this paper, we consider several algebras and their action on the corresponding Fock modules. We choose the following notation. All these representations are denoted by the letter F (for Fock) with some superscript to mention an algebra. Since $\qD$ is the most important algebra in our paper, we use no superscript for its representation.
Also, let us remark that we consider several copies of the Heisenberg algebra. To distinguish their Fock modules, we write a letter for generators as a superscript.
\end{Remark}

The standard bilinear form on $F^a_{\alpha}$ is defined by the following conditions: operator $a_{-k}$ is dual of $a_k$, the pairing of $\valpha$ with itself equals 1. We will use the bra-ket notation for this scalar product. For an operator $A$ we denote by $\langle \alpha | A | \alpha \rangle$ the scalar product of $A | \alpha \rangle$ with $\valpha$.

\begin{Proposition}\label{Fock} Suppose the algebra $\qD$ acts on $F^a_{\alpha}$ so that $H_k \mapsto a_k$ and $\langle \alpha | E(z) | \alpha \rangle = \frac{u}{1-q}$; $\langle \alpha | F(z) | \alpha \rangle = \frac{u^{-1}}{1-q^{-1}}$.
Then this representation is isomorphic to~$\mathcal{F}_u$.
\end{Proposition}

\begin{proof} Consider the current \[T(z) =\exp \! \left( - \sum_{k>0} \frac{q^{-k/2}-q^{k/2}}{k} a_{-k} z^k \right) E(z) \ \exp \! \left( - \sum_{k<0}\frac{q^{-k/2}-q^{k/2}}{k} a_{-k} z^k \right).\]

It is easy to verify that $[a_k, T(z)] =0$. Since $F_{\alpha}$ is irreducible, $T(z)= f(z)$ for some formal power series $f(z)$ with $\mathbb{C}$-coefficients. On the other hand, $f(z) = \langle \alpha | E(z) | \alpha \rangle = \frac{u}{1-q} $. This implies~\eqref{eq:E}. The proof of~\eqref{eq:F} is analogous.
\end{proof}

\begin{Proposition} \label{E_l}
Denote $E_l(z) = E_{l, k} z^{-k}$. The action of $E_l(z)$ on Fock representation $\mathcal{F}_u$ is given by the following formula
\begin{gather} \label{eq: El in Fock mod}
E_l(z) \rightarrow \frac{u^l}{1-q^l} \exp \! \left( \sum_{k>0} \frac{q^{-kl/2}-q^{kl/2}}{k} a_{-k} z^k \right) \exp \! \left( \sum_{k<0} \frac{q^{-kl/2}-q^{kl/2}}{k} a_{-k} z^k \right).
\end{gather}
\begin{proof}
The commutation relation \eqref{qDc} implies that formula \eqref{eq: El in Fock mod} holds up to a pre-exponential factor. Also, we see from \eqref{qDc} that
\begin{gather} \label{eq: OPE proof for El}
E(z) E_l(w) = \frac{q^{-1}w}{z-q^{-1}w} E_{l+1}\big(q^{-1}w\big) - \frac{q^l w }{z-q^l w} E_{l+1}(w) + \mathrm{reg}.
\end{gather}
The factor can be found inductively from \eqref{eq: OPE proof for El}.
\end{proof}
\end{Proposition}

\subsection{Free fermion realization} \label{fermi}

In this section we give another construction for the Fock representation of $\qD$. To do this, let us consider the Clifford algebra, generated by $\psi_i$ and $\psi^*_j$ for $i,j \in \mathbb{Z}$ subject to the relations
\begin{gather*}
\{ \psi_i, \psi_j \} =0, \qquad \{ \psi^*_i, \psi^*_j \} =0, %\label{fermi1} \\
\qquad \{ \psi_i , \psi^*_j \} = \delta_{i+j, 0}. % \label{fermi2}
\end{gather*}
Consider the currents
\begin{gather*}
\psi(z) = \sum_{i} \psi_i z^{-i-1}, \qquad \psi^*(z) = \sum_i \psi^*_i z^{-i}.
\end{gather*}
Consider a module $F^\psi$ with a cyclic vector $\n$ and relation
\begin{gather*}
\psi_i \n = 0 \quad \text{for} \ i \geqslant l, \qquad \psi^*_j \n = 0 \quad \text{for} \ j > -l.
\end{gather*}
The module $F^\psi$ is independent of $l$. The isomorphism can be seen from the formulas $\psi^*_{-l} \n = | l+1 \rangle$ and $\psi_{l-1} \n = | l-1 \rangle$.
Let us define the $l$-dependent normal ordered product (to be compatible with $\n$) by the following formulas
\begin{gather}
: \! \psi_i \psi^*_j \! :_{(l)} = - \psi^*_j \psi_i \qquad \text{for } i \geqslant l, \label{norm1} \\
: \! \psi_i \psi^*_j \! :_{(l)} = \psi_i \psi^*_j \qquad \text{for } i < l. \label{norm2}
\end{gather}

\begin{Proposition} \label{prop: level 1 Fermion}
The following formulas determine an action of $\qD$ on $F^{\psi}$:
\begin{gather}
c \mapsto 1, \quad c' \mapsto 0, \quad H_k \mapsto \sum_{i+j=k} \psi_i \psi_j^*, \label{eq:centerFermi} \\
E(z) \mapsto \frac{q^{l} u }{1-q} + u q^{ - 1/2 }z : \! \psi\big( q^{-1/2} z\big) \psi^* \big( q^{1/2} z\big) \! :_{(l)} = u q^{-1/2} z \psi\big( q^{-1/2} z\big) \psi^* \big( q^{1/2} z\big), \label{ferm:E} \\
F(z) \mapsto \frac{q^{-l} u^{-1} }{1-q^{-1}} + u^{-1} q^{1/2} z : \! \psi\big( q^{1/2} z\big) \psi^* \big(q^{-1/2} z\big)\!:_{(l)}= u^{-1} q^{1/2} z \psi\big( q^{1/2} z\big) \psi^* \big(q^{-1/2} z\big). \!\!\!\label{ferm:F}
\end{gather}
\end{Proposition}

Let us denote this representation by $\mathcal{M}_u$.

\begin{Remark}The Products $\psi\big( q^{-1/2} z\big) \psi^* \big( q^{1/2} z\big)$ and $\psi\big( q^{1/2} z\big) \psi^* \big(q^{-1/2} z\big)$ from formulas \eqref{ferm:E}--\eqref{ferm:F} are not normally ordered (see Appendix \ref{nonorm} for a formal definition and some other technical details on the \emph{regular product}). In particular, this reformulation implies that $\mathcal{M}_u$ does not depend on~$l$.
\end{Remark}

\subsection{Semi-infinite construction}

\begin{Definition} \label{eq: sec3 eval def}The evaluation representation $V_{u}$ of the algebra $\qD$ is a vector space with the basis $x^{k}$ for $k \in \mathbb{Z}$ and the action
\begin{gather*}
E_{a,b} x^{k} = u^a q^{\frac{ab}{2} + ak} x^{k+b}, \qquad c=c'=0.
\end{gather*}
\end{Definition}

\begin{Remark}The associative algebra $\qaD$ acts on $V_u$. The representation of $\qD$ is obtained via evaluation homomorphism $\mathrm{ev}\colon \qD \rightarrow \qaD$.
\end{Remark}

\begin{Remark}Informally, one can consider $x^k \in V_u$ as $x^{k-\alpha}$ for $u=q^{-\alpha}$. Define the action of $\qaD$ as follows. The generator $x$ acts by multiplication and $D x^{k-\alpha} = q^{k-\alpha} x^{k-\alpha} = u q^k x^{k-\alpha}$. However, $q^{-\alpha}$ is not well defined for arbitrary complex $\alpha$. So we consider $u$ as a parameter of representation instead of $\alpha$.
\end{Remark}

Let us consider the \emph{semi-infinite exterior power} of the evaluation representation $\siL V_{u}$. It is spanned by $| \lambda,l \rangle = x^{l-\lambda_1} \wedge x^{l+1-\lambda_2} \wedge \cdots \wedge x^{l+N} \wedge x^{l+N+1} \wedge x^{l+N+2} \wedge \cdots$, where $\lambda$ is a Young diagram and $l \in \mathbb{Z}$. Let $p_1 > \dots> p_i$ and $q_1> \dots> q_i$ be Frobenius coordinates of $\lambda$.

\begin{Proposition} \label{Prop:SemiInf-Fermi}
There is a $\qD$-modules isomorphism $\siL V_{u} \xrightarrow{
\,\smash{\raisebox{-0.65ex}{\ensuremath{\scriptstyle\sim}}}\,} \mathcal{M}_{u}$ given by
\begin{gather}\label{eq:lambda l}
| \lambda,l \rangle \mapsto (-1)^{\sum_k (q_k-1)} \psi_{-p_1+l} \cdots \psi_{-p_i+l} \, \psi_{-q_i-l+1}^* \cdots \psi_{-q_1-l+1}^* \n.
\end{gather}
\end{Proposition}

\begin{Proposition} \label{Boson-Fermion}
There is an isomorphism of $\qD$-modules $\mathcal{M}_u \cong \bigoplus_{l \in \mathbb{Z}} \mathcal{F}_{q^l u}$.
The sub\-module~$\mathcal{F}_{q^l u}$ is spanned by $| \lambda,l \rangle$.
\end{Proposition}

\begin{proof}Recall the ordinary boson-fermion correspondence (see \cite{KR}). The coefficients of
\begin{gather*}
a(z) = \sum_{n} a_n z^{-n-1} = : \! \psi(z) \psi^* (z) \! :_{(0)}
\end{gather*}
are indeed generators of the Heisenberg algebra. Moreover, $F^{\psi} = \oplus_{l \in \mathbb{Z}} F^{a}_{-l}$. The highest vector of $F^{a}_{-l}$ is $\n$ (in particular, $a_0 \n = -l \n $). Note that this is the decomposition of $\qD$-modules as well. Also, note that
\[\langle l | E(z) | l \rangle = \frac{q^{l} u }{1-q}, \qquad \langle l | F(z) | l \rangle = \frac{q^{-l} u^{-1} }{1-q^{-1}}.\]
Therefore one can use Proposition~\ref{Fock} for each summand $F^{a}_{-l}$.
\end{proof}

There is a basis in the Fock module $\mathcal{F}_u$ given by semi-infinite monomials
\begin{gather*}
|\lambda\rangle= x^{- \lambda_1} \wedge x^{1 - \lambda_2} \wedge \cdots \wedge x^{i - \lambda_{i+1}} \wedge \cdots.
\end{gather*}

To write the action of $\qD$ in this basis, let us remind the standard notation. Let $l(\lambda)$ be the number of non-zero rows. We will write $s=(i,j)$ for the $j$th box in the $i$th row (i.e., $j \leqslant \lambda_i$). The content of a box
$c(s):=i-j$. For the diagram $\mu \subset \lambda$, we define a skew Young diagram~$\lambda \backslash \mu$, being a set of boxes in $\lambda$ which are not in~$\mu$. Ribbon is a skew Young diagram without $2\times2$ squares. The height $\operatorname{ht}(\lambda \backslash \mu)$ of a ribbon is one less than the number of its rows.
\begin{Proposition}
The action of $\qD$ on $\mathcal{F}_u$ is given by the following formulas
\begin{gather}
E_{a,-b} \, |\lambda\rangle= q^{-\frac{a}{2}}u^{a} \sum_{\mu \backslash \lambda =b-\text{\rm ribbon}} (-1)^{\operatorname{ht}(\mu \backslash \lambda)}q^{\frac{a}{b}\sum\limits_{s\in \mu \backslash \lambda}c(s)}|\mu\rangle, \label{eq: Fock basis action1} \\
E_{a,b} \, |\lambda\rangle= q^{-\frac{a}{2}} u^{a} \sum_{\lambda \backslash \mu =b-\text{\rm ribbon}} (-1)^{\operatorname{ht}(\lambda \backslash \mu)}q^{\frac{a}{b}\sum\limits_{s\in \mu \backslash \lambda}c(s)}|\mu\rangle,\\
E_{a,0} \, |\lambda\rangle= u^a \left( \frac{1}{1-q^a} + \sum_{i=0}^{l(\lambda)-1} \big(q^{a(i-\lambda_{i+1})} - q^{ai} \big) \right) |\lambda\rangle, \label{eq: Fock basis action3}
\end{gather}
here $b>0$.
\end{Proposition}
In particular,
\begin{gather}
E_{a,0} \, \vac =\frac{u^{a}}{1-q^{a}}\, \vac. \label{eq:eigen vac}
\end{gather}

Let us introduce the notation $c(\lambda) = \sum\limits_{s \in \lambda} c(s)$. Define an operator $\I_{\tau} \in \End( \mathcal{F}_u)$ by the following formula
\begin{gather} \label{eq: def I tau}
\I_{\tau}\, |\lambda \rangle = u^{|\lambda|} q^{-\frac12 |\lambda|+c(\lambda)}\, | \lambda \rangle.
\end{gather}
The operator was introduced in \cite{BGHT} and is well known nowadays.
\begin{Proposition} \label{prop: I sigma}
The operator $\I_{\tau}$ enjoys the property $\I_{\tau} E_{a,b} \I_{\tau}^{-1} = E_{a-b,b}$.
\end{Proposition}

\begin{proof}Follows from \eqref{eq: Fock basis action1}--\eqref{eq: Fock basis action3}.
\end{proof}

\begin{Corollary} \label{corol: unipotant twist}
$\mathcal{F}_u^{\tau} \cong \mathcal{F}_u$ for $\tau = \left( \begin{smallmatrix} 1 &1 \\ 0 & 1\end{smallmatrix} \right)$.
\end{Corollary}
\begin{Remark}
Also, Corollary \ref{corol: unipotant twist} follows from Proposition \ref{Fock}: we will use this approach to prove Proposition \ref{LambdaFock}.
\end{Remark}

\begin{Corollary} \label{corol: not m and m prime}
The twisted representation $\mathcal{F}_u^{\sigma}$ is determined up to isomorphism by $n$ and $n'$.
\end{Corollary}
\begin{proof}
Corollary \ref{corol: unipotant twist} implies that $\mathcal{F}_u^{\, \tau^k \sigma} \cong \mathcal{F}_u^{\sigma}$. Note that
\begin{gather*} \tau^k \sigma =
\begin{pmatrix}
1 & k \\
0 & 1
\end{pmatrix}
\begin{pmatrix}
m' & m \\
n' & n
\end{pmatrix}=
\begin{pmatrix}
m'+kn' & m+kn \\
n' & n
\end{pmatrix}.
\end{gather*}
For the fixed $n$ and $n'$, all the possible choices of $m$ and $m'$ appear for the appropriate $k$.
\end{proof}

\section{Explicit formulas for twisted representation} \label{Section with answers}
In this section we provide three explicit constructions of twisted Fock module $\mathcal{F}^\sigma_u$ for
\begin{equation} \label{sigma}
\sigma = \begin{pmatrix}m' & m \\
n' & n
\end{pmatrix}.
\end{equation}
Constructions are called fermionic, bosonic, and strange bosonic. This section contains no proofs. We will give proofs in Sections \ref{S:sublattice}. In Section \ref{S:Semi-infinite} we will provide an independent proof of Theorem \ref{TheFermionicTh}.

\subsection{Fermionic construction} \label{section: The Fermionic Th}

We need to consider the $\mathbb{Z}/2 \mathbb{Z}$-graded $n$th tensor power of the Clifford algebra defined above. More precisely, consider an algebra generated by $\psia [i]$ and $\psib [j]$, for $i,j \in \mathbb{Z}$; $a,b=0, \dots , n-1$, subject to relations
\begin{gather} \label{eq:nClifford1}
\big\{ \psia[i], \psibb[j] \big\} =0, \qquad \big\{ \psiaa [i], \psib[j] \big\} =0, \\ \big\{ \psia [i] , \psib [j] \big\} = \delta_{a,b} \delta_{i+j, 0}. \label{eq:nClifford2}
\end{gather}
Consider the currents
\begin{gather*}
\psia (z) = \sum_{i} \psia [i] z^{-i-1}, \qquad \psib(z) = \sum_i \psib[i] z^{-i}. \label{eq:nFermions}
\end{gather*}
Consider a module $F^{n\psi}$ with a cyclic vector $|l_0, \dots, l_{n-1} \rangle$ and the relations
\begin{gather*}
\psia [i] \, |l_0, \dots, l_{n-1} \rangle = 0 \quad \text{for } i \geqslant l_a, \\ \psiaa [j] \, \left|l_0, \dots, l_{n-1} \right\rangle = 0 \quad \text{for } j > -l_a.
\end{gather*}
The module $F^{n\psi}$ does not depend on $l_0, \dots, l_{n-1}$. The isomorphism can be seen from the following formulas:
\begin{gather*}
\psiaa [ -l_a ] \, |l_0, \dots, l_a, \dots, l_{n-1} \rangle = | l_0, \dots, l_a+1, \dots , l_{n-1} \rangle, \\
\psia [ l_a-1 ] \, |l_0, \dots, l_a, \dots , l_{n-1} \rangle = |l_0, \dots, l_a-1, \dots , l_{n-1} \rangle.
\end{gather*}

\begin{Theorem}\label{TheFermionicTh}
The formulas below determine an action of $\qD$ on $F^{n\psi}$
\begin{gather}
c' = n', \qquad c=n,\nonumber\\
H^{tw}_k = \sum_{a} \sum_{i+j=k} \psi_a [i] \psi^*_a [j] , \nonumber\\
E^{tw}(z) = \sum_{b-a \equiv -n' \bmod n} u^{\frac{1}{n}} q^{-1/2} z \psia \big( q^{-1/2} z\big) \psib \big(q^{1/2} z\big) z^{\frac{n'-a+b}{n}} q^{(a+b)/2n}, \label{eq:EtwFermi} \\
F^{tw}(z) = \sum_{b-a \equiv n' \bmod n} u^{-\frac{1}{n}} q^{1/2} z \psia \big( q^{1/2} z\big) \psib \big(q^{-1/2} z\big) z^{\frac{-n'-a+b}{n}} q^{-(a+b)/2n}. \nonumber %\label{eq:FtwFermi}
\end{gather}
The module obtained is isomorphic to $\mathcal{M}_u^{\sigma}$.
\end{Theorem}
Since $\mathcal{F}_u^{\sigma} \subset \mathcal{M}_u^{\sigma}$, we have obtained a fermionic construction for $\mathcal{F}_u^{\sigma}$.

\subsection{Bosonic construction} \label{subsection: TW qDboson}

Let us consider the $n$th tensor power of the Heisenberg algebra. More precisely, this algebra is generated by $a_b [i]$ for $b=0, \dots, n-1$ and $i \in \mathbb{Z}$ with the relation $[ a_{b_1}[i], a_{b_2}[j]] = i \delta_{b_1,b_2} \delta_{i+j,0}$. Let us extend the algebra by adding the operators $e^{Q_b}$, obeying the following commutation relations. The operator $e^{Q_b}$ commutes with all the generators except for $a_b[0]$ and satisfy $a_b[0] e^{Q_b} = e^{Q_b} (a_b[0]+1)$. Denote
\begin{gather*}
\phi_{b}(z) = \sum_{j \neq 0} \frac{1}{j} a_{b} [j] z^{-j} + Q_b + a_{b}[0] \log z.
\end{gather*}
\begin{Remark}
Informally, one can think that there exists an operator $Q_b$ satisfying $[a_b[0], Q_b] \allowbreak = 1$. However, this operator will not act on our representation. We will use $Q_b$ as a formal symbol. Our final answer will consist only of $e^{Q_b}$, but not of $Q_b$ without the exponent.
\end{Remark}

We need a notion of a normally ordered exponent $:\!\exp(\dots)\!:$. The argument of a normally ordered exponent is a linear combination of $a_b[i]$ and $Q_b$. Let $\mathtt{a}_+$, $\mathtt{a}_-$, $\mathtt{a}_0$, and $\mathtt{Q}$ denote a linear combination of $a_b[i]$ for $i>0$, $a_b[i]$ for $i<0$, $a_b[0]$ and $Q_b$, correspondingly ($b$ is not fixed). Define
\begin{gather*}
: \! \exp ( \mathtt{a}_+ + \mathtt{a}_- + \mathtt{a}_0 + \mathtt{Q} )\! : \, \stackrel{\text{def}}{=} \exp ( \mathtt{a}_+ ) \exp ( \mathtt{a}_- ) \exp ( \mathtt{Q} ) \exp ( \mathtt{a}_0 ).
\end{gather*}

Also, note that $\mathtt{a}_0$ will have the coefficient $\log z$. We shall understand it formally; the action of the operator $\exp(a_b[0] \log z) = z^{a_b[0]}$ is well defined, since in the representation to be considered below, $a_b[0]$ acts as multiplication by an integer at each Fock module.

Let $\Q$ be a lattice with the basis $Q_0{-}Q_1, \dots, Q_{n-2}{-}Q_{n-1}$. Consider the group algebra~$\mathbb{C} \big[ \Q \big]$. This algebra is spanned by $e^{\lambda}$ for $\lambda= \sum_i \lambda_i Q_i \in \Q$. Let us define the action of the commutative algebra generated by $a_b[0]$ on $\mathbb{C} \big[ \Q \big]$:
\begin{gather*}
a_b[0] e^{\sum \lambda_i Q_i} = \lambda_b e^{\sum \lambda_i Q_i}.
\end{gather*}
Let $F^{na}$ be the Fock representation of the algebra generated by $a_b[i]$ for $i\neq 0$; i.e., there is a~cyclic vector $\vac \in F^{na}$ such that $a_b[i] \, \vac = 0$ for $i>0$.

Finally, we can consider $F^{na} \otimes \mathbb{C} \big[ \Q \big]$ as representation of the whole Heisenberg algebra as follows: $a_b[i]$ for $i \neq 0$ acts on the first factor, $a_b[0]$ acts on the second factor. Also, $\mathbb{C} \big[ \Q \big]$ acts on $F^{na} \otimes \mathbb{C} \big[ \Q \big]$.

\begin{Theorem} \label{Th:qD Boson}
There is an action of $\qD$ on $F^{na} \otimes \mathbb{C} \big[ \Q \big]$ determined by the formulas
\begin{gather}
H^{tw}[k] = \sum_{b} a_b [k], \qquad c' = n', \qquad c=n,\nonumber\\
E^{tw}(z) = \sum_{b-a \equiv -n' \bmod n} \!\!\!\!\! u^{\frac{1}{n}} q^{\frac{a+b-n}{2n}} z^{\frac{n'-a+b}{n}+1} : \! \exp \big( \phi_b\big(q^{1/2} z\big) - \phi_a\big(q^{-1/2} z\big) \big) \! : \epsilon_{a,b}, \label{eq:th:Etw boson}\\
F^{tw}(z) = \sum_{b-a \equiv n' \bmod n} \!\!\!\!\! u^{-\frac{1}{n}} q^{ \frac{-a-b+n}{2n} } z^{\frac{-n'-a+b}{n}+1} : \! \exp \big( \phi_b\big(q^{-1/2} z\big) - \phi_a\big(q^{1/2} z\big) \big) \! : \epsilon_{a,b}, \label{eq:th:Ftw boson}
\end{gather}
here $\epsilon_{a,b} = \prod_r (-1)^{a_r[0]}$ $($we consider the product over such $r$ that $a-1 \geqslant r \geqslant b$ for $a>b$ and $b-1 \geqslant r \geqslant a$ for $b>a)$.

The representation obtained is isomorphic to $\mathcal{F}_u^{\sigma}$.
\end{Theorem}

\subsection{Strange bosonic construction}

We will use notation of Section~\ref{boson}. Let $\zeta$ be a $n$th primitive root of unity, e.g., $\zeta = e^{\frac{2\pi i}{n}}$.

\begin{Theorem} \label{Th:StrangeBoson}
There is an action of $\qD$ on $F^a_{\alpha}$ determined by the formulas
\begin{gather}
H^{tw}_k = a_{nk}, \qquad c =n, \qquad c' = n',\nonumber\\
 E^{tw}(z) = z^{n'/n} \frac{u^{\frac{1}{n}}}{n\big(1-q^{1/n}\big)} \sum_{l=0}^{n-1} \zeta^{l n'} : \! \exp \! \left( \sum_{k} \frac{q^{-k/2n}-q^{k/2n}}{k} a_{k} \zeta^{-kl} z^{-k/n} \right) \! :, \label{eq:Th Strange boson E} \\
 F^{tw}(z) = z^{-n'/n} \frac{u^{-\frac{1}{n}}}{n\big(1-q^{-1/n}\big)} \sum_{l=0}^{n-1} \zeta^{-l n'} : \! \exp \! \left( \sum_{k}\frac{q^{k/2n}-q^{-k/2n}}{k} a_{k} \zeta^{-kl} z^{-k/n} \right) \! : . \label{eq:Th Strange boson F}
\end{gather}
The representation obtained is isomorphic to $\mathcal{F}_u^{\sigma}$.
\end{Theorem}

As before the representation does not depend on $\alpha$, see Remark~\ref{remark:NOalpha}.

\section{Twisted representation via a sublattice}\label{S:sublattice}

\subsection{Sublattices and subalgebras}

Consider a full rank sublattice $\Lambda \subset \mathbb{Z}^2$ of index $n$ (i.e., $\mathbb{Z}^2/\Lambda$ is a finite group of order $n$). Let us define a Lie subalgebra $\qD^{\Lambda} \subset \qD$ which is spanned by $E_{a,b}$ for $(a,b) \in \Lambda$ and central elements~$c$,~$c'$.

Denote by $E_{a,b}^{[n]}$, $c_{[n]}$, $c'_{[n]}$ standard generators of $\qDn$. Let $v_1 = (k_1, l_1)$ and $v_2 = (k_2, l_2)$ be a basis of $\Lambda$. Define a map $\varphi_{v_1, v_2}\colon \qDn \rightarrow \qD$
\begin{gather*}
\varphi_{v_1, v_2} E^{\scriptstyle{[n]}}_{a,b} = E_{a k_1 + b k_2, a l_1 + b l_2}, \qquad
\varphi_{v_1, v_2} c'_{\scriptstyle{[n]}} = k_1 c' + l_1 c, \qquad
\varphi_{v_1, v_2} c_{\scriptstyle{[n]}} = k_2 c' + l_2 c.
\end{gather*}

\begin{Proposition}
Let $v_1$, $v_2$ be a positively oriented basis $($i.e., $k_1 l_2 - k_2 l_1 =n)$. Then the map~$\varphi_{v_1, v_2}$ is a Lie algebra isomorphism $\qDn \cong \qD^{\Lambda}$.
\end{Proposition}

\begin{proof}It follows from \eqref{qDc} directly.
\end{proof}

Slightly abusing notation, denote the Fock representation of $\qDn$ by $\mathcal{F}_u^{[n]}$ .

\begin{Proposition} \label{LambdaFock}
Let $v_1 = (n,0)$ and $v_2 =(-m, 1)$. Then $ \mathcal{F}_{u^{n}}^{[n]} \cong \left. \mathcal{F}_u \right|_{\phi_{v_1, v_2} \left(\qd{n} \right)}$ as $\qd{n}$-modules.
\end{Proposition}

\begin{proof}
Note that the Fock module $\mathcal{F}_u$ is $\mathbb{Z}$-graded with grading given by
\begin{gather*}
\deg(| \alpha \rangle) = 0, \qquad \deg (E_{a,b})=-b.
\end{gather*}
Recall that a character of a $\mathbb{Z}$-graded module is the generating function of dimensions of the graded components. Then the character of Fock module $\ch \mathcal{F}_u = 1/(\mathfrak{q})_{ \infty } := \prod\limits_{k=1}^{\infty} 1/\big(1-\mathfrak{q}^k\big)$.

Consider a subalgebra $\Hem_0$ in $\qd{n}$ spanned by $E_{0, k}$ and $c$. Note that $\Hem_0$ is isomorphic to the Heisenberg algebra. Since $\deg (\phi_{v_1,v_2}E_{0,-k} )= \deg (E_{km,-k}) = k$, the character of the $\Hem_0$-Fock module is also $1/(\mathfrak{q})_{\infty}$; i.e., it coincides with $\ch \mathcal{F}_u$. This implies that $\left. \mathcal{F}_u \right|_{\phi_{v_1, v_2} \left(\qd{n} \right)}$ restricted to $\Hem_0$ is isomorphic to the $\Hem_0$-Fock module. To finish the proof, we use Propositions \ref{Fock} and \ref{E_l}.
\end{proof}

\subsection{Twisted Fock vs restricted Fock}
From now on we change $q \rightarrow q^{1/n}$.
Our goal is to construct an action of $\qD$ on the Fock module twisted by $\sigma \in {\rm SL}_2 ( \mathbb{Z} )$ as in \eqref{sigma} for $n \neq 0$. Consider a sublattice $\Lambda_{\sigma} \subset \mathbb{Z}^2$ spanned by $v_1 = (n,0)$ and $v_2 = (-m, 1)$. Consider another basis of $\Lambda_{\sigma}$ obtained by $\sigma$
\begin{gather}
w_1= m' v_1 + n' v_2 = (m' n - n' m, n') = (1, n'), \label{eq: v to w change 1} \\
w_2= m v_1 + n v_2 = (0, n). \label{eq: v to w change 2}
\end{gather}

\begin{figure}[t]\centering
\begin{tikzpicture}[scale=1.5]
\draw[step=.5cm,gray,very thin] (-3.3,-2.3) grid (4.3,2.3);
\draw[thin] (-3.3,0) -- (4.3,0);
\draw[thin] (0,-2.3) -- (0,2.3);
\foreach \point in {(-3,0),(0,0),(1.5,0)}
{
\draw[fill] \point circle (1pt);
\draw[fill] \point+(1.5,0) circle (1pt);
\draw[fill] \point circle (1pt);
\draw[fill] \point+(1,0.5) circle (1pt);
\draw[fill] \point+(2.5,0.5) circle (1pt);
\draw[fill] \point+(0.5,1) circle (1pt);
\draw[fill] \point+(2,1) circle (1pt);
\draw[fill] \point+(0,1.5) circle (1pt);
\draw[fill] \point+(1.5,1.5) circle (1pt);\draw[fill] \point+(0,-1.5) circle (1pt);
\draw[fill] \point+(1.5,-1.5) circle (1pt);
\draw[fill] \point+(1,-1) circle (1pt);
\draw[fill] \point+(2.5,-1) circle (1pt);
\draw[fill] \point+(0.5,-0.5) circle (1pt);
\draw[fill] \point+(2,-0.5) circle (1pt);
\draw[fill] \point+(1,2) circle (1pt);
\draw[fill] \point+(2.5,2) circle (1pt);
\draw[fill] \point+(0.5,-2) circle (1pt);
\draw[fill] \point+(2,-2) circle (1pt);
}
\draw[->,>=stealth,thick,blue] (0,0) to (0.5,1) node[anchor=south west] {$w_1$};
\draw[->,>=stealth,thick,blue] (0,0) to (0,1.5) node[anchor=south west] {$w_2$};
\draw[->,>=stealth,thick,red] (0,0) to (1.5,0) node[anchor=south west] {$v_1$};
\draw[->,>=stealth,thick,red] (0,0) to (-0.5,0.5) node[anchor=south west] {$v_2$};
\end{tikzpicture}
\caption{Lattice $\Lambda_\sigma$ for $n=3$, $m=1$.}\label{fig:Lambda_sigma}
\end{figure}
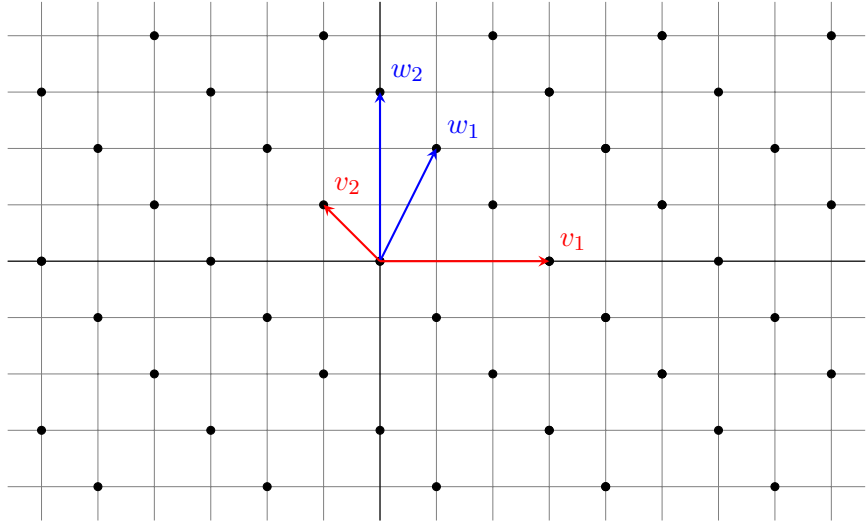

\begin{Remark}
The construction of the sublattice $\Lambda_{\sigma} \subset \mathbb{Z}^2$ naturally appears, if one require $\sigma$ to be a transition matrix from $v_i$ to $w_i$ and assume $v_1 = (n, 0)$, $w_2 = (0, n)$.
\end{Remark}

Denote the Fock module of $\qd{1/n}$ by $\mathcal{F}^{[1/n]}_{u}$.

\begin{Theorem}\label{MainTh}There is an isomorphism of $\qD$-modules $\left. ( \mathcal{F}_{u})^{ \sigma } \cong \mathcal{F}^{[1/n]}_{u^{1/n}} \right|_{\phi_{w_1,w_2} (\qD) }$.
\end{Theorem}

\begin{proof}Proposition \ref{LambdaFock} implies $\left. \mathcal{F}^{[1/n]}_{u^{1/n}} \right|_{\phi_{v_1,v_2}(\qD)} \cong \mathcal{F}_{u}$. On the other hand, relations~\eqref{eq: v to w change 1} and~\eqref{eq: v to w change 2} yield that $\sigma$ is the transition matrix from $w_1$, $w_2$ to $v_1$, $v_2$.
\end{proof}

\begin{Corollary}There is an isomorphism of $\qD$-modules $\left. ( \mathcal{M}_{u} )^{ \sigma } \cong \mathcal{M}^{[1/n]}_{u^{1/n}} \right|_{\phi_{w_1,w_2} (\qD ) }$.
\end{Corollary}

Theorem \ref{MainTh} combined with results from Section \ref{level1} enables us to find explicit formulas for action on $\mathcal{F}_{u}^{ \sigma }$. We will do this below.

\subsection{Explicit formulas for restricted Fock}
\subsubsection{Fermionic construction via sublattice}
Denote fermionic representation of $\qd{1/n}$ by $\mathcal{M}^{[1/n]}_u$. To be more specific, let us rewrite formulas from Section~\ref{fermi} for $\mathcal{M}^{[1/n]}_{u^{1/n}}$.
\begin{align}
&c \mapsto 1, \qquad c' \mapsto 0, \qquad H_m \rightarrow \sum_{i+j=m} \psi_i \psi_j^*, \nonumber\\
&E(z) \mapsto \frac{q^{l/n} u^{1/n} }{1-q^{1/n}} + u^{1/n} q^{ - 1/2n }z : \! \psi\big( q^{-1/2n} z\big) \psi^* \big( q^{1/2n} z\big) \! :_{(l)} ,
\label{ferm:En} \\
&F(z) \mapsto \frac{q^{-l/n} u^{-1/n} }{1-q^{-1/n}} + u^{-1/n} q^{1/2n} z : \! \psi\big( q^{1/2n} z\big) \psi^* \big(q^{-1/2n} z\big)\!:_{(l)}.
\label{ferm:Fn}
\end{align}
\begin{Proposition} \label{TheFermionicProp}
The following formulas below determine an action of $\qD$ on $F^{n\psi}$
\begin{gather}
c=n, \qquad c' = n_{tw}, \qquad
H^{tw}_k = \sum_{a} \sum_{i+j=k} \psi_a [i] \psi^*_a [j] , \label{eq1: TheFermionicProp} \\
E^{tw}(z) = \sum_{b-a \equiv -n_{tw} \bmod n} u^{\frac{1}{n}} q^{-1/2} z \psia \big( q^{-1/2} z\big) \psib \big(q^{1/2} z\big) z^{\frac{n_{tw}-a+b}{n}} q^{(a+b)/2n}, \label{prop:EtwFermi} \\
F^{tw}(z) = \sum_{b-a \equiv n_{tw} \bmod n} u^{-\frac{1}{n}} q^{1/2} z \psia \big( q^{1/2} z\big) \psib \big(q^{-1/2} z\big) z^{\frac{-n_{tw}-a+b}{n}} q^{-(a+b)/2n}. \label{prop:FtwFermi}
\end{gather}
The module obtained is isomorphic to $\left. \mathcal{M}^{[1/n]}_{u^{1/n}} \right|_{\phi_{w_1, w_2}(\qD)}$ for $w_1=(1, n_{tw})$, $w_2=(0,n)$.
\end{Proposition}

\begin{Remark}
Below we will substitute $n_{tw}=n'$ to prove Theorem \ref{TheFermionicTh}. However, Proposi\-tion~\ref{TheFermionicProp} is more general, than it is necessary for the proof, since we do not assume here that $\gcd(n, n_{tw})=1$. We will need the case of arbitrary $n_{tw}$ in Section~\ref{Section: general sublattice}.
\end{Remark}

\begin{proof}
We use the notation $E(z)$, $F(z)$ and $H(z)$ for the Chevalley generators of $\mathfrak{Diff}_{q^{1/n}}$. The generators of $\qD \cong \mathfrak{Diff}_{q^{1/n}}^{\Lambda}$ (identified by $\varphi_{w_1,w_2}$) will be denoted by $E^{tw}(z)$, $F^{tw}(z)$ and $H^{tw} (z)$. Let us write the identification $\phi_{w_1,w_2}$ explicitly for the Chevalley generators
\begin{gather}
H^{tw}(z) = \sum_k H_{ 0, nk} z^{-k}, \label{eq: tw chevalleyH}\\
E^{tw}(z) = z^{n_{tw}/n} \sum_{k \equiv n_{tw} \bmod n} E_{k} z^{-k/n}, \\
F^{tw}(z) = z^{-n_{tw}/n} \sum_{k \equiv -n_{tw} \bmod n} F_{ k} z^{-k/n} .\label{eq: tw chevalleyF}
\end{gather}

Let us consider currents $\psia (z)$ and $\psib (z)$ for $a,b = 0, 1, \dots, n-1$. These currents are defined by following equality
\begin{gather}
z \psi(z) = \sum_{a=0}^{n-1} z^{n-a} \psia (z^n) , \qquad
\psi^*(z) = \sum_{b=0}^{n-1} z^{b} \psib (z^n).\label{defa1}
\end{gather}
Let us denote the modes of $\psia (z)$ and $\psib (z)$ as in equality~\eqref{eq:nFermions}. It is easy to see that these modes satisfy Clifford algebra relations \eqref{eq:nClifford1}, \eqref{eq:nClifford2}. So we have identified the Clifford algebra and the $n$th power of Clifford algebra. This leads to an identification $F^{\psi}=F^{n \psi}$.

Substituting \eqref{defa1} into \eqref{ferm:En} and \eqref{ferm:Fn}, we obtain
\begin{gather*}
E(z) = \frac{q^{l/n} u^{1/n} }{1-q^{1/n}} + \sum_{a=0}^{n-1} \sum_{b=0}^{n-1} u^{1/n} q^{\frac{a+b}{2n}} q^{-1/2} z^{n-a+b} : \! \psia \big( q^{-1/2} z^n\big) \psib \big(q^{1/2} z^n\big) \! :_{(l)}, \\
F(z) = \frac{q^{-l/n} u^{-1/n} }{1-q^{-1/n}} + \sum_{a=0}^{n-1} \sum_{b=0}^{n-1} u^{-1/n} q^{- \frac{a+b}{2n}} q^{1/2} z^{n-a+b} : \! \psia \big( q^{1/2} z^n \big) \psib \big(q^{-1/2} z^n \big)\!:_{(l)}.
\end{gather*}
For technical reasons, we need to treat the cases $n_{tw} \neq 0$ and $n_{tw} =0$ separately. Let us first consider the case $n_{tw} \neq 0$. Using formulas \eqref{eq: tw chevalleyH}--\eqref{eq: tw chevalleyF}, we see that
\begin{gather*}
E^{tw}(z) = \sum_{a-b \equiv n_{tw} \bmod n} \! \! q^{\frac{a+b}{2n}} u^{1/n} z^{\frac{n_{tw}-a+b}{n}} q^{-1/2} z \psia \big( q^{-1/2} z\big) \psib \big(q^{1/2} z\big), \\
F^{tw} (z) = \sum_{a-b \equiv -n_{tw} \bmod n} \! \!q^{\frac{-(a+b)}{2n}} u^{-1/n} z^{\frac{-n_{tw}-a+b}{n}} q^{1/2} z \psia \big( q^{1/2} z \big) \psib \big(q^{-1/2} z \big) .
\end{gather*}
For $n_{tw}=0$ we obtain
\begin{gather*}
E^{tw}(z) = \frac{q^{l/n} u^{1/n} }{1-q^{1/n}} + \sum_{a=0}^{n-1} u^{1/n} q^{a/n} q^{-1/2} z : \! \psia \big( q^{-1/2} z\big) \psiaa \big(q^{1/2} z\big) \! :_{(l)}, \\
F^{tw}(z) = \frac{q^{-l/n} u^{-1/n} }{1-q^{-1/n}} + \sum_{a=0}^{n-1} u^{-1/n} q^{-a/n} q^{1/2} z : \! \psia \big( q^{1/2} z \big) \psiaa \big(q^{-1/2} z \big)\!:_{(l)}.
\end{gather*}
This can be rewritten as
\begin{gather*}
E^{tw}(z) = \sum_{a=0}^{n-1} u^{1/n} q^{a/n} \left( \frac{q^{ \ceil{\frac{l-a}{n}} }}{1-q} +q^{-1/2} z : \! \psia \big( q^{-1/2} z\big) \psiaa \big(q^{1/2} z\big) \! :_{(l)} \right), \\
F^{tw}(z) = \sum_{a=0}^{n-1} u^{-1/n} q^{-a/n} \left( \frac{q^{- \ceil{\frac{l-a}{n}} }}{1-q^{-1}} + q^{1/2} z : \! \psia \big( q^{1/2} z \big) \psiaa \big(q^{-1/2} z \big)\!:_{(l)} \right) .
\end{gather*}
Note that the $l$-dependent normal ordering is defined in terms of $\psi_i$ and $\psi^*_j$. One can check (cf.~\eqref{eq: regul prod psipsi})
\begin{align} \label{eq: psia OPE}
\psia(z) \psiaa(w) = \frac{(w/z)^{\ceil{\frac{l-a}{n}} }}{z(1-w/z)} + : \! \psia (z) \psiaa (w) \! :_{(l)}.
\end{align}
Hence
\begin{equation*}
q^{-1/2} z \psia\big(q^{-1/2}z\big) \psiaa\big(q^{1/2}z\big) = \frac{q^{\ceil{\frac{l-a}{n}} }}{1-q} + q^{-1/2} z : \! \psia \big(q^{-1/2}z\big) \psiaa \big(q^{1/2}z\big) \! :_{(l)}.\tag*{\qed}
\end{equation*}\renewcommand{\qed}{}
\end{proof}

\begin{proof}[Proof of Theorem \ref{TheFermionicTh}] Follows from Theorem \ref{MainTh} and Proposition \ref{TheFermionicProp}.
\end{proof}

\subsubsection{Bosonic construction via sublattices}

\begin{Proposition} \label{Prop:qD Boson}
There is an action of $\qD$ on $F^{na} \otimes \mathbb{C} [ \Q ]$ determined by the following formulas
\begin{align}
H^{tw}[k] =& \sum_{b} a_b [k], \quad c' = n_{tw}, \quad c=n, \label{eq1: Prop:qD Boson} \\
E^{tw}(z) =& \!\!\! \sum_{b-a \equiv -n_{tw} \bmod n} \!\!\!\!\! u^{\frac{1}{n}} q^{\frac{a+b-n}{2n}} z^{\frac{n_{tw}-a+b}{n}+1} : \! \exp \left( \phi_b(q^{1/2} z) - \phi_a(q^{-1/2} z) \right) \! : \epsilon_{a,b}, \\
F^{tw}(z) =& \!\!\! \sum_{b-a \equiv n_{tw} \bmod n} \!\!\!\!\! u^{-\frac{1}{n}} q^{ \frac{-a-b+n}{2n} } z^{\frac{-n_{tw}-a+b}{n}+1} : \! \exp \left( \phi_b(q^{-1/2} z) - \phi_a(q^{1/2} z) \right) \! : \epsilon_{a,b} , \label{eq3: Prop:qD Boson}
\end{align}
here $\epsilon_{a,b} = \prod_r (-1)^{a_r[0]}$ $($we consider the product over such $r$ that $a-1 \geqslant r \geqslant b$ for $a>b$ and $b-1 \geqslant r \geqslant a$ for $b>a)$.

The representation obtained is isomorphic to $\left. \mathcal{F}^{[1/n]}_{u^{1/n}} \right|_{\phi_{w_1, w_2}(\qD)}$ for $w_1=(1, n_{tw})$, $w_2=(0,n)$.
\end{Proposition}
\begin{proof}[Proof Proposition \ref{Prop:qD Boson}]
We need an upgraded version boson-fermion correspondence for the proof.
Namely, there is an action of $n$th tensor power of the Heisenberg algebra on $F^{n \psi}$ given by
\begin{gather*}
\partial \phib(z) = : \! \psibb (z) \psib (z) \! :_{(0)}.
\end{gather*}
Let $\P$ be a lattice spanned $Q_b$. According to boson-fermion correspondence $F^{n \psi} \cong \bigoplus F^{na} \otimes \mathbb{C} [ \P ]$.

\begin{Lemma} \label{lemma: subst boson fermi}
Vector subspace $F^{na} \otimes \mathbb{C} [ \Q ] \subset F^{na} \otimes \mathbb{C} [ \P ] = F^{n \psi}$ is a $\qD$-submodule $($with respect to action, defined in Proposition~{\rm \ref{TheFermionicProp})}. The action of $\qD$ on the subrepresentation is given by \eqref{eq1: Prop:qD Boson}--\eqref{eq3: Prop:qD Boson}.
\end{Lemma}
\begin{proof}
One should substitute
\begin{gather}
\psibb (z) = : \! \exp \left( -\phib (z) \right) \! : (-1)^{\sum\limits_{k=0}^{b-1} a_k[0] }, \label{eq:boson psi}\\
\psib (z) = : \! \exp \left( \phib (z) \right) \! : (-1)^{\sum\limits_{k=0}^{b-1} a_k[0] } \label{eq:boson psi*}
\end{gather}
into fermionic formulas \eqref{eq1: TheFermionicProp}--\eqref{prop:FtwFermi}.
\end{proof}

Recall that decomposition of $\mathcal{M}^{[1/n]}_{u}$ is given by eigenvalues of $a[0]$; more precisely, operator~$a[0]$ acts by $-j$ on $ \mathcal{F}^{[1/n]}_{q^{j/n} u}$.
\begin{Lemma} \label{lemma: zero mode of a, npsi is psi}
Using identification $F^{\psi}=F^{n \psi}$ $($cf.~\eqref{defa1}$)$, we obtain $a[0] = a_0[0]+\cdots + a_{n-1}[0]$.
\end{Lemma}
\begin{proof}[Proof]
This follows from $\big\lceil \frac{l-b}{n}\big\rceil =0$ for $l=0$ and $b=0, \dots, n-1$ (cf. \eqref{eq: psia OPE}).
\end{proof}
Lemma \ref{lemma: zero mode of a, npsi is psi} implies that the identification of vector spaces $F^{ \psi}=F^{n\psi}$ leads to identification of subspaces $F^a_{0} = F^{na}\otimes \mathbb{C} [\Q]$. Let us package identifications of vector subspaces into a~commutative diagram
\[
\begin{tikzcd}
\left.\mathcal{F}^{[1/n]}_{u^{1/n}} \right|_{\phi_{w_1, w_2}(\qD)} \arrow[d, hook'] \arrow[r,equal] &
\mathcal{F}^{[1/n]}_{ u} \arrow[d, hook'] \arrow[r,equal] & F^a_{0} \arrow[d, hook'] \arrow[r,equal] & F^{na}\otimes \mathbb{C} [\Q] \arrow[d, hook']
\\
\left. \mathcal{M}^{[1/n]}_{u^{1/n}} \right|_{\phi_{w_1, w_2}(\qD)}
\arrow[r,equal] &
\mathcal{M}^{[1/n]}_{u} \arrow[r,equal] &
F^{\psi} \arrow[r,equal] &
F^{n \psi}.
\end{tikzcd}
\]
Proposition \ref{TheFermionicProp} states that formulas \eqref{eq1: TheFermionicProp}--\eqref{prop:FtwFermi} gives an action of $\qD$ with respect to identification of bottom line of the diagram. Therefore, Lemma~\ref{lemma: subst boson fermi} implies that formulas \eqref{eq1: Prop:qD Boson}--\eqref{eq3: Prop:qD Boson} describes the action of~$\qD$ with respect to identification of top line of the diagram.
\end{proof}

\begin{proof}[Proof of Theorem \ref{Th:qD Boson}] Follows from Theorem \ref{MainTh} and Proposition \ref{Prop:qD Boson}.
\end{proof}

\subsubsection{Strange bosonic construction via sublattices}
\begin{Proposition} \label{Prop:StrangeBoson}
There is an action of $\qD$ on $F^a_{\alpha}$ defined by formulas
\begin{gather*}
c = n, \qquad
c' = n_{tw}, \qquad
H^{tw}_k = a_{nk}, \\
E^{tw}(z) = z^{n_{tw}/n} \frac{u^{\frac{1}{n}}}{n\big(1-q^{1/n}\big)} \sum_{l=0}^{n-1} \zeta^{l n_{tw}} : \! \exp \! \left( \sum_{k} \frac{q^{-k/2n}-q^{k/2n}}{k} a_{k} \zeta^{-kl} z^{-k/n} \right) \! :, \\
F^{tw}(z) = z^{-n_{tw}/n} \frac{u^{-\frac{1}{n}}}{n\big(1-q^{-1/n}\big)} \sum_{l=0}^{n-1} \zeta^{-l n_{tw}} : \! \exp \! \left( \sum_{k}\frac{q^{k/2n}-q^{-k/2n}}{k} a_{k} \zeta^{-kl} z^{-k/n} \right) \! : .
\end{gather*}
Obtained module is isomorphic to $\left. \mathcal{F}^{[1/n]}_{u^{1/n}} \right|_{\phi_{w_1, w_2}(\qD)}$ for $w_1=(1, n_{tw})$, $w_2=(0,n)$.
\end{Proposition}

\begin{proof}
Formulas \eqref{eq: tw chevalleyH}--\eqref{eq: tw chevalleyF} imply
\begin{gather}
H^{tw}_k = H_{nk}, \label{eq: proof of prop5.4 H}\\
E^{tw} (z) = \frac{1}{n} z^{n_{tw}/n} \sum_{l=0}^{n-1} \zeta^{ln_{tw}} E\big( \zeta^{l} z^{1/n} \big), \label{eq: proof of prop5.4 E} \\
F^{tw} (z) = \frac{1}{n} z^{-n_{tw}/n} \sum_{l=0}^{n-1} \zeta^{-ln_{tw}} F\big( \zeta^{l} z^{ 1/n } \big). \label{eq: proof of prop5.4 F}
\end{gather}
Substitution of $\mathfrak{Diff}_{q^{1/n}}$-version of \eqref{boson1}--\eqref{eq:F} to \eqref{eq: proof of prop5.4 H}--\eqref{eq: proof of prop5.4 F} finishes the proof.
\end{proof}
\begin{proof}[Proof of Theorem \ref{Th:StrangeBoson}] Follows from Theorem \ref{MainTh} and Proposition \ref{Prop:StrangeBoson}.
\end{proof}

\section{Twisted representation via a Semi-infinite construction}\label{S:Semi-infinite}

This section is devoted to another proof of the Theorem \ref{TheFermionicTh}. So we use the same notation
\[ \sigma = \begin{pmatrix}
m' & m \\
n' & n
\end{pmatrix}.\]

{\bf Twisted evaluation representation.} Let $e_{a,b}$ be a matrix unit (all entries are 0 except for one cell, where it is 1; this cell is in $b$th column and $a$th row).

Consider a homomorphism $\mathfrak{t}_{u, \sigma}\colon \qaD \rightarrow \qaD \otimes \mathrm{Mat}_{n \times n}$ defined by
\begin{gather}
E_{0,k} \mapsto E_{0, k} \otimes 1,\nonumber \\
E_{1,k} \mapsto u^{\frac{1}{n}} \sum_{b-a \equiv - n' \! \mod n} q^{\frac{a+b}{2n}} E_{1, k+ \frac{b-a+n'}{n}} \otimes e_{a,b}, \label{def:Etwev} \\
E_{-1,k} \mapsto u^{-\frac{1}{n}} \sum_{b-a \equiv n' \! \mod n} q^{-\frac{a+b}{2n}} E_{-1, k+ \frac{b-a-n'}{n}} \otimes e_{a,b}.\nonumber
\end{gather}

Algebra $\qaD \otimes \mathrm{Mat}_{n \times n}$ tautologically acts on $\mathbb{C}^n \big[z,z^{-1}\big]$. Therefore, homomorphism $\mathfrak{t}_{u, \sigma}$ induces an action of $\qaD$ on $\mathbb{C}^n [z,z^{-1}]$.

\begin{Proposition} \label{proposition:twev}
Obtained representation of $\qaD$ is isomorphic to $V_u^{\sigma}$.
\end{Proposition}

\begin{proof}

Consider a basis $v_l := q^{\frac{m l^2}{2n}} u^{ \frac{ml}{n} } x^l$ of evaluation representation $\mathbb{C}\big[x, x^{-1}\big]^{\sigma}$\!. Action with respect to this basis looks like
\begin{gather}
E_{0,k} v_l = v_{k+l}, \label{twevH} \\
E_{1,k} v_l = u^{\frac{1}{n}} q^{\frac{n'+kn+2l}{2n}} v_{l + nk + n'}, \label{twevE}\\
E_{-1,k} v_l = u^{-\frac{1}{n}} q^{\frac{n' - kn -2l}{2n}} v_{l+nk-n'}.\nonumber
\end{gather}

Let $a,b =0, \dots, n-1$ be such numbers that $l=nj+b$ and $a \equiv b + n' \mod{n}$. Substituting $l=nj+b$ into \eqref{twevE} we obtain
\begin{gather}
E_{1,k} v_{nj+b} = u^{\frac{1}{n}} q^{\frac{a+b}{2n}} q^{\frac{n'+b-a}{2n}} q^{\frac{k}{2}} q^j v_{n(k+j + \frac{n'+b-a}{n} ) + a}. \label{eq:EtwevINTER}
\end{gather}
Let us identify $\mathbb{C}^n \big[z, z^{-1}\big] \xrightarrow{\sim} \mathbb{C}\big[x,x^{-1}\big] $ by $z^j e_b \mapsto v_{nj+b}$. Then formula \eqref{eq:EtwevINTER} will be rewritten
\begin{gather*}
E_{1,k} (z^j e_b) = u^{\frac{1}{n}} q^{\frac{a+b}{2n}} \big( E_{1, k+ \frac{b-a+n'}{n}} \otimes e_{a,b} \big) \big( z^j e_b \big).
\end{gather*}

To be compared with formula \eqref{def:Etwev} this proves the proposition for $E_{1,k}$. The proof for $E_{-1,k}$ is analogous. For $E_{0,k}$ proposition is obvious from \eqref{twevH}.
\end{proof}

{\bf Semi-infinite construction.} To apply semi-infinite construction we need to pass from associative algebras to Lie algebras.

\begin{Definition}
Algebra $\Torgln$ is a Lie algebra with basis $E_{k,l} \otimes e_{a,b}$ (where $(k, l) \in \mathbb{Z}^2 \backslash (0,0)$ and $a,b = 0, \dots, n-1$), $c$ and $c'$. Elements $c$ and $c'$ are central. All other commutators are given by
\begin{gather*}
[ E_{k_1,l_1}\otimes e_{a_1, b_1} , E_{k_2,l_2} \otimes e_{a_2,b_2} ] = E_{k_1+k_2, l_1+l_2} \\
\qquad{} \otimes \big( q^{\frac{l_2 k_1 -l_1 k_2}{2}} \delta_{b_1, a_2} e_{a_1, b_2} -q^{\frac{l_1 k_2-l_2 k_1}{2}} \delta_{b_2, a_1} e_{a_2, b_1} \big) + \delta_{k_1, -k_2} \delta_{l_1, -l_2} \delta_{a_2,b_1} \delta_{a_1, b_2} ( c l_1 + c' k_1 ).
\end{gather*}
\end{Definition}
\begin{Proposition} \label{proposition:toidalGL}
There is an action of $\Torgln$ on $F^{n \psi}$ given by formulas
\begin{gather*}
c \mapsto 1, \qquad c' \mapsto 0, \\
E(z) \otimes e_{a,b} \mapsto q^{-\half} z \psia \big(q^{-\half} z\big) \psib \big(q^{\half} z\big), \\
F(z) \otimes e_{a,b} \mapsto q^{\half} z \psia \big(q^{\half} z\big) \psib \big(q^{{-}\half} z\big).
\end{gather*}
Obtained representation is isomorphic to $\siL \mathbb{C}^n \big[z,z^{-1}\big]$.
\end{Proposition}
\begin{proof}[Proof of Theorem \ref{TheFermionicTh}]
According to Proposition \ref{Prop:SemiInf-Fermi}, $\mathcal{M}_u \cong \siL V_u$. Then $\mathcal{M}_u^{\sigma} \cong \siL V_u^{\sigma}$. Therefore, Propositions~\ref{proposition:twev} and~\ref{proposition:toidalGL} imply Theorem~\ref{TheFermionicTh}.
\end{proof}

\section[$q$-$W$-algebras]{$\boldsymbol{q}$-$\boldsymbol{W}$-algebras} \label{section: q-W}

\subsection{Definitions}
{\bf Topological algebras and completions.} In this section we will work with topological algebras. Let us define topological algebra appearing as a completion of $\qD$. It is given by projective limit of $U(\qD)/J_k$ where $J_k$ is the left ideal generated by non-commutative polynomials in $E_{j_1,j_2}$ of degree $-k$ (with respect to grading $\deg E_{j_1, j_2} = -j_2$). Although each $U(\qD)/J_k$ does not have a structure of algebra, so does the projective limit. Moreover, the projective limit has natural topology.

Below we will ignore all corresponding technical problems concerning completions and topo\-logy. We will use term `generators' instead of `topological generators', the same notation for~$\qD$, and its completion and so on.

\subsubsection[Non-twisted $W$-algebras]{Non-twisted $\boldsymbol{W}$-algebras}
Let us introduce a notation
\begin{gather*}
\sum_{l=0}^{\infty}f_{k,n}[l] x^{l} = f_{k, n} (x)= \frac{ (1-qx)^{\frac{n-k}{n}} \big(1-q^{-1}x\big)^{ \frac{n-k}{n}
}}{(1-x)^{\frac{2(n-k)}{n}}}.
\end{gather*}

\begin{Definition} \label{def: non-twisted W}
Algebra $\W$ is generated by $T_k[r]$ for $r \in \mathbb{Z}$ and $k= 1, \dots, n-1$. It is convenient to add generators $T_0[r]=T_n[r] = \delta_{r,0}$. The defining relations are
\begin{gather}
\sum_{l=0}^{\infty} f_{k,n}[l] \big( T_1[r{-}l] T_k[s{+}l] - T_k[s{-}l] T_1[r{+}l] \big) = -\big(q^{\half}{-}q^{-\half}\big)^2 (kr-s) T_{k+1}[r{+}s], \label{relation:modesW1} \\
\sum_{l=0}^{\infty} f_{n{-}k,n}[l] \big( T_{n{-}1} [r{-}l] T_k[s{+}l] - T_k[s{-}l] T_{n{-}1}[r{+}l] \big) \nonumber\\
\qquad{} = -\big(q^{\half}{-}q^{-\half}\big)^2 ((n{-}k)r-s ) T_{k-1}[r{+}s] .\label{relation:modesWn-1}
\end{gather}
\end{Definition}

Introduce currents $T_k (z) = \sum\limits_{r \in \mathbb{Z} } T_k[r] z^{-r}$. Then relations \eqref{relation:modesW1}--\eqref{relation:modesWn-1} can be rewritten in current form
\begin{gather}
f_{k,n} (w/z) T_1(z) T_k (w) - f_{k,n}(z/w) T_k(w) T_1 (z) \nonumber\\
\qquad{}= -(q^{\half}-q^{-\half})^2 \left( (k+1) \frac{w}{z} \delta' \left(\frac{w}{z}\right) T_{k+1} (w) + w \delta\left(\frac{w}{z}\right) \partial_w T_{k+1} (w) \right), \label{relation:CurrentW}
\\
f_{n-k,n} (w/z) T_{n-1}(z) T_k (w) - f_{n-k,n}(z/w) T_k(w) T_{n-1} (z) \nonumber\\
\qquad{} =-\big(q^{\half}-q^{-\half}\big)^2 \left( (n-k+1) \frac{w}{z} \delta' \left(\frac{w}{z}\right) T_{k-1} (w) + w \delta\left(\frac{w}{z}\right) \partial_w T_{k-1} (w) \right). \label{relation:CurrentWn-1}
\end{gather}
Also note that $T_0(z)=T_n(z)=1$.
\begin{Remark}
There are different approaches to definition of $q$-${W}$-algebra. For example, in~\cite{FF} algebra $\mathcal{W}_{q,p} ( \mathfrak{sl}_n )$ was defined via bosonization. The currents $T_k(z)$ satisfy relation \cite[Theorem~2]{FF}
\begin{gather*}
f_{k,n} (w/z) T_1(z) T_k (w) - f_{k,n}(z/w)T_k(w) T_1 (z) \\
\qquad{} = \frac{(1-q)(1-p/q)}{1-p} \big( \delta(w/zp) T_{k+1}(z) - \delta\big(wp^k/z\big) T_{k+1} (w) \big),
\end{gather*}
where
\begin{gather*}
f_{k, n} (x)= \frac{\big(x| p^{m-1} q, p^m q^{-1}, p^n, p^{n-1}; p^n\big)}{\big(x| p^{m-1}, p^m, p^{n-1} q, p^n q^{-1}; p^n\big)}.
\end{gather*}

One can check that limit $p \rightarrow 1$ gives relation \eqref{relation:CurrentW}. However \cite{FF} do not provide presentation of $\mathcal{W}_{q,p} ( \mathfrak{sl}_n )$ in terms of generators and relations.

In the paper \cite{N16} relation~(2.62) defines algebra $\mathcal{W}_{q,p} ( \mathfrak{gl}_n)$ which (non-essentially) differs from $\mathcal{W}_{q,p} (\mathfrak{sl}_n )$ mentioned above (and from $\W$ defined above).
\end{Remark}

\subsubsection[Twisted $q$-$W$-algebras]{Twisted $\boldsymbol{q}$-$\boldsymbol{W}$-algebras}

Twisted $q$-${W}$-algebra depends on remainder of $n_{tw}$ modulo $n$. If $n_{tw}=0$, then we get definition of non-twisted $q$-${W}$-algebra from last section. One can find definition of $\mathcal{W}_{q,p}(\mathfrak{sl}_2, 1)$ in \cite[equations~(37)--(38)]{Sh}.

\begin{Definition} \label{def: twisted W}
Algebra $\twW$ is generated by $\tT_k[r]$ for $r \in n_{tw} k/n + \mathbb{Z}$ and $k= 1, \dots, n-1$. It is convenient to add $\tT_0[r]=\tT_n[r] = \delta_{r,0}$. The defining relations are
\begin{gather}
\sum_{l=0}^{\infty} f_{k,n}[l] \big( \tT_1[r{-}l] \tT_k[s{+}l] - \tT_k[s{-}l] \tT_1[r{+}l] \big) = -\big(q^{\half}-q^{-\half}\big)^2 (kr-s) \tT_{k+1}[r{+}s], \!\!\!\label{eq:TwWrel1} \\
\sum_{l=0}^{\infty} f_{n-k,n}[l] \big( \tT_{n-1} [r{-}l] \tT_k[s{+}l] - \tT_k[s{-}l] \tT_{n-1}[r{+}l] \big)\nonumber\\
\qquad{} = -\big(q^{\half}-q^{-\half}\big)^2 ((n{-}k)r-s ) \tT_{k{-}1}[r{+}s]. \label{eq:TwWrel2}
\end{gather}
\end{Definition}

Let us rewrite relations \eqref{eq:TwWrel1}--\eqref{eq:TwWrel2} in the current form. Define currents
\begin{gather*}
T_k^{tw}(z) := \sum T_k^{tw}[r] z^{-r}, \qquad T_k (z) := z^{ \frac{ \scriptstyle{kn}_{ \scriptscriptstyle{tw} }}{\scriptstyle{n}} } T_k^{tw}(z), \\
T^{\circ}_k (z) := z^{-\frac{\scriptstyle{(n-k) n}_{tw}}{\scriptstyle{n}} } T^{tw}_k (z) = z^{-n_{tw}} T_k(z).
\end{gather*}
Note that
\begin{gather}\label{eq:T boundary}
T_0(z)=T^{\circ}_n(z)=1, \qquad T_n(z) = T^{\circ}_0(z)= z^{n_{tw}}.
\end{gather}
\begin{Proposition}
Relation \eqref{eq:TwWrel1} is equivalent to
\begin{gather}
f_{k,n} (w/z) T_1(z) T_k (w) - f_{k,n}(z/w) T_k(w) T_1 (z) \nonumber\\
\qquad{} = -\big(q^{\half}-q^{-\half}\big)^2 \left( (k+1) \frac{w}{z} \delta' (w/z) T_{k+1} (w) + w \delta (w/z) \partial_w T_{k+1} (w) \right). \label{relation:currentTW1}
\end{gather}
Relation \eqref{eq:TwWrel2} is equivalent to
\begin{gather*}
f_{n-k,n} (w/z) T^{\circ}_{n-1}(z) T^{\circ}_k (w) - f_{n-k,n}(z/w) T^{\circ}_k(w) T^{\circ}_{n-1} (z) \nonumber\\ \qquad{}=-(q^{\half}-q^{-\half})^2 \left( (n-k+1) \frac{w}{z} \delta' (w/z) T^{\circ}_{k-1} (w) + w \delta(w/z) \partial_w T^{\circ}_{k-1} (w) \right). %\label{relation:CurrentTWn-1}
\end{gather*}
\end{Proposition}

\begin{Remark}In non-twisted case we have relations \eqref{relation:CurrentW} and \eqref{relation:CurrentWn-1} for currents $T_k(z)$. In twisted case we have the same relations, but for two different sets of currents $T_k(z)$ and $T_k^\circ(z)$. One should also keep in mind~\eqref{eq:T boundary}.
\end{Remark}

\subsection[Connection of $\twW$ with $\qD$]{Connection of $\boldsymbol{\twW}$ with $\boldsymbol{\qD}$}

Connection between $\W$ and $\qD$ is known (see \cite[Proposition~2.14]{FHSSY} or \cite[Proposition~2.25]{N16}). In this section we generalize it for arbitrary $n_{tw}$.

Let $\mathfrak{Heis}$ be a~Heisenberg algebra generated by $\ha_j$ with relation $\big[\ha_i, \ha_j\big]=ni\delta_{i+j, 0}$. We will prove that there is a surjective homomorphism $\qD \twoheadrightarrow \WH$. Secretly, generators $H_j$ are mapped to $\ha_j$ under the homomorphism. Let us introduce a notation to describe this homomorphism more precisely.

Define
\begin{alignat}{3}
& \varphi_{-} (z) = \sum_{j>0} \frac{q^{-j/2}-q^{j/2}}{j} H_{-j} z^j,\qquad & &
\varphi_{+} (z) = - \sum_{j>0} \frac{q^{j/2}-q^{-j/2}}{j} H_{j} z^{-j}, &\label{def:varphi} \\
& \tphi_{-} (z)= \sum_{j>0} \frac{q^{-j/2}-q^{j/2}}{j} \ha_{-j} z^j, \qquad &&
\tphi_{+} (z)= - \sum_{j>0} \frac{q^{j/2}-q^{-j/2}}{j} \ha_{j} z^{-j}.&\nonumber
\end{alignat}
Also, let introduce notation
\begin{gather}
\varphi(z) =\varphi_-(z)+\varphi_+(z), \qquad \tphi(z) =\tphi_-(z)+\tphi_+(z). \label{def:varphi without pm}
\end{gather}
Define
\begin{gather*}
\T_k (z) = \frac{1}{k!} \exp \left(-\frac{k}{n} \varphi_- (z) \right) E^k (z) \exp \left( - \frac{k}{n} \varphi_+ (z) \right).
\end{gather*}
Note that $\T_k (z)$ commute with $H_{j}$.

Let $J_{\mu ,n,n_{tw}}$ be two sided ideal in $\qD$ generated by $c-n$, $c'-n_{tw}$ and $\T_n(z) - \mu^n z^{n_{tw}}$ (here $\mu \in \mathbb{C} \backslash \{ 0 \}$). Parameter $\mu$ is not essential since automorphism $E_{a,b} \mapsto \mu^{-a} E_{a,b}$ maps $J_{\mu,n, n_{tw}}$ to $J_{1,n, n_{tw}}$. So we will abbreviate $J_{n, n_{tw}} = J_{\mu, n, n_{tw}}$.

\begin{Lemma} \label{lemma:T_k for big k}
$\T_k(z) \in J_{n,n'}$ for $k>n$.
\end{Lemma}
\begin{proof}
It holds in $U(\qD)/J_{n,n_tw}$
\begin{gather*}
E^{k}(z) = n! \mu^n z^{n_{tw}}E^{k-n}(z) : \! \exp{\varphi(z)} \! :.
\end{gather*}
On the other hand,
\begin{gather*}
E^{k-n}(z) :\!\exp{\varphi(w)}\!: = \frac{(z-w)^{2(k-n)}}{(z-qw)^{k-n}\big(z-q^{-1}w\big)^{k-n}} \exp{\varphi_-(w)} E^{k-n}(z) \exp{\varphi_+(w)}.
\end{gather*}
Hence, $E^{k-n}(z) :\!\exp{\varphi(z)} \!:=0$.
\end{proof}

\begin{Theorem} \label{Th:Wiso}
There is an algebra isomorphisms $\WqD \colon \WH \xrightarrow{\sim} U(\qD)/J_{n,n_{tw}}$ such that
\begin{gather}
T_k(z) \mapsto \mu^{-k} \T_k(z), \qquad
\ha_j \mapsto H_j. \label{eq:ThWqD}
\end{gather}
The map $\qDW$ in opposite direction is given by
\begin{gather}
H_j \mapsto \ha_j, \qquad c \mapsto n, \qquad c' \mapsto n_{tw}, \label{eq:ThqDW1}\\
E(z) \mapsto \ \mu \exp \left( \frac{1}{n}{\tphi_-(z)} \right) T_1(z) \exp \left( \frac{1}{n} \tphi_+ (z) \right), \label{eqThqDWE}\\
F(z) \mapsto - \frac{\mu^{-1} z^{-n_{tw}}}{\big(q^{\half}-q^{-\half}\big)^2} \exp \left(- \frac{1}{n}{\tphi_-(z)} \right) T_{n-1}(z) \exp \left(-\frac{1}{n} \tphi_+ (z) \right). \label{eq:ThqDW2}
\end{gather}
\end{Theorem}

The rest of this section is devoted to proof of Theorem \ref{Th:Wiso}. First of all, we will prove that formula \eqref{eq:ThWqD} indeed defines a homomorphism $\WqD \colon \WH \rightarrow U(\qD)/J_{n,n_{tw}}$ (see Proposition~\ref{prop:Hom(W,qD)}). Then we prove that formulas \eqref{eq:ThqDW1}--\eqref{eq:ThqDW2} defines a homomorphism in opposite direction (see Proposition~\ref{prop:HomDWtw}). Finally, we note that maps~$\qDW$ and~$\WqD$ are mutually inverse.

\begin{Proposition} \label{prop:W1}
Currents $\T_k(z)$ %$($considered
$($as power series with coefficients in $U(\qD)/J_{n,n_{tw}})$ satisfy relation \eqref{relation:currentTW1}.
\end{Proposition}

\begin{proof}
Let us define power series in two variables
\begin{gather*}
\mathbb{E}^{(k+1)} (z,w) = (z-qw)\big(z-q^{-1}w\big) E(z) E^{k}(w).
\end{gather*}
According Corollary \ref{Corollary:EE^k}, $\mathbb{E}^{(k+1)} (z,w)$ is regular in sense of Definition~\ref{def:regular}. Following relations follows from results of Appendix~\ref{nonorm}
\begin{gather}
\mathbb{E}^{(k+1)} (z,w) = (z-qw)\big(z-q^{-1}w\big) E^{k}(w) E(z), \label{eq:mathbbEk+1}\\
\mathbb{E}^{(k+1)} (w,w) = (1-q)\big(1-q^{-1}\big) w^2 E^{k+1} (w), \label{eq:mathbbEk+1ww}\\
 \mathbb \partial_z \mathbb{E}^{(k+1)} (z,w) \big|_{z=w} = (1-q)\big(1-q^{-1}\big) w^2 \frac{1}{k+1} \partial_w E^{k+1} (w)\nonumber\\
\hphantom{\mathbb \partial_z \mathbb{E}^{(k+1)} (z,w) \big|_{z=w} =}{} + (1-q)\big(1-q^{-1}\big) w E^{k+1}(w). \label{eq:EE^kderiv}
\end{gather}
More precisely, \eqref{eq:mathbbEk+1}--\eqref{eq:mathbbEk+1ww} easily follows from Propositions \ref{prop:ass&comm}. One can find a proof of \eqref{eq:EE^kderiv} at the end of Appendix \ref{nonorm}.

It is straightforward to check that
\begin{gather}
f_{k,n}(w/z) \T_{1}(z) \T_{k} (w) = \frac{ \big(1-q\frac{w}{z}\big) \big(1-q^{-1}\frac{w}{z}\big)}{k! \big(1-\frac{w}{z}\big)^{2}} \exp \left( - \frac{1}{n} \left( \varphi_-(z) + k \varphi_- (w) \right) \right)\nonumber\\
\hphantom{f_{k,n}(w/z) \T_{1}(z) \T_{k} (w) =}{}
\times E (z) E^k(w) \exp \left( - \frac{1}{n} \left( \varphi_+(z) + k \varphi_+ (w) \right) \right). \label{eq:normoderT1Tk}
\end{gather}

Formulas \eqref{eq:normoderT1Tk} and \eqref{eq:mathbbEk+1}--\eqref{eq:EE^kderiv} implies that
\begin{gather*}
f_{k,n}(w/z) \T_{1}(z) \T_{k} (w) - f_{k,n}(z/w) \T_k(w) \T_1 (z) \\
\qquad{} = \frac{1}{k!} \exp \left( - \frac{1}{n} \left( \varphi_-(z) + k \varphi_- (w) \right) \right) \mathbb{E}^{(k+1)}(z,w) \\
\qquad\quad{}\times \exp \left( - \frac{1}{n} \left( \varphi_+(z) + k \varphi_+ (w) \right) \right) \partial_w \big( w^{-1} \delta(w/z) \big) \\
\qquad {} = (1-q)\big(1-q^{-1}\big)(k+1) \T_{k+1}(w) \frac{w}{z} \delta'(w/z) + (1-q)\big(1-q^{-1}\big) w \T'_{k+1}(w) \delta(w/z).\!\!\!\tag*{\qed}
\end{gather*}
\renewcommand{\qed}{}
\end{proof}

\begin{Lemma} The following OPE holds in $\qD$
\begin{gather}
F(z) E^k(w) = k(c-k+1)\frac{\frac{w}{z} E^{k-1}(w)}{\big(1-\frac{w}{z}\big)^2} \nonumber\\
\hphantom{F(z) E^k(w) =}{} + k \frac{w\partial_w E^{k-1}(w) - w:\!\varphi'(w) E^{k-1}(w)\!: - c'E^{k-1}(w)}{1-\frac{w}{z}} + \mathrm{reg}; \label{eq:OPEFE^k}
\end{gather}
or, equivalently
\begin{gather}
\big[F(z), E^k(w)\big] = k(c-k+1) E^{k-1}(w) \frac{w}{z} \delta'(w/z) \nonumber\\
\hphantom{\big[F(z), E^k(w)\big] =}{} +k \big( w \partial_w E^{k-1}(w) - w:\!\varphi'(w) E^{k-1}(w)\!:-c'E^{k-1}(w) \big) \delta(w/z).\label{eq:commutatorFE^k}
\end{gather}
\end{Lemma}

\begin{proof}
Denote by $E(\mw) = E(w_1) \cdots E(w_k)$. We will write ${:\!F(z) E(\mw)\!:} = F_+ (z)E(\mw)+E(\mw) F_-(z)$; this definition reminds standard Definition \ref{def:norm}, but is applied in different situation ($E(\mw)$ is not a current in one variable). Note that \begin{gather}
F(z) E(\mw) = [F_- (z), E(\mw)] + {:\!F(z) E(\mw)\!:} = \sum_j E(w_1) \cdots E(w_{j-1})\label{eq:technF(z)E(mw)}\\
{}\times \left( \frac{H\big(q^{\half}w_{j}\big) - H\big(q^{-\half}w_{j}\big)-c'}{1-\frac{w_{j}}{z}} +c\frac{w_{j}}{z} \frac{1}{\big(1-\frac{w_j}{z}\big)^2} \right) E(w_{j+1}) \cdots E(w_k) +:\!F(z) E(\mw)\!:. \nonumber
\end{gather}
Recall that $w \varphi'(w) = H(q^{-\half}w) -H(q^{\half}w) $. It follows from \eqref{RqDHE-F}that
\begin{gather*}
\big(H(q^{\half}w_j) -H(q^{-\half}w_j) \big)E(w_l) = -\frac{2E(w_l)}{1-\frac{w_l}{w_j}} + \frac{E(w_l)}{1-q\frac{w_l}{w_j}}+\frac{E(w_l)}{1-q^{-1}\frac{w_l}{w_j}} -w_j :\!\varphi'(w_j) E(w_l)\!:,\\
E(w_j) \big(H\big(q^{\half}w_l\big) -H\big(q^{-\half}w_l\big) \big) = \frac{2E(w_j)}{1-\frac{w_l}{w_j}} - \frac{E(w_j)}{1-q\frac{w_l}{w_j}}-\frac{E(w_j)}{1-q^{-1}\frac{w_l}{w_j}} -w_l :\!\varphi'(w_l) E(w_j)\!:.
\end{gather*}
Using identity
\begin{gather*}
\frac{1}{\big(1-\frac{w_j}{z}\big)\big(1-\frac{w_j}{w_l}\big)}-\frac{1}{\big(1-\frac{w_l}{z}\big)\big(1-\frac{w_j}{w_l}\big)} = - \frac{\frac{w_l}{z}}{\big(1-\frac{w_l}{z}\big)\big(1-\frac{w_j}{z}\big)}
\end{gather*}
we obtain
\begin{gather*}
\left. \frac{E(w_l)}{\big(1-\frac{w_j}{z}\big)\big(1-\frac{w_j}{w_l}\big)}-\frac{E(w_j)}{\big(1-\frac{w_l}{z}\big) \big(1-\frac{w_j}{w_l}\big)} \right|_{w=w_j=w_l} =- \frac{\frac{w}{z}E(w)}{\big(1-\frac{w}{z}\big)^2} + \frac{w \partial_{w} E(w)}{1-\frac{w}{z}}.
\end{gather*}
Finally, we conclude that
\begin{gather}
\left. \sum_j E(w_1) \cdots E(w_{j-1}) \frac{H\big(q^{\half}w_{j}\big) - H\big(q^{-\half}w_{j}\big)}{1-\frac{w_{j}}{z}} E(w_{j+1}) \cdots E(w_k) \right|_{w= w_1 = \dots = w_k} \nonumber\\
\qquad{} = - \binom{k}{2} \frac{2\frac{w}{z} E^{k-1}(w)}{\big(1-\frac{w}{z}\big)^2}
+ \binom{k}{2} \frac{2w E^{k-2}(w) \partial_w E(w)}{1-\frac{w}{z}} - k \frac{w:\!\varphi'(w) E^{k-1}(w)\!:}{1-\frac{w}{z}}.
\label{eq:techEEHEE}
\end{gather}

Relation \eqref{eq:OPEFE^k} follows from \eqref{eq:technF(z)E(mw)} and \eqref{eq:techEEHEE}.
\end{proof}
\begin{Proposition} \label{prop:F=Tn-1}
In algebra $U(\qD)/J_{n,n_{tw}}$ holds
\begin{gather*}
F(z) = - \frac{\mu^{-n} z^{-n_{tw}}}{\big(q^{\half}-q^{-\half}\big)^2 } \exp \left( -\frac{1}{n} \varphi_-(z) \right) \T_{n-1} (z) \exp \left( - \frac{1}{n} \varphi_+(z) \right).
\end{gather*}
\end{Proposition}

\begin{proof}
Using relation $\T_n(z) - \mu^n z^{n_{tw}} \in J_{n, n_{tw}}$, we find a relation in $U(\qD)/J_{n,n_{tw}}$
\begin{gather*} F(z) E^n(w) = n! \mu^n w^{n_{tw}} F(z) :\! \exp(\phi(w))\!: \\
\hphantom{F(z) E^n(w)}{} =n! \mu^n w^{n_{tw}} \frac{(1-qw/z)\big(1-q^{-1}w/z\big)}{(1 - w/z)^2} \exp(\phi_- (w)) F(z) \exp(\phi_+ (w)).
\end{gather*}
Comparing coefficient of $(z-w)^{-2}$ with \eqref{eq:OPEFE^k}, we obtain
\begin{gather*}
n E^{n-1} (w) = n! \mu^n w^{n_{tw}} (1-q)\big(1-q^{-1}\big) \exp(\phi_-(w)) F(w) \exp( \phi_+(w) ).\tag*{\qed}
\end{gather*}\renewcommand{\qed}{}
\end{proof}

Denote by $\T^{\circ}_k (z)= \mu^{-n} z^{-n_{tw}} \T_k(z)$.

\begin{Proposition} \label{prop:Wn-1}
Following relation holds in $U(\qD)/J_{n,n_{tw}}$
\begin{gather*}
f_{n-k,n} (w/z) \T^{\circ}_{n-1}(z) \T^{\circ}_k (w) - f_{n-k,n}(z/w) \T^{\circ}_k(w) \T^{\circ}_{n-1} (z) \\
\qquad{}= -\big(q^{\half}-q^{-\half}\big)^2 \left( (n-k+1) \frac{w}{z} \delta' \left(\frac{w}{z}\right) \T^{\circ}_{k-1} (w) + w \delta\left(\frac{w}{z}\right) \partial_w \T^{\circ}_{k-1} (w) \right).
\end{gather*}
\end{Proposition}

\begin{proof}
Proposition \ref{prop:F=Tn-1} imply
\begin{gather*}
f_{n-k,n} (w/z) \T^{\circ}_{n-1}(z) \T^{\circ}_k (w)- f_{n-k,n}(z/w) \T^{\circ}_k(w) \T^{\circ}_{n-1} (z)
= -\big(q^{\half}-q^{-\half}\big)^2 \frac{\mu^{-n} w^{-n_{tw}}}{k!} \\
\qquad{}\times \exp \left(\frac{1}{n}\varphi_-(z) - \frac{k}{n} \varphi_-(w) \right)\big[F(z), E^k(w)\big] \exp \left(\frac{1}{n}\varphi_+(z) - \frac{k}{n} \varphi_+(w) \right).
\end{gather*}
It is straightforward to finish the proof using \eqref{eq:commutatorFE^k}.
\end{proof}

\begin{Proposition} \label{prop:Hom(W,qD)}
Formula \eqref{eq:ThWqD} defines a homomorphism $\WqD$ from $\WH$ to the algebra $ U(\qD)/J_{n,n_{tw}}$.
\end{Proposition}

\begin{proof}
Evidently, $H_j$ and $\T_k(z)$ commute, and $H_j$ form a Heisenberg algebra. We only have to check that $\frac{1}{\mu^k} \T_k(z)$ form $\twW$ algebra. Relation of $\twW$ algebra follows from Propositions \ref{prop:W1} and \ref{prop:Wn-1}.
\end{proof}

\begin{Proposition} \label{prop:HomDWtw}
Formulas \eqref{eq:ThqDW1}--\eqref{eq:ThqDW2} defines a homomorphism $\qDW$ from
$U(\qD)/J_{n,n_{tw}}$ to the algebra $\WH$.
\end{Proposition}

\begin{proof}
Let us check that these formulas define morphism from $\qD$. According to Proposition~\ref{relation}, it is enough to prove relations \eqref{RqDHH}--\eqref{RqDFFF}. It is done in Appendix~\ref{Appendix:HomWD}.
Evidently, $\qDW$~annihilates~$J_{n, n_{tw}}$.
\end{proof}

\begin{Proposition} \label{prop:mutinv}
Maps $\qDW$ and $\WqD$ are mutually inverse.
\end{Proposition}

\begin{proof}
Let us prove $\qDW \WqD = \id_{\WH}$ first. The algebra $\twW$ is generated by modes of $T_1(z)$. Hence it is sufficient to check $\qDW \WqD (\ha_n) = \ha_n$ and $\qDW \WqD \big( T_1(z) \big) = T_1(z)$. Both of them are straightforward.

The algebra $U(\qD)/J_{n,n_{tw}}$ is generated by modes of $E(z)$ and $F(z)$. Evidently, $\WqD \qDW \big( E(z) \big)\allowbreak = E(z)$. Proposition \ref{prop:F=Tn-1} implies $\WqD \qDW \big( F(z) \big) = F(z)$.
\end{proof}

\begin{proof}[Proof of Theorem \ref{Th:Wiso}] Follows from Propositions \ref{prop:Hom(W,qD)}, \ref{prop:HomDWtw} and \ref{prop:mutinv}
\end{proof}

\subsection[Bosonization of $\twW$]{Bosonization of $\boldsymbol{\twW}$}
Let $\sigma$ be as in \eqref{sigma}. Corollary \ref{corol: not m and m prime} states that representation $\mathcal{F}_u^{\sigma}$ actually does not depend on~$m'$ and~$m$; it is determined by $n'$ and $n$. Let us denote the representation by $\mathcal{F}_u^{(n',n)}$.

\subsubsection[Fock representation via $\qD$]{Fock representation via $\boldsymbol{\qD}$}
In this section we will discuss connection of twisted $q$-$W$ algebras and twisted representa\-tions~$\mathcal{F}_u^{(n',n)}$.
\begin{Lemma} \label{lemma:T_n for twFock}
In representation $\mathcal{F}_u^{(n',n)}$ operator $\T_n(z)$ acts by $-q^{-1/2} u z^{n'} \frac{1}{(q^{1/2} - q^{-1/2})^n}$.
\end{Lemma}

\begin{proof}We will use formula \eqref{eq:EtwFermi} to calculate $E^n(z)$.

By Proposition \ref{prop:ass&comm}
\begin{gather*}
\psia\big(q^{-1/2} z\big) \psibb \big(q^{-1/2} z\big) = - \psibb \big(q^{-1/2} z\big) \psia\big(q^{-1/2} z\big),\\
\psia \big(q^{-1/2} z\big) \psib \big(q^{1/2}z\big)= - \psib \big(q^{1/2}z\big) \psia \big(q^{-1/2} z\big)
\end{gather*}
for any $a$, $b$ (even for $a=b$).

Consider a sequence of numbers $0 \leqslant a_1, \dots, a_n \leqslant n-1$ such that $a_{i+1} - a_i \equiv -n' \mod{n}$. Thus,
\begin{gather*}
E^n(z) = n! u q^{-1/2} z^{n' + n} \psi_{(a_1)} \big(q^{-1/2} z\big) \psi^*_{(a_2)} \big(q^{1/2}z\big) \psi_{(a_2)} \big(q^{-1/2} z\big) \psi^*_{(a_3)} \big(q^{1/2}z\big)\cdots\\
\hphantom{E^n(z) =}{}\times \psi_{(a_n)} \big(q^{-1/2} z\big) \psi^*_{(a_1)} \big(q^{1/2}z\big).
\end{gather*}
Using bosonization \eqref{eq:boson psi}, \eqref{eq:boson psi*} we obtain (cf.\ \eqref{def:varphi} and \eqref{def:varphi without pm})
\begin{gather*}
E^n(z) = - n! u q^{-1/2} z^{n'} \frac{1}{\big(q^{1/2} - q^{-1/2}\big)^n} : \! \exp ( \varphi(z) ) \! :.
\end{gather*}
Consequently,
\begin{gather*}
\T_n (z) = - q^{-1/2} u z^{n'} \frac{1}{\big(q^{1/2} - q^{-1/2}\big)^n}.\tag*{\qed}
\end{gather*}\renewcommand{\qed}{}
\end{proof}
\begin{Proposition} \label{prop:tensorIDEAL}
Suppose, $M_i$ are representation of $\qD$ such that ideal $J_{\mu_i, n_i, n_i'}$ acts by zero $($for $i=1, \dots, k)$. Then $J_{\mu, n, n'}$ acts by zero on $M_1 \otimes \cdots \otimes M_k$ for $n = \sum\limits_{i=1}^k n_i$, $n'= \sum\limits_{i=1}^k n_i'$ and $\mu^{n} = \mu_1^{n_1} \cdots \mu_k^{n_k}$.
\end{Proposition}

\begin{proof}
Recall that $E^{n_i+1}(z) \in J_{n_i, n_i'}$ by Lemma \ref{lemma:T_k for big k}. Thus, $E^n(z)$ act on $M_1 \otimes \cdots \otimes M_k$ as
\begin{gather*}
\frac{n!}{n_1! \cdots n_k!} E^{n_1}(z) \otimes \cdots \otimes E^{n_k}(z) = n! \prod_{i=1}^k \mu_i^{n_i} z^{n_i'} : \! \exp(\varphi(z)) \! :.\tag*{\qed}
\end{gather*}\renewcommand{\qed}{}
\end{proof}

\begin{Proposition} \label{prop:tensor Fock IDEAL}
Ideal $J_{\mu, nd, n'd}$ acts by zero on $\mathcal{F}^{(n',n)}_{u_1} \otimes \cdots \otimes \mathcal{F}_{u_d}^{(n',n)}$ for
\begin{gather}
\mu = (-1)^{\frac{1}{n}} \frac{q^{-\frac{1}{2n}}}{q^{\half}-q^{-\half}} (u_1 \cdots u_d)^{\frac{1}{nd}} .\label{eq:mu tensor Fock}
\end{gather}
\end{Proposition}
\begin{proof}
Follows from Lemma \ref{lemma:T_n for twFock} and Proposition \ref{prop:tensorIDEAL}.
\end{proof}

\begin{Theorem} \label{Theorem:twFocktwW}
There is an action of $\WHnd$ on $\mathcal{F}^{(n',n)}_{u_1} \otimes \cdots \otimes \mathcal{F}_{u_d}^{(n',n)}$ such that action of $\qD$ factors through $\WHnd$.
\end{Theorem}
\begin{proof}
According to Proposition $\ref{prop:tensor Fock IDEAL}$, algebra $U(\qD)/J_{nd,n'd}$ acts on $\mathcal{F}^{(n',n)}_{u_1} \otimes \cdots \otimes \mathcal{F}_{u_d}^{(n',n)}$. By Theorem \ref{Th:Wiso}, algebra $U(\qD)/J_{nd,n'd}$ is isomorphic to $\WHnd$.
\end{proof}
\begin{Remark}
One can consider tensor product of Fock modules with different twists $\mathcal{F}^{\sigma_1}_{u_1} \otimes \cdots$ $\otimes \mathcal{F}_{u_d}^{\sigma_d}$. According to Proposition~\ref{prop:tensorIDEAL} algebra $\mathcal{W}\big(\mathfrak{sl}_{\sum n_i}, \sum n'_i\big) \UH$ acts on this space. Obtained representation is `irregular' (cf.~\cite{Na15}). In Section~\ref{section: construction of Whittaker} we consider an intertwiner between irregular and (graded completion of) regular representation.
\end{Remark}

\subsubsection{Explicit formula for bosonization}
Below we will write explicit formula for bosonization of $\mathcal{W}_q (\mathfrak{sl}_{nd}, n'd)$. This bosonization comes from action of $\mathcal{W}(\mathfrak{sl}_{nd}, n'd)$ on $\mathcal{F}^{(n',n)}_{u_1} \otimes \cdots \otimes \mathcal{F}_{u_d}^{(n',n)}$. Recall that realization of $\mathcal{F}_{u}^{(n',n)}$ is written via $a_b[k]$ for $b=0, \dots, n-1$.Denote by $a_b^i[k]$ (for $i=1, \dots, d$ and $b=0, \dots, n-1$) generators of Heisenberg algebra action on $i$th factor of the tensor product $\mathcal{F}^{(n',n)}_{u_1} \otimes \cdots \otimes \mathcal{F}_{u_d}^{(n',n)}$.

To write action of $\mathcal{W}_q (\mathfrak{sl}_{nd}, n'd)$ we need to introduce slightly different version of Heisenberg algebra. Namely, consider an algebra, generated by $\eta_b^{i}[k]$ for $b=0, \dots, n-1$; $i=1, \dots, d$ and $k \in \mathbb{Z}$. Relations are given by linear dependence and commutation relations
\begin{gather*}
\sum_{i=1}^{d} \sum_{b=0}^{n-1} \eta_b^{i}[k]=0, \\
\big[\eta_{b_1}^{i_1}[k_1], \eta_{b_2}^{i_2}[k_2] \big] = k_1 \delta_{k_1+k_2,0} \left( \delta_{i_1, i_2} \delta_{b_1, b_2}-\frac{1}{nd} \right).
\end{gather*}

Let us define representation $F^{\eta} \otimes \mathbb{C} \big[ \dQ \big]$ (cf.\ Section~\ref{subsection: TW qDboson}). Lattice $\dQ$ consist of elements $\sum \lambda^i_b Q^i_b$ such that $\lambda^i_b \in \mathbb{Z}$ and for any $i$ it holds $\sum_b \lambda^i_b =0$. Define an action
\begin{gather}
\eta_b^i[0] e^{\sum \lambda_b^i Q^i_b} = \lambda_b^{i} e^{\sum \lambda^i_b Q_b^i} . \label{eq:eta Fock zerom}
\end{gather}
First factor $F^{\eta}$ is a Fock space for subalgebra $\eta_b^i[k]$ for $k \neq 0$. We can consider $F^{\eta} \otimes \mathbb{C} \big[ \dQ \big]$ as representation of whole Heisenberg algebra as follows: $\eta_b^i[k]$ for $k \neq 0$ acts on first factor, $\eta_a^i[0]$ acts on the second factor by~\eqref{eq:eta Fock zerom}. Note, that also $\mathbb{C} \big[ \dQ \big]$ acts on $F^{\eta} \otimes \mathbb{C} \big[ \dQ \big]$. Let us introduce notation
\begin{gather*}
\eta_{b}^i(z) = \sum_{k \neq 0} \frac{1}{k} \eta_b^i[k] z^{-k} +Q^i_b + \eta_b^i[0] \log z.
\end{gather*}

\begin{Proposition} \label{prop:W bosonization}
There is an action of $\mathcal{W}_q (\mathfrak{sl}_{nd}, n'd)$
on $F^{\eta} \otimes \mathbb{C} \big[ \dQ \big]$ given by formulas
\begin{gather*} %\label{eq: T1 bosonisation general}
T_1 (z) \mapsto (-1)^{-\frac{1}{n}} (u_1 \cdots u_d)^{-\frac{1}{nd}}\big(q^{\half}-q^{-\half}\big) q^{\frac{1}{2n}} \\
\hphantom{T_1 (z) \mapsto}{}
\times\sum_{i=1}^d \sum_{b-a \equiv -n' \bmod n}u_i^{\frac{1}{n}} q^{\frac{a+b-n}{2n}} z^{\frac{n'-a+b}{n}+1} : \! \exp \big( \eta^i_b\big(q^{1/2} z\big) - \eta^i_a\big(q^{-1/2} z\big) \big) \! : \epsilon_{a,b}^{(i)},
\\
T_{nd-1}(z) \mapsto - (-1)^{\frac{1}{n}} (u_1 \cdots u_d)^{\frac{1}{nd}} \big(q^{\half}-q^{-\half}\big) q^{-\frac{1}{2n}} z^{n'd} \\ \hphantom{T_{nd-1}(z) \mapsto}{}\times \sum_{i=1}^d \sum_{b-a \equiv n' \bmod n} u_i^{-\frac{1}{n}} q^{ \frac{-a-b+n}{2n} } z^{\frac{-n'-a+b}{n}+1} : \! \exp \big( \eta^i_b\big(q^{-1/2} z\big) - \eta^i_a\big(q^{1/2} z\big) \big) \! : \epsilon_{a,b}^{(i)},
\end{gather*}
here $\epsilon_{a,b}^{(i)} = \prod_{r} (-1)^{\eta_r^i[0]}$ $($we consider product over such $r$ that $a-1 \geqslant r \geqslant b$ for $a>b$ and $b-1 \geqslant r \geqslant a$ for $b>a)$.
\end{Proposition}

\begin{proof}
Denote by $\big[ \mathcal{F}^{(n',n)}_{u_1} \otimes \cdots \otimes \mathcal{F}_{u_d}^{(n',n)} \big]_{h}$ subspace of such $ v \in \mathcal{F}^{(n',n)}_{u_1} \otimes \cdots \otimes \mathcal{F}_{u_d}^{(n',n)}$ that $H_k v=0$ for $k>0$. According to Theorem \ref{Theorem:twFocktwW}, the algebra $\mathcal{W}_q (\mathfrak{sl}_{nd}, n'd)$ acts on $\big[ \mathcal{F}^{(n',n)}_{u_1} \otimes \cdots \allowbreak \otimes \mathcal{F}_{u_d}^{(n',n)} \big]_{h}$.

On the other hand map $\eta_b^i[k] \mapsto a_b^i[k] - \sum_{b,i} a_b^i[k]$ is a homomorphism.
Therefore, one can identify $F^{\eta} \otimes \mathbb{C} \big[ \dQ \big]$ and $\big[ \mathcal{F}^{(n',n)}_{u_1} \otimes \cdots \otimes \mathcal{F}_{u_d}^{(n',n)} \big]_{h}$.

Substitution of \eqref{eq:th:Etw boson}, \eqref{eq:th:Ftw boson}, and \eqref{eq:mu tensor Fock} to \eqref{eqThqDWE}, \eqref{eq:ThqDW2} finishes the proof.
\end{proof}
Denote obtained representation by $\dWFock$.
\begin{Remark} \label{remark mu abuse}
The parameter $\mu$ is determined by $u_1,\dots, u_d$ only up to $nd$-th root of unity. The modules $\dWFock$ with different $\mu$ are non-isomorphic in general (so notation $\dWFock$ is ambiguous). For example, one can see this from the highest weights $\lambda_s$ defined in the next section.

The modules $\dWFock$ with different $\mu$ are related by an external automorphism of $\twW$.
\end{Remark}

\begin{Example}Let us consider case of twisted Virasoro algebra, i.e., $n=2$, $n'=1$, $d=1$. Then everything is expressed via one boson $\eta(z)$ with relation $[\eta[k_1], \eta[k_2]] = \frac12 k_1 \delta_{k_1+k_2,0}$ and $[\eta[0], Q]=\frac12$. So there is an action of $\mathcal{W}_{q}(\mathfrak{sl}_2, 1)$ on $F^{\eta} \otimes \mathbb{C}[ \mathbb{Z}]$ given by
\begin{gather*}
T_1 (z) \mapsto (-1)^{-\frac12} \big(q^{\half}-q^{-\half}\big) (-1)^{\eta[0]} \\
\hphantom{T_1 (z) \mapsto}{}\times \big[ z : \! \exp \big( \eta\big(q^{1/2} z\big) + \eta\big(q^{-1/2} z\big) \big) \! : + z^{2} : \! \exp \big( {-}\eta(q^{1/2} z\big) - \eta\big(q^{-1/2} z\big) \big) \! : \big].
\end{gather*}
We can simplify the formula using conjugation by $(-1)^{\eta[0]^2/2}$
\begin{gather*}
T_1 (z) \mapsto \big(q^{\half}-q^{-\half}\big) \big[ z :\!\exp \big( \eta(q^{1/2} z\big) + \eta\big(q^{-1/2} z\big) \big)\!: + z^{2}:\! \exp \big( {-}\eta\big(q^{1/2} z\big) - \eta\big(q^{-1/2} z\big) \big)\!: \big].
\end{gather*}
\end{Example}

\subsubsection{Explicit formulas for strange bosonization}
To write formulas for strange bosonization we need to consider Heisenberg algebra generated by $\xi_i[k]$ for $i=1, \dots, d$ and $k \in \mathbb{Z}$. Relations are given by linear dependence and commutation relations
\begin{gather*}
\sum_{i=1}^{d} \xi_{i}[nk] =0,\\
[\xi_{i_1}[k_1], \xi_{i_2}[k_2] ] = k_1 \delta_{k_1+k_2,0} \delta_{i_1, i_2} \qquad \text{for either $n \nmid k_1$ or $n \nmid k_2$}, \\
[\xi_{i_1}[nj_1], \xi_{i_2}[nj_2] ] = nj_1 \delta_{j_1+j_2,0} \left( \delta_{i_1, i_2}-\frac{1}{d} \right).
\end{gather*}
Denote corresponding Fock module by $F^{\xi}$.
\begin{Proposition}
There is an action of $\mathcal{W}_q (\mathfrak{sl}_{nd}, n'd)$ on $F^{\xi}$ given by
\begin{gather*}
T_1 (z) \mapsto -(-1)^{-\frac{1}{n}} \frac{ q^{\half}-q^{-\half} }{n\big(q^{\frac{1}{2n}}-q^{-\frac{1}{2n}}\big)} (u_1 \cdots u_d)^{-\frac{1}{nd}} z^{\frac{n'}{n}} \\
\hphantom{T_1 (z) \mapsto}{}\times \sum_{i=1}^d u_i^{1/n} \sum_{l=0}^{n-1} \zeta^{l n'} : \! \exp \left( \sum_{k} \frac{q^{-k/2n}-q^{k/2n}}{k} \xi_{i}[k] \zeta^{-kl} z^{-k/n} \right) \! :,\\
T_{nd-1}(z) \mapsto -(-1)^{\frac{1}{n}} \frac{q^{\half}-q^{-\half}}{n\big(q^{\frac{1}{2n}}-q^{-\frac{1}{2n}}\big)} (u_1 \cdots u_d)^{\frac{1}{nd}} z^{\frac{nd-1}{n}n'} \\
\hphantom{T_{nd-1}(z) \mapsto}{}\times \sum_{i=1}^{d} u_i^{-1/n} \sum_{l=0}^{n-1} \zeta^{-l n'} : \! \exp \left( \sum_{ k}\frac{q^{k/2n}-q^{-k/2n}}{k} \xi_{i}[k] \zeta^{-kl} z^{-k/n} \right) \! :.
\end{gather*}
Obtained representation is isomorphic to $\dWFock$.
\end{Proposition}
\begin{proof}
The proof is analogous to proof of Proposition \ref{prop:W bosonization}. The only difference is that we have to use \eqref{eq:Th Strange boson E}, \eqref{eq:Th Strange boson F} instead of \eqref{eq:th:Etw boson}, \eqref{eq:th:Ftw boson}.
This representation is isomorphic to $\dWFock$ since it also corresponds to $\mathcal{F}^{(n',n)}_{u_1} \otimes \cdots \otimes \mathcal{F}_{u_d}^{(n',n)}$.
\end{proof}
\subsection{Verma modules vs Fock modules} \label{Connection with Verma module}

Connection of Fock module and Verma module is known in non-twisted case. In this Subsection we will generalize it for $\twW$. Denote $d= \gcd(n_{tw},n)$.

\begin{Definition}
$\twWgl = U(\qD)/J_{n,n_{tw}}$.
\end{Definition}

Denote modes of $E^k(z)$ by $E^k[j]$. Namely, $E^k(z) = \sum_d E^k[j] z^{-j}$.

\begin{Definition} \label{definition: Verma}
Verma module $\Vergl$ is a module over $\twWgl$ with cyclic vector~$\vacgl$ and relations
\begin{gather*}
H_k \vacgl =0 \qquad \text{for $k>0$}, \\
E^k[j] \vacgl =0 \qquad \text{for $j+\frac{kn_{tw}}{n}>0$}, \\
E^{\frac{ns}{d}}[-sn_{tw}/d] \vacgl = \fact \lambda_s \vacgl .
\end{gather*}
\end{Definition}

Consider a grading on $\twWgl$ given by $\deg E_k[j] \!=\! -j-\frac{k n_{tw}}{n}$. Verma module $\Vergl$ is a graded module with grading defined by requirement $\deg \vacgl = 0$.

\begin{Proposition} \label{prop: cyclic Verma}
$\Vergl$ is spanned by $E^{k_1}[j_1] \cdots E^{k_t}[j_t] \vacgl$ for $\frac{j_1}{k_1} \leqslant \frac{j_2}{k_2} \leqslant \dots \leqslant \frac{j_t}{k_t} <\!{-}\frac{n_{tw}}{n}$ and $1 \leqslant k_i \leqslant n$.
\end{Proposition}

\begin{proof}Recall that $E^n(z) = n! z^{n_{tw}} : \! \exp(\varphi(z)) \! :$. Therefore $E^n[j]$ acts on $F^H$.
\begin{Lemma} \label{lemma: Heis and En}Module $F^H$ is spanned by $E^n[-j_1] \cdots E^n[-j_k] \vac$.
\end{Lemma}
\begin{proof}[Sketch of the proof]
Let us consider operator $E_{-}^n[-j]$ defined by $\sum_j E_{-}^n [-j] z^j = \exp(\varphi_- (z))$. On one hand $E^n_-[-j_1] \cdots E^n_-[-j_k] \vac$ is a basis of~$F^H$ (this basis coincide with a basis of complete homogeneous polynomials up to renormalization of Heisenberg algebra generators). One the other hand,
\begin{gather*}
E^n[-j_1] \cdots E^n[-j_k] \vac = E^n_-[-j_1] \cdots E^n_-[-j_k] \vac + \text{lower terms},
\end{gather*}
here lower terms are taken with respect to lexicographical order.
\end{proof}

\begin{Remark}
Lemma \ref{lemma: Heis and En} holds for any exponent $: \! \exp \big( \sum \alpha_j H_j z^{-j} \big) \! :$ such that $\alpha_{-i} \neq 0$ for $i>0$. The proof does not use any other properties of coefficients $\alpha_j$. One can find the proof as the last part of proof of \cite[Proposition~2.29]{N16}.
\end{Remark}

Define $\qD^{\ \geqslant 0}$ and $\qD^{\ >0}$ as subalgebras of $\qD$ spanned by $E_{k,j}$ with $k \geqslant 0$ and $k>0$ correspondingly.

\begin{Lemma} \label{lemma: q-dif >0 Verma}
Vector $\vacgl \in \Vergl$ is cyclic with respect to action of $\qD^{\, > 0}$.
\end{Lemma}
\begin{proof}
Theorem \ref{Th:Wiso} implies that the natural map $U\big(\qD^{\ \geqslant 0}\big) \rightarrow U(\qD)/J_{n,n_{tw}}$ is surjective. Hence Verma module is generated by non-commutative monomials in $E_i$ and $H_j$ applied to $\vacgl$. Using relation $[H_j, E_i] = \big(q^{-j/2}-q^{j/2}\big) E_{i+j} $, we see that the module is spaned by $E_{i_1} \cdots E_{i_k} H_{j_1} \cdots H_{j_l} \vacgl$. Lemma~\ref{lemma: Heis and En} implies that the module is spanned by $E_{i_1} \cdots E_{i_k} E^n[j_1]\allowbreak \cdots E^n[j_l] \vacgl$
\end{proof}

In \cite[(3.48)]{N17} author states, that algebra $\qD^{\ > 0}$ has a PBW-like basis $E^{k_1}[j_1] \cdots E^{k_t}[j_t]$ for $\frac{j_1}{k_1} \leqslant \frac{j_2}{k_2} \leqslant \cdots \leqslant \frac{j_t}{k_t} $. Thus $\Vergl$ is spanned by $E^{k_1}[j_1] \cdots E^{k_t}[j_t] \vacgl$ (with the same condition on~$k_i$ and~$j_i$).

Note that if $\frac{j_t}{k_t}> -\frac{n_{tw}}{n}$ then such vector is 0; moreover if $\frac{j_t}{k_t}= -\frac{n_{tw}}{n}$ then $E^{k_t}[j_t]$ acts by multiplication on a constant, hence it can be excluded.Also note that $E^k[j] \in J_{n, n_{tw}}$ for $k \geqslant n+1$, therefore we assume $k_i \leqslant n$. The Proposition~\ref{prop: cyclic Verma} is proven.
\end{proof}

\begin{Corollary} \label{Prop:Characterless}
Coefficients of $\ch \Vergl$ are less or equal to coefficients of $\prod\limits_{k=1}^{\infty} \big(1- \mathfrak{q}^{\frac{kd}{n}} \big)^{-d} $.
\end{Corollary}

\begin{Theorem} \label{Th: Verma vs Fock}
If $\mathcal{F}^{(n_{tw}/d, n/d)}_{u_1} \otimes \dots \otimes \mathcal{F}^{(n_{tw}/d, n/d)}_{u_d}$ is irreducible then natural map $p \colon \Vergl \allowbreak \rightarrow \mathcal{F}^{(n_{tw}/d, n/d)}_{u_1} \otimes \dots \otimes \mathcal{F}^{(n_{tw}/d, n/d)}_{u_d}$ is an isomorphism for
\begin{gather} \label{eq: lambdas}
\lambda_s = \big( {-} q^{-\frac12} \big)^s \frac{1}{\big(q^{\frac12} - q^{-\frac12} \big)^{\frac{ns}{d}}} e_s (u_1, \dots, u_d).
\end{gather}
\end{Theorem}

\begin{proof}
Let us first prove existence of the map $p$. This will follow automatically from the following lemma.
\begin{Lemma}
The highest vector $\vacF \in \mathcal{F}^{(n_{tw}/d, n/d)}_{u_1} \otimes \dots \otimes \mathcal{F}^{(n_{tw}/d, n/d)}_{u_d}$ satisfy following conditions
\begin{gather}
H_k \vacF=0 \qquad \text{for $k>0$}, \label{eq: Lemma 7.18 7.69}\\
E^k[j] \vacF =0 \qquad \text{for $j+\frac{kn_{tw}}{n}>0$}, \label{eq: Lemma 7.18 7.70} \\
E^{\frac{ns}{d}}[-sn_{tw}/d] \vacF = \fact \lambda_s \vacF. \label{eq: Lemma 7.18 7.71}
\end{gather}
\end{Lemma}
\begin{proof}
Assertions \eqref{eq: Lemma 7.18 7.69}--\eqref{eq: Lemma 7.18 7.70} are evident. Let us check \eqref{eq: Lemma 7.18 7.71}. Denote by $\e_j(z)$ action of~$E(z)$ on $j$th tensor multiple; in particular $E(z) \mapsto \e_1(z) + \dots + \e_d(z)$. Recall that $\e_j(z)^{\frac{n}{d}+1}=0$. Thus
\begin{gather*}
E^{\frac{ns}{d}} (z) \mapsto \frac{\fact}{(\frac{n}{d}!)^{s}} \sum_{k_1< k_2 < \dots < k_j} \e_{k_1}^{\frac{n}{d}} (z) \e_{k_2}^{\frac{n}{d}} (z) \cdots \e_{k_s}^{\frac{n}{d}} (z) + \cdots,
\end{gather*}
here dots denote summands with a power which is not a multiple of $\frac{n}{d}$ (thus, this summands do not contribute to $E^{\frac{ns}{d}}[-sn_{tw}/d] \vacF$). To finish the proof we note that
\begin{gather*}
\frac{\fact}{(\frac{n}{d}!)^{s}} \langle \bar{u} | \e_{k_1}^{\frac{n}{d}} (z) \e_{k_2}^{\frac{n}{d}} (z) \cdots \e_{k_s}^{\frac{n}{d}} (z) \vacF = \fact u_{k_1} u_{k_2} \cdots u_{k_s} \big( {-} q^{-1/2} \big)^s z^{ \frac{n_{tw}}{d}s } \frac{1}{(q^{1/2} - q^{-1/2})^{\frac{ns}{d}}} .\!\!\!\!\tag*{\qed}
\end{gather*}\renewcommand{\qed}{}
\end{proof}

Image of $p$ is whole $\mathcal{F}^{(n_{tw}/d, n/d)}_{u_1} \otimes \dots \otimes \mathcal{F}^{(n_{tw}/d, n/d)}_{u_d}$, since $\mathcal{F}^{(n_{tw}/d, n/d)}_{u_1} \otimes \dots \otimes \mathcal{F}^{(n_{tw}/d, n/d)}_{u_d}$ is irreducible. Note that $\ch \big( \mathcal{F}^{(n_{tw}/d, n/d)}_{u_1} \otimes \dots \otimes \mathcal{F}^{(n_{tw}/d, n/d)}_{u_d} \big) = \prod\limits_{k=1}^{\infty} \frac{1}{ \big(1-q^{\frac{kd}{n}} \big)^d }$. Corollary~\ref{Prop:Characterless} implies that $p$ is injective and $\ch \Vergl = \prod\limits_{k=1}^{\infty} \frac{1}{ \big(1-q^{\frac{kd}{n}} \big)^d }$.
\end{proof}
\begin{Remark}
Note that representation $\WFock$ is irreducible iff $\mathcal{F}_{u_1} \otimes \dots \otimes \mathcal{F}_{u_d}$ is irreducible.
In particular, for $d =1$ then $\mathcal{F}^{\twW}_u$ is irreducible automatically. Generally, criterion of irreducibility of $\mathcal{F}_{u_1} \otimes \dots \otimes \mathcal{F}_{u_d}$ is given by Lemma~\ref{lemma 3.1 FJMM}.
\end{Remark}

\begin{Definition}
Verma module $\Ver$ is a module over $\twW$ with cyclic vector~$\vacsl$ and relations
\begin{gather*}
\tT_k[r] \vacsl =0 \qquad \text{for $r>0$}, \\
\tT_{ns/d}[0] \vacsl = \lambda_s \vacsl.
\end{gather*}
\end{Definition}

Introduce grading on $\twW$ by $\deg T^{tw}_k[r] = -r$. Verma module $\Vergl$ is a graded module with grading defined by $\deg \vacsl = 0$.

To simplify notation for comparison of $\mathfrak{sl}_n$ and $\mathfrak{gl}_n$ cases, we will assume below that $\mu=1$ (cf.\ Remark~\ref{remark mu abuse}).
\begin{Proposition}
$\Vergl \cong \Ver \otimes F^H$ with respect to identification $\twWgl = \twW \UH$.
\end{Proposition}

\begin{proof}
The existence of maps in both directions can be checked directly using universal property of the Verma module.
Evidently, these maps are mutually inverse.
\end{proof}

\begin{Corollary}If $\WFock$ is irreducible then natural map $\tilde{p} \colon \Ver \rightarrow \WFock$ is an isomorphism for $\lambda_s$ as in~\eqref{eq: lambdas}.
\end{Corollary}

\section[Restriction on $\qD^{\Lambda}$ for general sublattice]{Restriction on $\boldsymbol{\qD^{\Lambda}}$ for general sublattice} \label{Section: general sublattice}
We generalize results of Section \ref{S:sublattice} for arbitrary sublattice. For applications in Section~\ref{section: conformal} we will need only case of sublattice $\Lambda_0 = \Span (e_1, ne_2 ) \subset \mathbb{Z}^2$.

\subsection{Decomposition of restriction}
Let $\Lambda \subset \mathbb{Z}^2$ be a sublattice of finite index. Let us choose basis $w_1$, $w_2$ of lattice $\Lambda$ so that $w_1 = (r,n_{tw})$ and $w_2 = (0, n)$. Let $d$ be the greatest common divisor of $n$ and $n_{tw}$.
\begin{Theorem} \label{th:general Lambda}
There is an isomorphism of $\qD$-modules
\begin{gather} \label{eq:iso from Lattice Th}
\left. \mathcal{F}^{[1/nr]}_{u^{d/nr}} \right|_{\phi_{w_1,w_2}(\qD)} \cong \bigoplus_{ l \in \mathbf{Q}_{(d)}} \ \mathcal{F}_{u q^{ rl_0}}^{(n_{tw}/d, \, n/d)} \otimes \cdots \otimes \mathcal{F}_{u q^{r(\frac{\alpha}{d}+ l_{\alpha})
}}^{(n_{tw}/d, \, n/d)} \otimes \cdots \otimes \mathcal{F}_{u q^{r(\frac{d-1}{d}+ l_{d-1})}}^{(n_{tw}/d, \, n/d)}.
\end{gather}
\end{Theorem}
\begin{proof}Proposition \ref{E_l} implies that $\left. \mathcal{F}^{[1/nr]}_{u^{d/nr}} \right|_{\phi_{re_1,e_2}} \cong \mathcal{F}^{[1/n]}_{u^{d/n}}$. Hence it is enough to consider case $r=1$. We will use realization of $\left. \mathcal{F}^{[1/n]}_{u^{d/n}} \right|_{\phi_{w_1,w_2}(\qD)}$ constructed in Proposition~\ref{Prop:qD Boson}. Strategy of our proof is as follows; first we will construct decomposition on the level of vector spaces and then study action on each direct summand.

For each $\alpha = 0, \dots, d-1$ let $\mathbf{Q}_{(n/d)}^{\alpha, l_{\alpha}}$ be a subset of lattice $\P$ consisting of elements
\[
\sum_{a \equiv \alpha \bmod{d}} \tilde{l}_a Q_a \qquad \text{such that} \quad \sum_{a \equiv \alpha \bmod{d}} \tilde{l}_a = l_{\alpha}.
\]
Note that
\[
\Q = \coprod_{l_0 + \dots + l_{d-1} = 0} \mathbf{Q}_{(n/d)}^{0, l_0} \oplus \dots \oplus \mathbf{Q}_{(n/d)}^{d-1, l_{d-1}}.
\]
Or equivalently
\[
\mathbb{C}[\Q] = \bigoplus_{l_0 + \dots + l_{d-1} = 0} \mathbb{C}\big [\mathbf{Q}_{(n/d)}^{0, l_0} \big] \otimes \dots \otimes \mathbb{C}\big [ \mathbf{Q}_{(n/d)}^{d-1, l_{d-1}} \big].
\]
Let $F^{\frac{n}{d}a, \alpha}$ be Fock module for Heisenberg algebra generated by $a_b[k]$ for $b \equiv \alpha \bmod{d}$. Then
\begin{gather}
F^{na} \otimes \mathbb{C}[\Q] = \bigoplus_{l_0 + \dots + l_{d-1} = 0} \big( F^{\frac{n}{d}a, 0} \otimes \mathbb{C}\big[\mathbf{Q}_{(n/d)}^{0, l_0} \big] \big) \otimes \dots \otimes \big( F^{\frac{n}{d}a, d-1} \otimes \mathbb{C}\big[ \mathbf{Q}_{(n/d)}^{d-1, l_{d-1}} \big] \big).\!\!\!\! \label{eq: decomp vect space}
\end{gather}
Let us show that \eqref{eq: decomp vect space} is a decomposition of $\qD$-modules (moreover, that it leads to decomposition~\eqref{eq:iso from Lattice Th}).Let us define
\begin{gather}
H_{\alpha}^{tw}[k] = \sum_{b \equiv \alpha \bmod d} a_b[k], \label{eq: H alpha tw} \\
E^{tw}_{\alpha}(z) = \sum_{\substack{b-a \equiv -n_{tw} \bmod n \\ a \equiv \alpha \bmod d}} u^{\frac{d}{n}} q^{\frac{a+b-n}{2n}} z^{\frac{n_{tw}-a+b}{n}+1} : \! \exp \big( \phi_b\big(q^{1/2} z\big) - \phi_a\big(q^{-1/2} z\big) \big) \! : \epsilon_{a,b},
\\
F_{\alpha}^{tw}(z) = \sum_{\substack{b-a \equiv n_{tw} \bmod n\\ a \equiv \alpha \bmod d}} u^{-\frac{d}{n}} q^{ \frac{-a-b+n}{2n} } z^{\frac{-n_{tw}-a+b}{n}+1} : \! \exp \big( \phi_b\big(q^{-1/2} z\big) - \phi_a\big(q^{1/2} z\big) \big) \! : \epsilon_{a,b}, \label{eq: F alpha tw}
\end{gather}
here $\epsilon_{a,b} = \prod_r (-1)^{a_r[0]}$ (product over such $r$ that $a-1 \geqslant r \geqslant b$ for $a>b$ and $b-1 \geqslant r \geqslant a$ for $b>a$).
\begin{Lemma}
Formulas \eqref{eq: H alpha tw}--\eqref{eq: F alpha tw} defines an action of $\qD$ on $F^{\frac{n}{d}a, \alpha} \otimes \mathbb{C}\big[\mathbf{Q}_{(n/d)}^{\alpha, l_{\alpha}}\big]$; obtained representation is isomorphic to $\mathcal{F}_{u q^{ \frac{\alpha}{d} + l_{\alpha}}}^{(n_{tw}/d, n/d)}$.
\end{Lemma}
\begin{proof}[Sketch of the proof.]
Let us define $\tilde{\epsilon}_{a,b}=\prod_r (-1)^{a_r[0]}$ (product over $r$ satisfying above inequalities and condition $r \equiv \alpha \bmod d$. One can check that there exists an index set $I$ such that conjugation of $E^{tw}_{\alpha}(z)$ and $F^{tw}_{\alpha}(z)$ by $\prod\limits_{(i,j) \in I} (-1)^{a_i[0] a_j[0]}$ will turn $\epsilon_{a,b}$ to $\tilde{\epsilon}_{a,b}$. Theorem \ref{Th:qD Boson} finishes the proof.
\end{proof}

On the other hand, formulas \eqref{eq1: Prop:qD Boson}--\eqref{eq3: Prop:qD Boson} implies
\begin{gather*}
H^{tw}[k] = \sum_{\alpha} H_{\alpha}^{tw}[k], \qquad
E^{tw}(z) = \sum_{\alpha} E^{tw}_{\alpha} (z), \qquad F^{tw}(z) = \sum_{\alpha} F^{tw}_{\alpha} (z).
\end{gather*}
Therefore embedding of vector space from first row of following commutative diagram leads to second row embedding of $\qD$-modules
\begin{equation*}
\begin{tikzcd}
\big( F^{\frac{n}{d}a, 0} \otimes \mathbb{C} \big[\mathbf{Q}_{(n/d)}^{0, l_0} \big] \big) \otimes \dots \otimes \big( F^{\frac{n}{d}a, d-1} \otimes \mathbb{C} \big[ \mathbf{Q}_{(n/d)}^{d-1, l_{d-1}} \big] \big) \arrow[d, equal] \arrow[r,hook] &
F^{na}\otimes \mathbb{C} [\Q] \arrow[d, equal]
\\
\mathcal{F}_{u q^{ l_0}}^{(n_{tw}/d, n/d)} \otimes \cdots \otimes \mathcal{F}_{u q^{\frac{\alpha}{d}+ l_{\alpha}
}}^{(n_{tw}/d, n/d)} \otimes \cdots \otimes \mathcal{F}_{u q^{\frac{d-1}{d}+ l_{d-1}}}^{(n_{tw}/d, n/d)} \arrow[r,hook] &
\left. \mathcal{F}^{[1/n]}_{u^{d/n}} \right|_{\phi_{w_1,w_2}(\qD)}.
\end{tikzcd}\tag*{\qed}
\end{equation*}\renewcommand{\qed}{}
\end{proof}

\begin{Corollary}Following $\qD$-modules are isomorphic
\begin{gather}
\left. \mathcal{F}^{[1/n]}_{u} \right|_{\phi_{e_1,ne_2}(\qD)} \cong \bigoplus_{ l \in \mathbf{Q}_{(n)}} \ \mathcal{F}_{u q^{ l_0}} \otimes \cdots \otimes \mathcal{F}_{u q^{\frac{\alpha}{n}+ l_{\alpha}
}} \otimes \cdots \otimes \mathcal{F}_{u q^{\frac{n-1}{n}+ l_{n-1}}}. \label{eq:th:decomposition}
\end{gather}
\end{Corollary}

\begin{Remark}
Lattice $\Lambda$ admits another basis $v_1=(N,0)$, $v_2 = (R,d)$. There is an isomorphism
\begin{gather*}
\left. \mathcal{F}^{[1/nr]}_{u^{d/nr}} \right|_{\phi_{v_1,v_2}(\qD)} \cong \bigoplus_{ l \in \mathbf{Q}_{(d)}} \ \mathcal{F}_{u q^{ rl_0}} \otimes \cdots \otimes \mathcal{F}_{u q^{r\left(\frac{\alpha}{n}+ l_{\alpha} \right)
}} \otimes \cdots \otimes \mathcal{F}_{u q^{r \big(\frac{d-1}{d}+ l_{d-1} \big)}}.
\end{gather*}
\end{Remark}

\subsection{Strange bosonization and odd bosonization}
Representation $\left. \mathcal{F}_u^{[1/n]} \right|_{\phi_{w_1, w_2} \qD}$ admits fermionic, bosonic, and strange bosonic realizations; formulas are given in Propositions \ref{TheFermionicProp}, \ref{Prop:qD Boson} and \ref{Prop:StrangeBoson} correspondingly. This Section is devoted to study of corresponding $\twW$ algebra action on $\left. \mathcal{F}_u^{[1/n]} \right|_{\phi_{w_1, w_2} \qD}$. We will consider strange bosonization and its classical limit.

Let us introduce following notation (cf.\ with \eqref{eq:E})
\begin{gather*}
\se(z) = {: \! \exp \left( \sum_{k \ne 0} \frac{q^{-k/2n}-q^{k/2n}}{k} a_{k} z^{-k} \right) \! :}.
\end{gather*}

\begin{Proposition} \label{prop: mu for restriction}
For $w_1 = e_1+n_{tw}e_2$, $w_2 = ne_2$ ideal $J_{\mu, n, n_{tw}}$ acts by zero on $\left. \mathcal{F}^{[1/n]}_{u} \right|_{\phi_{w_1,w_2}(\qD)}$ for
\begin{gather} \label{eq:prop: mu for restriction}
\mu = \frac{u}{n\big(1-q^{\frac1n}\big)} \times (-1)^{\frac{d}{n}} \frac{q^{-\frac{1}{2n}}-q^{\frac{1}{2n}}}{q^{\half}-q^{-\half}} n = (-1)^{\frac{d}{n}} \frac{q^{-\frac{1}{2n}} u}{\big(q^{\half}-q^{-\half}\big)}.
\end{gather}
\end{Proposition}
\begin{proof}
We will use strange bosonization of $\left. \mathcal{F}^{[1/n]}_{u} \right|_{\phi_{w_1,w_2}(\qD)}$. Recall that we denote by $\zeta$ a~primitive root of unity of degree $n$. We have
\begin{gather*}
E(z) \mapsto \frac{u}{n\big(1-q^{\frac1n}\big)} z^{\frac{n_{tw}}{n}} \big(\se\big(z^{1/n}\big)+\zeta^{n_{tw}} \se\big(\zeta z^{1/n}\big) + \cdots + \zeta^{(n-1)n_{tw}} \se\big(\zeta^{n-1}z^{1/n}\big) \big).
\end{gather*}
Let us calculate
\begin{gather*}
 \left( \frac{n\big(1-q^{\frac1n}\big)}{u} E(z)\right)^n = z^{n_{tw}}\big( \se\big(z^{1/n}\big)+\zeta^{n_{tw}} \se\big(\zeta z^{1/n}\big) + \cdots + \zeta^{(n-1)n_{tw}} \se\big(\zeta^{n-1}z^{1/n}\big)\big)^n \\
{} = n! z^{n_{tw}} (-1)^{(n-1)n_{tw}} \se\big(z^{1/n}\big) \se\big(\zeta z^{1/n}\big) \cdots \se\big(\zeta^{n-1} z^{1/n}\big) \\
{} = (-1)^{(n-1)n_{tw}}
n! z^{n_{tw}} \prod_{i<j} \frac{\big( 1 - \zeta^{j-i} \big)^2}{\big(\big( 1- q^{\frac{1}{n}}\zeta^{j-i} \big)\big(1 - q^{-\frac{1}{n}} \zeta^{j-i}\big)\big)} {:\!\se(z^{1/n}) \se\big(\zeta z^{1/n}\big) \cdots \se\big(\zeta^{n-1} z^{1/n}\big)\!:}.
\end{gather*}
We need to compute
\begin{gather*}
\prod_{i<j} \frac{\big( 1 - \zeta^{j-i} \big)^2}{ \big( 1- q^{\frac{1}{n}}\zeta^{j-i} \big)\big(1 - q^{-\frac{1}{n}} \zeta^{j-i} \big)} = \prod_{i < j} \frac{\big(1-\zeta^{j-i}\big)\big(1-\zeta^{i-j}\big)}{q^{-\frac{1}{n}}\big(1-q^{\frac{1}{n}}\zeta^{j-i}\big) \big(1-q^{\frac{1}{n}}\zeta^{i-j}\big)}\\
\qquad{} = q^{\frac{1}{n} \binom{n}{2}}\prod_{k\neq 0} \frac{\big( 1 - \zeta^{k} \big)^{n}}{\big( 1- q^{\frac{1}{n}}\zeta^{k} \big)^{n} } = q^{\frac{n-1}{2}} \frac{\big(1-q^{\frac{1}{n}}\big)^n }{(1-q)^n } n^{n} = \frac{\big(q^{\frac{1}{2n}}-q^{-\frac{1}{2n}}\big)^n}{\big(q^{\half}-q^{-\half}\big)^n} n^n.
\end{gather*}
So
\begin{gather*}
\left( \frac{n\big(1-q^{\frac1n}\big)}{u} E(z)\right)^n = n! z^{n_{tw}} (-1)^{(n-1)n_{tw}} \\
\hphantom{\left( \frac{n\big(1-q^{\frac1n}\big)}{u} E(z)\right)^n =}{}
\times \left( \frac{\big(q^{\frac{1}{2n}}-q^{-\frac{1}{2n}}\big)}{\big(q^{\half}-q^{-\half}\big)} n \right)^n :\!\se\big(z^{1/n}\big) \se\big(\zeta z^{1/n}\big) \cdots \se\big(\zeta^{n-1} z^{1/n}\big)\!:.
\end{gather*}
Note that $:\!\se\big(z^{1/n}\big) \se\big(\zeta z^{1/n}\big) \cdots \se\big(\zeta^{n-1} z^{1/n}\big)\!: = :\!\exp{\varphi(z)}\!:$. Hence
\begin{gather*}
\mu^n = (-1)^{(n-1)n_{tw}} \left( \frac{u}{n\big(1-q^{\frac1n}\big)} \frac{\big(q^{\frac{1}{2n}}-q^{-\frac{1}{2n}}\big)}{\big(q^{\half}-q^{-\half}\big)} n \right)^n = (-1)^{(n-1)n_{tw}+n} \left( \frac{u q^{-\frac{1}{2n}}}{q^{\frac12}-q^{-\frac12}} \right)^n.
\end{gather*}
Finally note that $(n-1)n_{tw}+n \equiv d \mod{2}$.
\end{proof}

\begin{Remark}
Another way to prove Proposition \ref{prop: mu for restriction} is to derive it from Proposition \ref{prop:tensor Fock IDEAL} (since isomorphism \eqref{eq:iso from Lattice Th}). Beware inconsistency of our notation in \eqref{eq:mu tensor Fock} and \eqref{eq:prop: mu for restriction}. Let us rewrite~\eqref{eq:mu tensor Fock}
\begin{gather*}
\mu = (-1)^{\frac{d}{n}} \frac{q^{-\frac{d}{2n}}}{q^{\frac12}-q^{-\frac12}} (u_1 \cdots u_d)^{\frac1n} = (-1)^{\frac{d}{n}} \frac{q^{-\frac{d}{2n}}}{q^{\frac12}-q^{-\frac12}} \big( \big( u^{\frac{n}{d}}\big)^d q^{\frac{d-1}{2}}\big)^{\frac1n} = (-1)^{\frac{d}{n}} \frac{q^{-\frac{1}{2n}}}{q^{\frac12}-q^{-\frac12}} u.
\end{gather*}
\end{Remark}

Let us consider subalgebra of Heisenberg algebra generated by $J_k = a_k$ for $n \nmid k$. Denote corresponding Fock module by $F^J$.

\begin{Corollary}
There is an action of $\twW$ on $F^{J}$ given by
\begin{gather*}
T_1 (z) = -(-1)^{\frac{d}{n}} \frac{1}{n} \frac{q^{\half}-q^{-\half}}{q^{ \frac{1}{2n}} - q^{-\frac{1}{2n}}} z^{\frac{n_{tw}}{n}} \sum_{l=0}^{n-1} \zeta^{l n_{tw}} : \! \exp \left( \sum_{n \nmid k} \frac{q^{-\frac{k}{2n}}-q^{\frac{k}{2n}}}{k} J_k \zeta^{-kl} z^{-\frac{k}{n}} \right) \! \!:.
\end{gather*}
Obtained representation corresponds to $\left. \mathcal{F}^{[1/n]}_{u} \right|_{\phi_{w_1,w_2}(\qD)}$.
\end{Corollary}
\begin{proof}
Follows from Theorem \ref{Th:Wiso} and Proposition \ref{prop: mu for restriction}
\end{proof}
\begin{Remark}
Let us consider non-twisted case $n_{tw}=0$. Then $d=n$ and total sign $-(-1)^{\frac{d}{n}}=1$. More accurately, the coefficient is a root of unity of degree $n$ (cf.\ Remark~\ref{remark mu abuse}). Nevertheless, this freedom disappears if we require $T_1(z) = n + o(\hbar)$ for $\hbar = \log q $ (this is a standard setting for classical limit).
\end{Remark}

\begin{Example}Odd bosonization is a particular case of strange bosonization for $n=2$ and $n_{tw}=0$,
\begin{gather}
T_1 (z) = \frac{q^{\frac{1}{4}}+q^{-\frac{1}{4}}}{2} \left[ : \! \exp\! \left(\! \sum_{2 \nmid r} \frac{q^{-\frac{r}{4}}-q^{\frac{r}{4}}}{r} J_r z^{-\frac{r}{2}}\! \right) \! : + : \! \exp \!\left(\! -\sum_{2 \nmid r} \frac{q^{-\frac{r}{4}}-q^{\frac{r}{4}}}{r} J_r z^{-\frac{r}{2}} \!\right) \! : \right]\!.\!\!\! \label{eq:quantum Virasoro}
\end{gather}
Consider classical limit $q \rightarrow 1$. It is convenient to assume $q=e^{\hbar}$ and $\hbar \rightarrow 0$. If there exists an expansion
\begin{gather*}
T_1(z)= 2 + z^2 L(z) \hbar^2 + o\big(\hbar^2\big),
\end{gather*}
then modes of current $L(z) = L_n z^{-n-2}$ form `not $q$-deformed' Virasoro algebra. Note that
\begin{gather*}
T_1(z) \rightarrow 2 + z^2 \left(\frac{z^{-2}}{16} + \frac{1}{4} \sum :\!J_{\text{odd}}(z)^2\!\!: \right) \hbar^2 + o\big(\hbar^2\big),
\end{gather*}
where $J_{\text{odd}}(z) = \sum_{2 \nmid r} J_r z^{-\frac{r}{2}-1}$.
Hence
\begin{gather*}
L(z) = \frac{z^{-2}}{16} + \frac{1}{4} \sum :\!J_{\text{odd}}(z)^2\!\!:.
\end{gather*}
Or equivalently
\begin{gather}
L_k = \frac{1}{4} \sum_{\frac12 (r+s)=k} : \! J_r J_s \! : + \frac{1}{16} \delta_{k,0}. \label{eq:classical Virasoro}
\end{gather}
Formula \eqref{eq:classical Virasoro} is well-known; it coincides with \cite[equation~(2.16)]{Z} after substitution $J_r = 2I_{\frac12r}$. Let us emphasize that formula \eqref{eq:quantum Virasoro} is a $q$-deformation of \eqref{eq:classical Virasoro}.
\end{Example}

\section{Relations on conformal blocks}\label{section: conformal}

\subsection{Whittaker vector}
In this section we define and study basic properties of Whittaker vector $W(z|u_1, \dots, u_N) \in \mathcal{F}_{u_1} \otimes \cdots \otimes \mathcal{F}_{u_N}$. We will restrict ourself to case when $\mathcal{F}_{u_1} \otimes \cdots \otimes \mathcal{F}_{u_N}$ is irreducible.

\begin{Lemma} \label{lemma 3.1 FJMM}
$\mathcal{F}_{u_1} \otimes \cdots \otimes \mathcal{F}_{u_N}$ is irreducible if $u_i/u_j \neq q^{k}$ for any $k \in \mathbb{Z}$.
\end{Lemma}
\begin{proof}
Follows from the proof of \cite[Lemma 3.1]{FFJMM}.
\end{proof}
In papers \cite{N12,Ts} Whittaker vector is defined geometrically. We will define Whittaker vector by algebraic properties (cf. \cite[Proposition~4.15]{N12}). Then we will prove that these properties define Whittaker vector uniquely up to normalization if the module $\mathcal{F}_{u_1} \otimes \cdots \otimes \mathcal{F}_{u_N}$ is irreducible.
%Below we will always assume that $\mathcal{F}_{u_1} \otimes \cdots \otimes \mathcal{F}_{u_N}$ is irreducible and will not repeat this condition.

\begin{Definition} \label{def: Whit eigen}
Whittaker vector $W(z|u_1, \dots, u_N) \in \mathcal{F}_{u_1} \otimes \cdots \otimes \mathcal{F}_{u_N}$ is an eigenvector of operators $E_{a,b}$ for $Nb \geqslant a \geqslant 0$ with eigenvalues
\begin{gather}
E_{0,k} W(z|u_1, \dots, u_N) = \frac{z^k}{q^{k/2}-q^{-k/2}} W(z|u_1, \dots, u_N), \label{eq: Whit eigen1} \\
E_{Nk,k} W(z|u_1, \dots, u_N) = \frac{\big(\big({-}q^{-\half}
\big)^Nu_1 \cdots u_N z\big)^{k}}{q^{-k/2}-q^{k/2}} W(z|u_1, \dots, u_N) \label{eq: Whit eigen2}
\end{gather}
for $k>0$;
\begin{gather}
E_{k_1, k_2} W(z|u_1, \dots, u_N) =0 \label{eq: Whit eigen3}
\end{gather}
for $N k_2 > k_1 > 0$.We require $W(z|u_1, \dots, u_N) = \vac \otimes \cdots \otimes \vac + \cdots$ to fix normalization (by dots we mean lower vectors).
\end{Definition}

\begin{Remark}
Whittaker vector is an element of graded completion of $\mathcal{F}_{u_1} \otimes \cdots \otimes \mathcal{F}_{u_N}$. Abusing notation, we use the same symbols for modules and their completions.
\end{Remark}

\begin{Remark}
Whittaker vector is an eigenvector for surprisingly big algebra. This explains why we have to consider specific eigenvalues (for general eigenvalues there is no eigenvector in corresponding representation). Theorem \ref{th:Whit construction} clarify origin of this eigenvalues.
\end{Remark}

\begin{Theorem} \label{th:Whittaker condition}
If $\mathcal{F}_{u_1} \otimes \cdots \otimes \mathcal{F}_{u_N}$ is irreducible, then there exists unique Whittaker vector.
\end{Theorem}

One can find a proof of Theorem \ref{th:Whittaker condition} in Appendix \ref{Appendix:Whittaker}. This statement can be considered as a~part of folklore; unfortunately, we do not know a precise reference for the theorem.

\begin{Proposition} \label{prop: Whit up to norm}
Decomposition of Whittaker vector $W\big(z^{1/n}|u\big) \in \mathcal{F}^{[1/n]}_{u}$ with respect to \eqref{eq:th:decomposition} is given by Whittaker vectors $W\big(z | u q^{ l_0}, \dots ,u q^{\frac{n-1}{n}+ l_{n-1}}\big)$ up to normalization.
\end{Proposition}

\begin{proof}
The idea of the proof is that relations \eqref{eq: Whit eigen1} and \eqref{eq: Whit eigen2} for $N=1$ implies these relations for $N=n$. Let us work out conditions \eqref{eq: Whit eigen2} for $W\big(z | u q^{ l_0}, \dots ,u q^{\frac{n-1}{n}+ l_{n-1}}\big)$
\begin{gather*}
E_{nk,k} W\big(z | u q^{ l_0}, \dots ,u q^{\frac{n-1}{n}+ l_{n-1}}\big) = \frac{\left( \big({-}q^{-\half}
\big)^n \prod_k \big( u q^{\frac{k}{n}+ l_k} \big) z \right)^{k}}{\big(q^{-k/2}-q^{k/2}\big)} W\big(z | u q^{ l_0}, \dots ,u q^{\frac{n-1}{n}+ l_{n-1}}\big).
\end{gather*}

Let us calculate
\begin{gather*}
\big({-}q^{-\half} \big)^n \prod_{k=0}^{n-1} \big( u q^{\frac{k}{n}+ l_k} \big) = \big({-}q^{-\half}\big)^n u^n q^{\frac{n-1}{2}} = \big({-}q^{-\frac{1}{2n}} u \big)^n.
\end{gather*}
So $ W\big(z | u q^{ l_0}, \dots ,u q^{\frac{n-1}{n}+ l_{n-1}}\big)$ is defined (up to normalization) by conditions
\begin{gather}
E_{0,k} W\big(z | u q^{ l_0}, \dots ,u q^{\frac{n-1}{n}+ l_{n-1}}\big) = \frac{z^k}{q^{k/2}-q^{-k/2}} W\big(z | u q^{ l_0}, \dots ,u q^{\frac{n-1}{n}+ l_{n-1}}\big), \label{eq:proof restr Prop9.1 1}\\
E_{nk,k} W\big(z | u q^{ l_0}, \dots ,u q^{\frac{n-1}{n}+ l_{n-1}}\big) = \frac{\big({-}q^{-\frac{1}{2n}} z^{\frac{1}{n}}u \big)^{nk} }{q^{-k/2}-q^{k/2}} W\big(z | u q^{ l_0}, \dots ,u q^{\frac{n-1}{n}+ l_{n-1}}\big),\\
E_{k_1, k_2} W\big(z | u q^{ l_0}, \dots ,u q^{\frac{n-1}{n}+ l_{n-1}}\big) = 0
\label{eq:proof restr Prop9.1 2}
\end{gather}
for $k>0$ and $nk_2> k_1 >0$. Denote by $E^{[1/n]}_{a,b}$ generators of $\qd{1/n}$. Then
\begin{gather}
E_{0,k} W\big(z^{1/n}|u\big) = E^{[1/n]}_{0,nk} W\big(z^{1/n}|u\big) =
\frac {\big(z^{1/n} \big)^{nk}} { \big( q^{1/n} \big)^{nk/2}-\big( q^{1/n} \big)^{-nk/2} }
 W\big(z^{1/n}|u\big), \label{eq:proof restr Prop9.1 3} \\
E_{nk,k} W\big(z^{1/n}|u\big) = E^{[1/n]}_{nk,nk} W\big(z^{1/n}|u\big) = \frac{\big({-}q^{-\frac{1}{2n}} z^{\frac{1}{n}}u \big)^{nk} }{\big( q^{1/n} \big)^{-nk/2}-\big( q^{1/n} \big)^{nk/2}} W\big(z^{1/n}|u\big), \label{eq:proof restr Prop9.1 4} \\
E_{k_1, k_2} W\big(z^{1/n}|u\big) = E^{[1/n]}_{k_1, n k_2} W\big(z^{1/n}|u\big) =0.
\label{eq:proof restr Prop9.1 5}
\end{gather}

Note that conditions \eqref{eq:proof restr Prop9.1 3}--\eqref{eq:proof restr Prop9.1 5} and conditions \eqref{eq:proof restr Prop9.1 1}--\eqref{eq:proof restr Prop9.1 2} coincide. Hence each component of~$W\big(z^{1/n}|u\big)$ also satisfy those conditions, i.e., coincide with Whittaker vector up to normalization.
\end{proof}

\subsubsection[Whittaker vector for $\mathcal{F}_u$]{Whittaker vector for $\boldsymbol{\mathcal{F}_u}$} \label{sssec: Whit in 1Fock}

Recall that we use notation $c(\lambda) = \sum\limits_{s \in \lambda} c(s)$.
\begin{Proposition} \label{prop: Whittaker Shur}
We have an expansion of vector $W(z|u)$ in the basis $|\lambda\rangle$
\begin{gather}\label{eq:W:inlambda}
W(z|u) = \sum \frac{q^{-\frac12c(\lambda)} }{\prod\limits_{s\in\lambda}\big(q^{\frac12h(s)}-q^{-\frac12h(s)}\big)} z^{| \lambda |} |\lambda\rangle.
\end{gather}
\end{Proposition}

To prove the proposition we need the following lemmas.

\begin{Lemma} \label{lemma: Whittaker Cond 1}
Following vectors in $\mathcal{F}_u$ coincide
\begin{gather}
\exp\left(\sum_{k=1}^\infty \frac{z^k}{k\big(q^{k/2}-q^{-k/2}\big)}a_{-k}\right)|\varnothing\rangle =\sum \frac{q^{-\frac12c(\lambda)} }{\prod\limits_{s\in\lambda}\big(q^{\frac12h(s)}-q^{-\frac12h(s)}\big)} z^{| \lambda |} |\lambda\rangle.
\label{eq:lemma: Whittaker Cond 1}
\end{gather}
\end{Lemma}

\begin{proof}
Recall Cauchy identity
\begin{gather*}
\exp \left( -\sum_k \frac1kp_k(x)p_k(y) \right)=\prod_{i,j}(1-x_iy_j )=\sum_{\lambda} (-1)^{|\lambda|} s_{\lambda'}(x)s_\lambda(y).
\end{gather*}
Let us use specialization of Cauchy identity (see \cite[Section~1.4, Example~2]{M})
\begin{gather*}p_k(x) \mapsto \frac{-z^k}{q^{k/2}-q^{-k/2}}, \qquad s_{\lambda'}(x) \mapsto (-1)^{|\lambda|} \frac{q^{-\frac12c(\lambda)} z^{|\lambda|}}{\prod\limits_{s\in\lambda}\big(q^{\frac12h(s)}-q^{-\frac12h(s)}\big)}.
\end{gather*}
To finish the proof let us recall that there is an identification of space of symmetric polynomials and Fock module $F^a_{\alpha}$ given by $s_{\lambda} \mapsto | \lambda \rangle$ and $p_k \mapsto a_{-k}$ (see~\cite{KR}).
\end{proof}

\begin{Remark} For $|q| < 1$ this specialization comes from substitution $p_k\big(zq^{1/2},zq^{3/2},\dots\big)$.
\end{Remark}

\begin{Lemma} \label{lemma: Whittaker Cond 2}
Following vectors in $\mathcal{F}_u$ coincide
\begin{gather*}%\label{eq:W:second}
\exp\left(-\sum_{k=1}^\infty \frac{\big({-}q^{-1/2}uz \big)^{k}}{k\big(q^{k/2}-q^{-k/2}\big)}\rho(E_{-k,-k})\right)|\varnothing\rangle = \sum \frac{q^{-\frac12c(\lambda)} }{\prod\limits_{s\in\lambda}\big(q^{\frac12h(s)}-q^{-\frac12h(s)}\big)} z^{| \lambda |} |\lambda\rangle.
\end{gather*}
\end{Lemma}
\begin{proof}
Recall that we have defined an operator $\I_{\tau}$ by \eqref{eq: def I tau}. Let us calculate
\begin{gather} \label{eq: Itau to Schur}
\I_{\tau} \left( \sum \frac{q^{-\frac12 c(\lambda)} }{\prod\limits_{s\in\lambda}\big(q^{\frac12h(s)}-q^{-\frac12h(s)}\big)} z^{| \lambda |} |\lambda\rangle \right) = \sum \frac{u^{| \lambda |} q^{\frac12 c(\lambda) - \frac12 | \lambda|} }{\prod\limits_{s\in\lambda}\big(q^{\frac12h(s)}-q^{-\frac12h(s)}\big)} z^{| \lambda |} |\lambda\rangle.
\end{gather}
Proposition \ref{prop: I sigma} implies that
\begin{gather}
\I_{\tau} \left(\exp\left(-\sum_{k=1}^\infty \frac{\big({-}q^{-1/2}uz \big)^{k}}{k\big(q^{k/2}-q^{-k/2}\big)}\rho(E_{-k,-k})\right)|\varnothing\rangle \right) \nonumber\\
\qquad{} = \exp\left(-\sum_{k=1}^\infty \frac{\big({-}q^{-1/2}uz \big)^{k}}{k\big(q^{k/2}-q^{-k/2}\big)} a_{-k}\right)|\varnothing\rangle.\label{eq: Itau to Newton}
\end{gather}
Using Cauchy identity for another specialization
\begin{gather*}
p_k(x) \mapsto \frac{\big({-}q^{-1/2}u z \big)^k}{q^{k/2}-q^{-k/2}},\qquad s_{\lambda'}(x) \mapsto (-1)^{|\lambda|} \frac{u^{|\lambda|}q^{\frac12 c(\lambda) - \frac12|\lambda| }}{\prod\limits_{s\in\lambda}\big(q^{\frac12h(s)}-q^{-\frac12h(s)}\big)} z^{|\lambda|},
\end{gather*}
we see that r.h.s.\ of \eqref{eq: Itau to Schur} and~\eqref{eq: Itau to Newton} coincide.
\end{proof}

\begin{Remark} For $|q| > 1$ this specialization comes from substitution $p_k\big({-}zuq^{-1},-zuq^{-2},\dots\big)$.
\end{Remark}

\begin{proof}[Proof of Proposition \ref{prop: Whittaker Shur}]
To prove the Proposition let us check that r.h.s.\ of \eqref{eq:W:inlambda} satisfy condition \eqref{eq: Whit eigen1}--\eqref{eq: Whit eigen3}. Conditions \eqref{eq: Whit eigen1} and \eqref{eq: Whit eigen2} are equivalent to Lemmas~\ref{lemma: Whittaker Cond 1} and~\ref{lemma: Whittaker Cond 2} correspondingly. To finish the proof we note that for $N=1$ conditions \eqref{eq: Whit eigen1} and \eqref{eq: Whit eigen2} imply~\eqref{eq: Whit eigen3}.
\end{proof}

\begin{Remark}
Note that we did not use Theorem \ref{th:Whittaker condition} in the proof of Proposition \ref{prop: Whittaker Shur}. Moreover, we have proven a particular case of the Theorem for $W(z|u)$.
%Actually, reference to the Theorem would allow us to omit of Lemma \ref{lemma: Whittaker Cond 2}. Namely, note that condition \eqref{eq: Whit eigen1} uniquely determines Whittaker vector $W(z|u)$, then Theorem \ref{th:Whittaker condition} imply that r.h.s.\ of \eqref{eq:W:inlambda} satisfy also \eqref{eq: Whit eigen2} and \eqref{eq: Whit eigen3}.
\end{Remark}
\subsubsection{Whittaker vector and restriction on sublattice}

Let us recall interpretation of decomposition \eqref{eq:th:decomposition} in terms of boson-fermion correspondence. One can identify $F^{n \psi} = F^{\psi}$. Embedding $F^a_{0} \subset F^{\psi}$ corresponded to embedding $F^{na} \otimes \mathbb{C}\big[\Q\big] \subset F^{n \psi} $. Hence we have decomposition
\begin{gather} \label{eq: n bonos n femion correspondance}
F^a = \bigoplus_{l \in \Q} F^{na} \otimes e^{\sum_i l_i Q_i}.
\end{gather}
We argue by construction that decomposition \eqref{eq:th:decomposition} correspond to \eqref{eq: n bonos n femion correspondance}.
\begin{Proposition} \label{prop: n-core}
Decomposition \eqref{eq: n bonos n femion correspondance} identifies $\vac \otimes e^{-\sum_i l_i Q_i}$ with $| \lambda \rangle$ for some $\lambda$. Moreover, partition $\lambda$ satisfy following properties.
\begin{enumerate}[$(i)$]\itemsep=0pt
\item \label{i}Hooks of $\lambda$ are in bijection with tuples $\{ (i,j,k_i, k_j) \mid k_i<l_i; \ k_j \geqslant l_j; \ nk_i + i > nk_j + j \}$. Length of a hook corresponding to a tuple $(i, j, k_i, k_j)$ equals to $n(k_i - k_j) + i-j$.
\item \label{ii} $\frac{1}{2} \sum\limits_{i=0}^{n-1} (\frac{i}{n}+l_i)^2 = \frac{| \lambda |}{n} + \frac12 \sum\limits_{i=0}^{n-1} \frac{i^2}{n^2}$.
\end{enumerate}
\end{Proposition}

\begin{proof}
The $n$-fermion Fock space $F^{n \psi}$ is isomorphic to tensor product $F^{ \psi}\otimes \dots \otimes F^{ \psi}$. The $n$-Heisenberg highest vectors are products $|l_0\rangle \otimes\dots \otimes |l_{n-1}\rangle$. After identification of $F^{n \psi}$ with one $F^{ \psi}$, these products becomes~\eqref{eq:lambda l} for special $\lambda$. Such diagrams $\lambda$ are called $n$-cores.

Combinatorially boson-fermion correspondence is a correspondence between Maya diagrams and charged partitions $(\lambda,l )$, see, e.g., \cite[Section~6.4]{Negut15} or~\cite{FM}. Boxes of a partition correspond to pairs of white and black points in Maya diagram such that the coordinate of white point is greater than the coordinate of the black point. The hook length equals the difference between the coordinates of white and black points (cf.~\cite[Section~6.4]{Negut15}). This proves~\eqref{i}.

For formula \eqref{ii} see, e.g., \cite[Proposition~2.30]{FM}.
\end{proof}

\begin{Lemma}
Let $l_i > l_j$. Let us consider hooks with fixed $i$ and $j$ $($see Proposition~{\rm \ref{prop: n-core})}.
\begin{itemize}\itemsep=0pt
\item If $i>j$, then possible lengths of hooks are $nk+i-j$ for $k= 0, 1, \dots, l_i-l_j-1$.
\item If $i<j$ then possible lengths of hooks are $nk+i-j$ for $k= 1, 2, \dots, l_i-l_j-1$.
\end{itemize}
There are exactly $l_i-l_j-k$ such hooks of length $nk + i -j$ for all possible $k$.
\end{Lemma}

For each $l \in \Q$ we will use notation $\prod\limits_{(i,j,k)}$ for product over triples $(i,j,k)$ satisfying following conditions. Numbers $i$, $j$ run over $0, \dots, n-1$ with condition $l_i>l_j$. If $i>j$, then $k= 0, 1, \dots, l_i-l_j-1$; if $i<j$ then $k= 1, 2, \dots, l_i-l_j-1$.

\begin{Corollary} \label{corol: product s ijk} If diagram $\lambda$ corresponds to $\vac \otimes e^{-\sum_i l_i Q_i}$, then
\begin{gather*}
\prod_{s\in\lambda}\big(q^{\frac12h(s)}-q^{-\frac12h(s)}\big) = \prod_{(i,j,k)} \big(q^{\frac12 (nk+i-j)}-q^{-\frac12(nk+i-j)}\big)^{l_i-l_j-k}.
\end{gather*}
\end{Corollary}

\begin{Theorem} \label{th: Whit decomposition}
Decomposition of Whittaker vector $W\big(z^{1/n}|u^{1/n}\big) \in \mathcal{F}^{[1/n]}_{u^{1/n}}$ with respect to~\eqref{eq:th:decomposition} is given by
\begin{gather*}
\frac{q^{-\frac{1}{2n} c(\lambda)} z^{\frac12 \sum \big( \big(l_i + \frac{i}{n}\big)^2 - \big(\frac{i}{n}\big)^2 \big) }}{ \prod\limits_{(i,j,k)} \Big(q^{\frac12 \big(k+\frac{i-j}{n} \big)}-q^{-\frac12 \big(k+\frac{i-j}{n} \big)}\Big)^{l_i-l_j-k} } W\big(z | u q^{ l_0}, \dots ,u q^{\frac{n-1}{n}+ l_{n-1}}\big).
\end{gather*}
\end{Theorem}

\begin{proof}
Recall that according to Proposition~\ref{prop: Whit up to norm} we just have to verify the coefficient. This coefficient can be found as coefficient of $W\big(z^{1/n}|u^{1/n}\big)$ at highest vector of $\mathcal{F}_{u q^{ l_0}} \otimes \cdots \otimes \mathcal{F}_{u q^{\frac{k}{n}+ l_k }} \otimes \cdots \otimes \mathcal{F}_{u q^{\frac{n-1}{n}+ l_{n-1}}}$.

By Proposition \ref{prop: Whittaker Shur}, coefficient of $W\big(z^{1/n}|u^{1/n}\big)$ at $| \lambda \rangle$ is
\begin{gather} \label{eq: factor for lambda vs l}
\frac{q^{-\frac{1}{2n} c(\lambda)} z^{\frac{| \lambda |}{n} } }{\prod\limits_{s\in\lambda}\big(q^{\frac{1}{2n}h(s)}-q^{-\frac{1}{2n}h(s)}\big)} = \frac{q^{-\frac{1}{2n} c(\lambda)} z^{\frac12 \sum \big( (l_i + \frac{i}{n})^2 - (\frac{i}{n})^2 \big) } }{ \prod\limits_{(i,j,k)} \Big(q^{\frac12 \big(k+\frac{i-j}{n} \big)}-q^{-\frac12 \big(k+\frac{i-j}{n} \big)}\Big)^{l_i-l_j-k} }. \end{gather}
Equality~\eqref{eq: factor for lambda vs l} follows from Corollary~\ref{corol: product s ijk} and Proposition~\ref{prop: n-core}\eqref{ii}.
\end{proof}

\subsection{Shapovalov form}

\begin{Definition}
Let $M_1$, $M_2$ be two representations of $\qD$. A pairing $\langle-,-\rangle_s\colon M_1 \otimes M_2 \rightarrow \mathbb{C}$ is called Shapovalov if $\langle v, E_{a, b}w \rangle_s = -\langle E_{-a,-b}v, w \rangle_s$.
\end{Definition}

\begin{Proposition} \label{prop: Sapovalov p for Fock}
There exists a unique Shapovalov pairing $\mathcal{F}_{u} \otimes \mathcal{F}_{q u^{-1}} \rightarrow \mathbb{C}$ such that $\langle 0 | 0 \rangle_s \allowbreak =1$.
\end{Proposition}

\begin{proof}
There exists a unique pairing on Fock space such that $a_k$ is dual to $-a_{-k}$. Since algebra $\qD$ is generated by modes of $E(z)$ and $F(z)$, it remains to check Shapovalov property for them. Formulas \eqref{eq:E} and \eqref{eq:F} implies $\langle v, E(z) w \rangle = -\big\langle F\big(z^{-1}\big) v, w \big\rangle$.
\end{proof}

\begin{Remark}Note that this pairing differs from the pairing defined in Section~\ref{boson}. More precisely, in Section \ref{boson} we required $a_k$ to be dual to $a_{-k}$, not $-a_{-k}$.
\end{Remark}

\begin{Proposition} \label{prop: U(1) Poch}
$\big\langle W\big(1| q u^{-1}\big), W(z|u) \big\rangle = (qz;q,q)_{\infty}$.
\end{Proposition}

\begin{proof}
Formulas \eqref{eq:W:inlambda} and \eqref{eq:lemma: Whittaker Cond 1} imply
\begin{gather} \label{eq: W via exponent}
W(z|u) =
\exp\left(\sum_{k=1}^\infty \frac{z^k}{k\big(q^{k/2}-q^{-k/2}\big)}a_{-k}\right)|\varnothing\rangle.
\end{gather}
Using \eqref{eq: W via exponent} one can finish the proof by straightforward computation.
\end{proof}

\begin{Proposition} \label{prop: Sapovalov p for Fermion}
There exists a unique Shapovalov pairing $\mathcal{M}_{u} \otimes \mathcal{M}_{u^{-1}} \rightarrow \mathbb{C}$ such that
\begin{gather} \label{eq: Fermion and Shapoval}
\langle 1 | 0 \rangle_s =1, \qquad \langle v, \psi_i w \rangle = \langle \psi_{-i} v, w \rangle_s, \qquad \langle v, \psi^*_i w \rangle = \langle \psi^*_{-i} v, w \rangle_s.
\end{gather}
\end{Proposition}

\begin{proof}
There exists unique pairing satisfying \eqref{eq: Fermion and Shapoval}. The Shapovalov property can be checked directly using \eqref{eq:centerFermi}--\eqref{ferm:F}.
\end{proof}

\begin{Proposition} \label{prop: sign l}Shapovalov pairing for Fock modules for basis $| \lambda \rangle$ has form
\[ \langle \lambda | \mu \rangle_s = (-1)^{|\lambda|} \delta_{\lambda, \mu'}.\]
\end{Proposition}

\begin{proof}Let $p_1, \dots, p_i$ and $q_1, \dots, q_i$ be Frobenius coordinates of $\lambda$; analogously, $\tilde{p}_1, \dots, \tilde{p}_j$ and $\tilde{q}_1, \dots, \tilde{q}_j$ be Frobenius coordinates of $\mu$. Using identification given by~\eqref{eq:lambda l} we obtain
\begin{gather}
\langle \mu, 1 | \lambda,0 \rangle_s = (-1)^{\sum_k (q_k-1)} (-1)^{\sum_k (\tilde{q}_k-1)} \nonumber\\
\hphantom{\langle \mu, 1 | \lambda,0 \rangle_s =}{}\times \langle 1 | \psi_{\tilde{q}_1}^* \cdots \psi_{\tilde{q}_i}^* \psi_{\tilde{p}_i-1} \cdots \psi_{\tilde{p}_1-1} \psi_{-p_1} \cdots \psi_{-p_i} \, \psi_{-q_i+1}^* \cdots \psi_{-q_1+1}^* \vac_s.\label{eq: Sh prod in lam proof}
\end{gather}
Evidently, if this product is non-zero, then $i=j$, $q_k=\tilde{p}_k$, and $p_l = \tilde{q}_l$; this exactly means that $\mu = \lambda'$. It remains to calculate r.h.s.\ of~\eqref{eq: Sh prod in lam proof} in this case; it equals $(-1)^{\sum_k q_k+\sum_l p_l+i^2} = (-1)^{|\lambda|}$ since $|\lambda| = \sum_k q_k+\sum_l p_l-i $.
\end{proof}

\begin{Definition}
Standard Shapovalov pairing $\langle -, - \rangle_{ss} \colon M_1 \otimes M_2 \rightarrow \mathbb{C}$ for $M_1 = \mathcal{F}_{u_1} \otimes \cdots \otimes \mathcal{F}_{u_n} $ and $M_2 = \mathcal{F}_{q/u_n} \otimes \cdots \otimes \mathcal{F}_{q/u_1}$ is defined by
\begin{gather*}
\langle x_1 \otimes \cdots \otimes x_n, y_n \otimes \cdots \otimes y_1 \rangle_{ss} = \prod_i \langle x_i , y_i \rangle_i.
\end{gather*}
Here $\langle -,-\rangle_i$ stands for Shapovalov pairing $\mathcal{F}_{u_i} \otimes \mathcal{F}_{q u_i^{-1}} \rightarrow \mathbb{C}$ as in Proposition~\ref{prop: Sapovalov p for Fock}.
\end{Definition}

\begin{Proposition} \label{prop: Shapoval sign}
Shapovalov pairing on $\mathcal{F}^{[1/n]}_{u} \otimes \mathcal{F}^{[1/n]}_{q^{\frac{1}{n}}/u}$ restricts to $(-1)^{| \lambda |} \langle -, - \rangle_{ss} \colon M_1 \otimes M_2 \rightarrow \mathbb{C}$ for $M_1 = \mathcal{F}_{u q^{ l_0}} \otimes \cdots \otimes \mathcal{F}_{u q^{\frac{n-1}{n}+ l_{n-1}
}}$ and $M_2 = \mathcal{F}_{ q^{\frac{1}{n}- l_{n-1}}/u} \otimes \cdots \otimes \mathcal{F}_{q^{1- l_0}/u}$
{\upshape(}with respect to decomposition~\eqref{eq:th:decomposition}{\upshape)}. Other pairs of direct summands are orthogonal.
\end{Proposition}

\begin{proof}
Property $\langle v, E_{a, b}w \rangle_s = -\langle E_{-a,-b}v, w \rangle_s$ is preserved under restriction. Since module $M_1$ and $M_2$ are irreducible, there exists unique Shapovalov pairing $M_1 \otimes M_2 \rightarrow \mathbb{C}$. So restriction of pairing coincides with the standard pairing up to multiplicative constant. The constant equals to $\langle \lambda' | \lambda \rangle = (-1)^{|\lambda|}$ due to Proposition \ref{prop: sign l} (and Proposition \ref{prop: n-core}).

Orthogonality with all other summands also follows from Proposition \ref{prop: sign l} and irreducibi\-lity.
\end{proof}

\begin{Remark}Let us comment on another way to prove orthogonality mentioned in Proposition~\ref{prop: Shapoval sign}. All direct summands are pairwise non-isomorphic. Hence there is no non-zero pairing for all other pairs of direct summands.
\end{Remark}

\subsection{Conformal blocks}

\begin{Definition}
Pochhammer and double Pochhammer symbols are defined by
\begin{gather*}
(u; q_1)_{\infty} = \prod_{i=0}^{\infty} \big(1-q_1^i u\big), \qquad
(u; q_1, q_2)_{\infty} = \prod_{i,j=0}^{\infty} \big(1-q_1^i q_2^j u\big).
\end{gather*}
\end{Definition}

\begin{Remark}
Standard definition works for $|q_1|, |q_2|<1$ and any~$u$. For sufficiently small $u$ double Pochhammer symbol can be presented as \[
(u;q_1, q_2)_{\infty} =\exp \left( -\sum_{k=1}^{\infty} \frac{u^k}{k\big(1-q_1^k\big)\big(1-q_2^k\big)} \right).
\]
 The series in $u$ has non-zero radius of convergence for $|q_1|, |q_2| \neq 1$. Moreover the series enjoys property $\big(u; q_1^{-1}, q_2\big)_{\infty} = 1/(q_1 u; q_1,q_2)_{\infty}$, hence we can define double Pochhammer symbol for any $|q_1|, |q_2| \neq 1$.

In particular, new definition implies $\big(u; q, q^{-1}\big)_{\infty} = 1/(qu; q,q)_{\infty}$; this is important to compare our formulas with~\cite{BGM}. Below we assume $|q| \neq 1$.
\end{Remark}

\begin{Definition}\label{def:conformal block}
Let us define $q$-deformed conformal block
\begin{gather}
\mathcal{Z} (u_1, \dots, u_n ; z) = z^{\frac{\sum (\log u_i)^2}{2 (\log q)^2}}\nonumber\\
\hphantom{\mathcal{Z} (u_1, \dots, u_n ; z) =}{}\times \prod_{i \neq j} \frac{1}{\big(qu_i u_j^{-1}; q, q\big)_{\infty}} \big\langle W_u\big(1|qu_n^{-1}, \dots, qu_1^{-1}\big) , W(z| u_1, \dots, u_n) \big\rangle_{ss}.\!\!\!\!\label{eq: def q-der conf block}
\end{gather}
\end{Definition}

\begin{Remark}
AGT statement claims that function $\mathcal{Z} (u_1, \dots, u_n ; z)$ is equal to Nekrasov partition function for pure supersymmetric ${\rm SU}(n)$ 5d theory. This was conjectured in \cite{AY10}, the proof follows from the geometric construction of the Whittaker vector given in the \cite{N12} and \cite{Ts}.
\end{Remark}

\begin{Theorem} \label{Theorem: the relation}
\begin{gather*}
z^{\frac12 \sum \frac{i^2}{n^2}} \prod_{i \neq j} \frac{1}{\big(q^{1+\frac{i-j}{n}}; q, q\big)_{\infty}} \big( q^{\frac{1}{n}} z^{\frac{1}{n}}; q^{\frac{1}{n}} ,q^{\frac{1}{n}} \big)_{\infty} \\
\qquad{} = \sum_{(l_0, \dots, l_{n-1}) \in \Q} \mathcal{Z} \big( q^{l_0}, q^{\frac{1}{n}+l_1}, \dots, q^{\frac{n-1}{n} + l_{n-1}}; z \big).
\end{gather*}
\end{Theorem}

The idea of the proof is to find two different expressions for $\big\langle W\big(1| q^{1/n} \big) , W(z^{1/n}|1)\big\rangle$ using Theorem~\ref{th: Whit decomposition}. To do this we need to simplify $\mathcal{Z} (u_1, \dots, u_n ; z)$ after substitution $u_{i-1} = q^{\frac{i}{n}+l_i}$. Let us concentrate on the second factor of \eqref{eq: def q-der conf block}.
\begin{Proposition} \label{prop: Z 1-loop}
\begin{gather*}
\prod_{i \neq j} \frac{\big( q^{1+\frac{i-j}{n}} ;q,q \big)_{\infty}}{ \big(q^{l_i - l_j+1+\frac{i-j}{n}};q,q \big)_{\infty}} = (-1)^{| \lambda |} \prod_{(i,j,k)} \frac{1}{\big(q^{\frac{k}{2}+\frac{i-j}{2n}} -q^{-\frac{k}{2}-\frac{i-j}{2n}} \big)^{2(l_i-l_j-k)}}.
\end{gather*}
\end{Proposition}
\begin{proof}
Let $l_i-l_j>0$. It is straightforward to check that
\begin{gather}
\frac{\big( q^{1+\frac{i-j}{n}} ;q,q \big)_{\infty}}{\big(q^{l_i - l_j+1+\frac{i-j}{n}};q,q \big)_{\infty}} = \big(q^{1+\frac{i-j}{n}} ; q \big)_{\infty}^{l_i-l_j} \prod_{k=1}^{l_i-l_j-1} \frac{1}{\big(1-q^{k+\frac{i-j}{n}} \big)^{l_i-l_j-k}}, \label{eq: Poch1} \\
\frac{\big(q^{1+\frac{j-i}{n}} ;q,q \big)_{\infty}}{\big(q^{l_j - l_i+1+\frac{j-i}{n}} ;q,q \big)_{\infty}} = \big(q^{1+\frac{j-i}{n}}; q \big)_{\infty}^{l_j-l_i} \prod_{k=0}^{l_i-l_j-1} \frac{1}{\big(1-q^{-k-\frac{i-j}{n}} \big)^{l_i-l_j-k}}. \label{eq: Poch2}
\end{gather}

Denote by $v_{ij} = q^{\frac{i-j}{n}}$ for $i>j$ and $v_{ij} = q^{\frac{n+i-j}{n}} $ for $i<j$. Formulas~\eqref{eq: Poch1}--\eqref{eq: Poch2} implies the following assertions. For $i>j$
\begin{gather}
\frac{\big( q^{1+\frac{i-j}{n}} ;q,q \big)_{\infty}}{\big(q^{l_i - l_j+1+\frac{i-j}{n}};q,q \big)_{\infty}} \times \frac{\big(q^{1+\frac{j-i}{n}};q,q \big)_{\infty}}{\big(q^{l_j - l_i+1+\frac{j-i}{n}};q,q \big)_{\infty}} \nonumber \\
\qquad{}= ( v_{ij}; q )_{\infty}^{l_i-l_j} ( v_{ji}; q )_{\infty}^{l_j-l_i} \prod_{k=0}^{l_i-l_j-1} \frac{(-1)^{l_i - l_j - k}}{\big(q^{\frac{k}{2}+\frac{i-j}{2n}} -q^{-\frac{k}{2}-\frac{i-j}{2n}} \big)^{2(l_i-l_j-k)}}. \label{eq: Poch3}
\end{gather}
For $j>i$
\begin{gather}
\frac{\big( q^{1+\frac{i-j}{n}} ;q,q \big)_{\infty}}{ \big(q^{l_i - l_j+1+\frac{i-j}{n}};q,q \big)_{\infty}} \times \frac{\big(q^{1+\frac{j-i}{n}};q,q \big)_{\infty}}{ \big(q^{l_j - l_i+1+\frac{j-i}{n}};q,q \big)_{\infty}} \nonumber \\
\qquad{}= ( v_{ij}; q )_{\infty}^{l_i-l_j} ( v_{ji}; q )_{\infty}^{l_j-l_i} \prod_{k=1}^{l_i-l_j-1} \frac{(-1)^{l_i - l_j - k}}{\big(q^{\frac{k}{2}+\frac{i-j}{2n}} -q^{-\frac{k}{2}-\frac{i-j}{2n}} \big)^{2(l_i-l_j-k)}}. \label{eq: Poch4}
\end{gather}
Using identities \eqref{eq: Poch3}--\eqref{eq: Poch4}, we obtain
\begin{gather*}
\prod_{i \neq j} \frac{\big( q^{1+\frac{i-j}{n}} ;q,q \big)_{\infty}}{\big(q^{l_i - l_j+1+\frac{i-j}{n}};q,q \big)_{\infty}} = \prod_{(i,j,k)} \frac{(-1)^{l_i - l_j - k}}{\big(q^{\frac{k}{2}+\frac{i-j}{2n}} -q^{-\frac{k}{2}-\frac{i-j}{2n}} \big)^{2(l_i-l_j-k)}}.
\end{gather*}
To finish the proof, it remains to clarify the sign. This product already appeared as the product over all hooks. For diagram $\lambda$ the number of hooks is~$| \lambda |$.
\end{proof}

\begin{proof}[Proof of Theorem \ref{Theorem: the relation}]
We will provide two different expressions for $\big\langle W\big(1| q^{1/n} \big) ,\! W(z^{1/n}|1) \big\rangle$ to prove the theorem. On one hand (by Proposition~\ref{prop: U(1) Poch})
\begin{gather}
\big\langle W\big(1| q^{1/n} \big) , W\big(z^{1/n}|1\big) \big\rangle = \big( q^{\frac{1}{n}} z^{\frac{1}{n}}; q^{\frac{1}{n}} ,q^{\frac{1}{n}} \big)_{\infty}. \label{eq: proof conformal1}
\end{gather}
On the other hand (by Theorem~\ref{th: Whit decomposition} and Proposition~\ref{prop: Shapoval sign})
\begin{gather}
\big\langle W\big(1| q^{1/n} \big) , W\big(z^{1/n}|1\big) \big\rangle = \sum_{(l_0, \dots, l_{n-1}) \in \Q} \frac{ z^{\frac12 \sum \big( \big(l_i + \frac{i}{n}\big)^2 - \big(\frac{i}{n}\big)^2 \big) } }{ \prod\limits_{(i,j,k)} \big(q^{\frac12 \big(k+ \frac{i-j}{n} \big)}-q^{-\frac12 \big( k+\frac{i-j}{n}\big)}\big)^{2(l_i-l_j-k)}} \nonumber\\ \qquad{} \times(-1)^{|\lambda|} \big\langle W\big(1 | q^{1-\frac{n-1}{n}- l_{n-1}} , \dots , q^{1-l_0}\big), W\big(z | q^{ l_0}, \dots , q^{\frac{n-1}{n}+ l_{n-1}}\big) \big\rangle_{ss}.
\label{eq: proof conformal2}
\end{gather}
Note that to prove \eqref{eq: proof conformal2} we also used following observations: lengths of hooks in $\lambda$ and $\lambda'$ coincides, and $c(\lambda)+ c(\lambda')=0$.

Multiplying r.h.s.\ of \eqref{eq: proof conformal1} and \eqref{eq: proof conformal2} by $z^{\frac12 \sum \frac{i^2}{n^2}} \prod\limits_{i \neq j} \frac{1}{\big(q^{1+\frac{i-j}{n}}; q, q \big)_{\infty}}$ we obtain
\begin{gather*}
z^{\frac12 \sum \frac{i^2}{n^2}}\! \prod_{i \neq j} \frac{1}{\big(q^{1+\frac{i-j}{n}}; q, q\big)_{\infty}} \big( q^{\frac{1}{n}} z^{\frac{1}{n}}; q^{\frac{1}{n}} ,q^{\frac{1}{n}} \big)_{\infty} = \sum_{(l_0, \dots, l_{n-1}) \in \Q}\!\!\! \mathcal{Z} \big(q^{l_0}, q^{\frac{1}{n}+l_1}, \dots, q^{\frac{n-1}{n} + l_{n-1}}; z \big).
\end{gather*}
Note that here we applied Proposition \ref{prop: Z 1-loop}.
\end{proof}

\appendix
\section{Regular product} \label{nonorm}

In this section, we develop general theory of regular product. Term `regular product' should be considered as an opposite to regularized (i.e., normally ordered) product.

Let $A(z) = \sum\limits_{k \in \mathbb{Z}} A_k z^{-k}$ be a formal power series with coefficients in $\End(V)$ for a vector space~$V$.
\begin{Definition}
The series $A(z)$ is called smooth if for any vector $v \in V$ there exists $N$ such that $A_k v = 0$ for $k \geqslant N$.
\end{Definition}
Let $G(z,w) = \sum\limits_{k,l \in \mathbb{Z}}G_{k,l} z^{-k} w^{-l}$ be a formal power series in two variables with operator coefficients. The operators $G_{k,l}$ acts on a vector space $V$.
\begin{Definition} \label{def:regular}
We will call $G(z,w)$ regular if for any $N$ and for any $v \in V$ there are only finitely many $G_{k,l}$ such that $k+l=N$ and $G_{k,l} v \neq 0$.
\end{Definition}

If a current $G(z,w)$ in two variables is regular one can substitute $w=az$ and obtain well-defined power series $G(z, az)$ for any $a \in \mathbb{C}$.

Let $A(z)$ and $B(w)$ be two smooth formal power series with operator coefficients. Recall definition of normal ordering. Denote $A_+ (z) = \sum\limits_{k \geqslant 0} A_{-k} z^{k}$ and $A_- (z) = \sum\limits_{k < 0} A_{-k} z^{k}$.

\begin{Definition} \label{def:norm}
Normal ordered product is defined as
\[ {: \! A(z) B(w) \! :} = A_+ (z) B(w) + (-1)^{\epsilon} B(w) A_- (z).\]
\end{Definition}

The sign $(-1)^{\epsilon}$ depends on \emph{parity} of $A(z)$ and $B(z)$ in the standard way. Note that smooth formal power series in two variables $: \! A(z) B(w) \! :$ is regular. Formal power series $A(z)$ and $B(z)$ are called local (in weaker sense) if
\begin{gather} \label{def:weakLocal}
A(z)B(w)-(-1)^{\epsilon} B(w) A(z) = \sum_{i=1}^{N} \sum_{j=0}^{s_i} C_{j}^{(i)} (w) \partial_w^j \delta(a_i z, w),
\end{gather}
where $s_1, \dots, s_N \in \mathbb{Z}_{\geq 0}$, $a_1, \dots, a_N \in \mathbb{C}$ and $C_{j}^{(i)} \! (w)$ are operator valued power series.

Then one has the following OPE
\begin{gather}
A(z) B(w) = \sum \frac{C_{j}^{(i)} \! (w)}{(a_i z)^j \big(1-\frac{w}{a_i z} \big)^{j}}+: \! A(z) B(w) \! :. \label{OPE}
\end{gather}
\begin{Proposition}
If currents $A(z)$ and $B(w)$ are smooth and satisfy \eqref{def:weakLocal}, then the following product \mbox{$(a_1 z - w)^{s_1} \cdots (a_N z - w)^{s_N} A(z) B(w)$} is regular.
\end{Proposition}
\begin{Definition} \label{def: the product}
For $a \neq a_i$ define regular product
\begin{gather*}
A(z) B(az) := \frac{ \big( (a_1 z - w)^{s_1} \cdots (a_N z - w)^{s_N} A(z) B(w) \big) \big|_{w=az}}{(a_1 z - az)^{s_1} \cdots (a_N z - az)^{s_N}}.
\end{gather*}
\end{Definition}

From \eqref{OPE} one obtains that normally ordered product and regular product are connected by the following relation
\begin{gather*}
A(z) B(az) = \sum \frac{C_{j}^{(i)} \! (az)}{(a_i z)^j (1-a/a_i)^{j}}+: \! A(z) B(az) \! :.
\end{gather*}

\begin{Example}
Let us consider case of fermions $A(z) = \psi(z)$, $B(z) = \psi^*(z)$, introduced in Section~\ref{fermi}. Beware, that we use notation $A(z)=A_k z^{-k}$, but $\psi(z) = \psi_i z^{-i-1}$ (hence \mbox{$A_k = \psi_{k-1}$}). Comparing formulas \eqref{norm1}--\eqref{norm2} with Definition \ref{def:norm}, we conclude
\begin{gather*}
{: \! \psi(z) \psi^* (w) \! :} = : \psi(z) \psi^*(w):_{(0)}.
\end{gather*}
Using $l$-depended normal ordering, we obtain
\begin{gather} \label{eq: regul prod psipsi}
\psi(z) \psi^* (w) = \frac{w^l z^{-l-1}}{1-w/z} + : \! \psi (z) \psi^* (w) \! :_{(l)}.
\end{gather}
Hence
\begin{gather*}
\psi(z) \psi^* (q z) = \frac{q^l}{1-q} z^{-1} + : \! \psi (z) \psi(qz) \! :_{(l)}.
\end{gather*}
This relation was used in formula~\eqref{ferm:E}.
\end{Example}

\begin{Example}
Let $A(z)=B(z)=E(z)$. Then
\begin{gather*}
E(z) E(w) = \frac{q^{-1}w}{z-q^{-1}w} E_{2}\big(q^{-1}w\big) - \frac{q w }{z-q w} E_{2}(w) + :E(z) E(w):.
\end{gather*}
Therefore,
\begin{gather}
E^2 (w)= \frac{1}{q-1} E_{2}\big(q^{-1}w\big) + \frac{q}{q-1 } E_{2}(w) + :E(w) E(w):. \label{eq:E^2andE_2}
\end{gather}
Let us comment on deep meaning of formula~\eqref{eq:E^2andE_2}. One can present algebra $\qD$ using cur\-rents~$E_k(z)$ (currents of Lie algebra type) or $E^k(z)$ (currents of $q$-$W$ algebra type). This two series of currents are connected in non-trivial way starting from $k=2$. For $k=2$ they are related by~\eqref{eq:E^2andE_2}. For general $k$ see formula~(7.17) in~\cite{N16}.
\end{Example}

\begin{Proposition} \label{prop:ass&comm}
Regular product is $($super$)$ commutative and associative
\begin{gather*}
A(z)B(az)= (-1)^{\epsilon} B(az) A(z),\\
A(a_1z) ( B(a_2z) C(a_3z) ) = ( A(a_1z) B(a_2z) ) C(a_3z)
\end{gather*}
$($of course we assume that these regular products are well defined$)$.
\end{Proposition}

\begin{Proposition} \label{prop:deriv}
Let $A(z)$, $B(z)$, and $a\in \mathbb{C}$ be as in Definition~{\rm \ref{def: the product}}. Then
\begin{gather*}
\partial_z ( A(z) B(az) )= A'(z) B(az) + aA(z) B'(az).
\end{gather*}
\end{Proposition}
\begin{proof}
Let $f(z,w)$ be a polynomial such that following power series in two variables are regular
\begin{gather*}
 ( \partial_z f(z,w) ) A(z) B(w), \qquad ( \partial_w f(z,w) ) A(z) B(w), \\ f(z,w) ( \partial_z A(z) ) B(w), \qquad f(z,w) A(z) ( \partial_w B(w) ).
\end{gather*}
Moreover, assume that $f(z, az) \neq 0$. It is easy to see that such $f(z,w)$ exists.
\[
\partial_z ( A(z) B(az) )= (\partial_{z} + a \partial_{w}) \left. \frac{f(z,w) A(z) B(w)}{f(z,az)} \right|_{w=az}.
\]
One should differentiate this expression by application of Leibniz rule (and obtain six summands). Due to our assumptions, each of these summands is regular in the sense of Definition~\ref{def:regular}. Hence, one can substitute $w=az$ to each summand separately. The proof is finished by straightforward computation.
\end{proof}

As a corollary we prove formula \eqref{eq:EE^kderiv}.
\begin{proof}[Proof of~\eqref{eq:EE^kderiv}]
\begin{gather*}
 \partial_z \mathbb{E}^{(k+1)} (z,w) \big|_{z=w}= (z-qw)\big(z-q^{-1}w\big) E'(z) E^{k}(w) \\
 \hphantom{\partial_z \mathbb{E}^{(k+1)} (z,w) \big|_{z=w}=}{} + \big(2z-qw-q^{-1}w\big) E(z) E^{k}(w) \big|_{z=w},
\\
 \frac{(z-qw)^2 \big(z-q^{-1}w\big)^2}{(1-q)\big(1-q^{-1}\big)z^2} E'(z) E^{k}(w)\\
 \qquad{}{} + \frac{(z-qw)\big(z-q^{-1}w\big)}{(1-q)\big(1-q^{-1}\big)z^2}\big(2z-qw-q^{-1}w\big) E(z) E^{k}(w) \bigg|_{z=w}.
\end{gather*}
Note that each summand is regular. Hence, we are allowed substitute $z=w$ to each of them separately
\begin{gather} \label{proving 728}
 \partial_z \mathbb{E}^{(k+1)} (z,w) \big|_{z=w}= (1-q) \big(1-q^{-1}\big) w^2 E'(w) E^{k}(w) + \big(2-q-q^{-1}\big) w E^{k+1}(w).
\end{gather}
Using Propositions \ref{prop:ass&comm} and \ref{prop:deriv}, we can prove inductively that $\partial_w E^{k+1}(w)= (k+1) E'(w) E^{k}(w)$. To finish the proof, we substitute last formula into~\eqref{proving 728}.
\end{proof}

\section{Serre relation} \label{Appendix:Serre}

This appendix is devoted to detailed study of Serre relation.
\begin{gather*}
z_2 z_3^{-1} [E(z_1), [E(z_2), E(z_3)]] + \mathrm{cyclic} =0.
\end{gather*}
Let $\bE(z)$ be a current satisfying
\begin{gather} \label{bEbilinar}
(z-qw)\big(z-q^{-1}w\big) [ \bE(z),\bE(w)]=0.
\end{gather}
\begin{Remark}
Let us emphasize the difference between $E(z)$ and $\bE(z)$. Current $E(z)$ is a~current from $\qD$, but current $\bE(z)$ is just a current satisfying~\eqref{bEbilinar}. We need $\bE(z)$ to formulate equivalent conditions to Serre relations.
\end{Remark}

Define $\delta(z_1/z_2/z_3)= \sum\limits_{a+b+c=0} z_1^a z_2^b z_3^c$ for $a,b,c \in \mathbb{Z}$.
\begin{Proposition} \label{Prop:withoutSerre} There exist three currents $R_1(z)$, $R_2(z)$ and $\bE_3(z)$ such that triple commutator $[ \bE(z_1) , [ \bE(z_2), \bE(z_3) ]]$ equals to
\begin{gather}
\bE_3(z_1) \delta\big(q^2 z_1/q z_2 /z_3\big) - \bE_3(z_2) \delta\big(q^2 z_2/ qz_3/z_1 \big) - \bE_3 (z_1) \delta\big(q^2z_1/ qz_3 / z_2\big)\nonumber\\
\qquad{} + \bE_3 (z_3) \delta\big(q^2z_3/qz_2/z_1\big)+ R_1(z_1) \delta(z_1/ qz_2/ z_3) - R_1(z_1) \delta(z_1/ z_2/ qz_3)\nonumber\\
\qquad{} + R_2(z_1) \delta\big(z_1/z_2/q^{-1}z_3\big) - R_2(z_1) \big(z_1/q^{-1}z_2/z_3\big).\label{eq:general3comm}
\end{gather}
\end{Proposition}

\begin{proof}
First of all, note that condition \eqref{bEbilinar} is equivalent to existence of currents $\bE_2(z)$ and~$\bE^{\circ}_2(w)$ such that
\begin{gather} \label{eq: bar E two}
[\bE(z), \bE(w)] = \bE_2(z) \delta(w/qz) - \bE^{\circ}_2(w) \delta(z/qw).
\end{gather}
Commutator $[\bE(z), \bE(w)]$ is skew symmetric on $z$ and $w$, hence $\bE_2(z) = \bE^{\circ}_2(z)$.
Now consider triple commutator $[ \bE(z_1) , [ \bE(z_2), \bE(z_3) ]$. Jacobi identity and relation \eqref{bEbilinar} imply \begin{gather} \label{eq: triple bar E}
(z_1 -qz_2)\big(z_1-q^{-1}z_2\big)(z_1-qz_3)\big(z_1-q^{-1}z_3\big) [ \bE(z_1) , [ \bE(z_2), \bE(z_3) ] = 0.
\end{gather}
Substituting \eqref{eq: bar E two} to \eqref{eq: triple bar E}, we obtain
\begin{gather*}
(z_1 -qz_2)\big(z_1-q^{-1}z_2\big)\big(z_1-q^2 z_2\big)(z_1-z_2) [ \bE(z_1) , \bar{E}_2(z_2) ] \delta(z_3/qz_2) \\
\qquad{}+\big(z_1 -q^2 z_3\big)(z_1-z_3)(z_1-q z_3)\big(z_1-q^{-1} z_3\big) [ \bE(z_1) , \bar{E}_2(z_3) ] \delta(z_2/qz_3)= 0.
\end{gather*}
This implies that $ [ \bE(z_1) [\bE(z_2), \bE(z_3)]]$ is indeed a sum of triple delta functions with some operator coefficients as in~\eqref{eq:general3comm}; it remains to prove proposed relations on the coefficients.

Note that triple commutator $ [ \bE(z_1) , [ \bE(z_2), \bE(z_3)] ]$ is skew symmetric on $z_2$, $z_3$. Also note that the sum over cyclic permutations is zero. This implies relation~\eqref{eq:general3comm}.
\end{proof}

\begin{Proposition} \label{prop:R1R2}
Serre relation for $\bE(z)$ is equivalent to $R_1(z)=R_2(z)=0$.
\end{Proposition}

\begin{proof}
Straightforward computation.
\end{proof}

\subsection[Operator product expansion for $E(w_1) \cdots E(w_k)$]{Operator product expansion for $\boldsymbol{E(w_1) \cdots E(w_k)}$}

One can find reformulation of Serre relation in terms of OPE in \cite[Section~3.3]{FJMM}.

\begin{Proposition}
Formal power series in three variables $\bE (z_1) \bE (z_2) \bE (z_3)$ can be presented as sum some of regular part and singular part. Regular part is some regular power series in three variables, singular part has a form
\begin{gather*}
\sum_{\epsilon_1, \epsilon_2 = \pm 1} \left( \left(1-q^{\epsilon_1}\frac{z_2}{z_1}\right)^{-1} \left(1-q^{\epsilon_2}\frac{z_3}{z_1}\right)^{-1} R^{(1)}_{\epsilon_1, \epsilon_2} (z_3)\right. \\
\qquad{}+\left(1-q^{\epsilon_1}\frac{z_2}{z_1}\right)^{-1} \left(1-q^{\epsilon_2}\frac{z_3}{z_2}\right)^{-1} R^{(2)}_{\epsilon_1, \epsilon_2} (z_3)\\
\left.\qquad{} +\left(1-q^{\epsilon_1}\frac{z_3}{z_1}\right)^{-1} \left(1-q^{\epsilon_2}\frac{z_3}{z_2}\right)^{-1} R^{(3)}_{\epsilon_1, \epsilon_2} (z_3) \right)+ \sum_{\epsilon = \pm 1} \left( \left(1-q^{\epsilon}\frac{z_3}{z_2}\right)^{-1} R^{(1)}_{\epsilon} (z_1, z_3) \right.\\
\left. \qquad {} +\left(1-q^{\epsilon}\frac{z_3}{z_1}\right)^{-1} R^{(2)}_{\epsilon} (z_2, z_3)+\left(1-q^{\epsilon}\frac{z_2}{z_1}\right)^{-1} R^{(3)}_{\epsilon} (z_1, z_3) \right) + \mathrm{reg}.
\end{gather*}
\end{Proposition}

\begin{proof}
Denote $G(z_1, z_2, z_3) := \prod\limits_{i<j} (z_i - q z_j)\big(z_i - q^{-1} z_j\big) E(z_1) E(z_2) E(z_3)$. Relation~\eqref{bEbilinar} yields $G(z_1, z_2, z_3)$ to be regular.
\end{proof}

\begin{Proposition} \label{prop:SerreOPE}Serre relation for $\bE(z)$ is equivalent to condition that singular part of $E(z_1) E(z_2) E(z_3) $ restricted to $z_1=z_3$ has no poles of order greater than~$1$.
\end{Proposition}

\begin{proof}Note that second order pole can appear only from terms of form
\begin{gather*}
\left(1-q^{ \epsilon} \frac{z_2}{z_1} \right)^{-1} \left(1-q^{-\epsilon}\frac{z_3}{z_2} \right)^{-1} R^{(2)}_{\epsilon , -\epsilon} (z_3)
\end{gather*}
for $\epsilon = \pm 1$. Note that these two poles can not cancel because they are at the different points $z_1 = q^{\epsilon} z_2$. On the other hand,$R_1(z) = R^{(2)}_{1 , -1} (z)$ and $R_2(z) = R^{(2)}_{-1 , 1} (z)$. Application of Proposition~\ref{prop:R1R2} completes the proof.
\end{proof}

\begin{Corollary} \label{Corollary:EE^k}$E(z) E^k(w)$ has no poles of order greater than one. Poles may appear only at points $z=q^{\pm 1} w$.
\end{Corollary}

\begin{proof}
To study $E(z) E(w)^k$ we will consider OPE $E(z) E(w_1)\cdots E(w_{k})$ and substitute $w_i=w$. Only term $\big(z-q^{\epsilon} w_i\big)^{-1} \big(z-q^{\epsilon}w_j\big)^{-1} \cdots$ can give poles of order higher than 1 after substitution. OPE is symmetric on $z, w_1, \dots, w_{k}$ as a rational function. We will consider order $\big( E(w_i) E(z) E(w_j) \big) \cdots$. According to Proposition~\ref{prop:SerreOPE}, the term $\big(z-q^{\epsilon} w_i\big)^{-1} \big(z-q^{\epsilon}w_j\big)^{-1} \cdots$ does not appear.
\end{proof}

\section[Homomorphism from $\qD$ to $W$-algebra]{Homomorphism from $\boldsymbol{\qD}$ to $\boldsymbol{W}$-algebra} \label{Appendix:HomWD}

This appendix is devoted to proof of Proposition~\ref{prop:HomDWtw}. The proof is a straightforward check of relation from Proposition~\ref{relation}. The relations will be checked for operators
\begin{gather*}
\tH (z)= \sum_{j \neq 0} \ha_j z^{-j}, \qquad c=n, \qquad c'=n_{tw}, \\
\tE(z) = \mu \exp \left( \frac{1}{n}{\tphi_-(z)} \right) T_1(z) \exp \left( \frac{1}{n} \tphi_+ (z) \right), \\
\tF(z) = - \frac{\mu^{-1} z^{-n_{tw}}}{\big(q^{\half}-q^{-\half}\big)^2} \exp \left(- \frac{1}{n}{\tphi_-(z)} \right) T_{n-1}(z) \exp \left( -\frac{1}{n} \tphi_+ (z) \right).
\end{gather*}

\begin{Proposition}
Relations \eqref{RqDHH}, \eqref{RqDHE-F} are satisfied.
\end{Proposition}

\begin{proof}
Straightforward.
\end{proof}

\begin{Proposition}
Currents $\tE(z)$ and $\tF(z)$ satisfy relation \eqref{RqDE2}.
\end{Proposition}

\begin{proof}
It is easy to see that
\begin{gather*}
\tE(z) \tE(w) = \frac{(1-\frac{w}{z})^{2}}{(1-q\frac{w}{z})(1-q^{-1} \frac{w}{z} )} f_{1, n} (w/z) \mu^2 \\
\hphantom{\tE(z) \tE(w) =}{}\times \exp \left( \frac{1}{n} (\tphi_- (z) + \tphi_- (w)) \right) T_1 (z) T_1(w) \exp \left( \frac{1}{n} (\tphi_+ (z) + \tphi_+ (w)) \right).
\end{gather*}
Thus,
\begin{gather*}
(z-qw)\big(z-q^{-1}w\big) \big( \tE(z) \tE(w) - \tE(w) \tE(z) \big)= (z-w)^2 \exp \left( \frac{1}{n} (\tphi_- (z) + \tphi_- (w)) \right) \\
\qquad{}\times \big( f_{1,n}(w/z) T_1(z) T_1(w) - f_{1,n} (z/w) T_1(w) T_1(z)\big) \exp \left( \frac{1}{n} (\tphi_+ (z) + \tphi_+ (w)) \right)=0.
\end{gather*}
Proof for $\tF(z)$ is analogous.
\end{proof}

\begin{Proposition}Currents $\tE(z)$ and $\tF(z)$ satisfy relation \eqref{RqDEF} for $c=n$ and $c'=n_{tw}$.
\end{Proposition}

\begin{proof}It is easy to see that
\begin{gather*}
\tE(z) \tF(w) = -\frac{1}{\big(q^{\half}-q^{-\half}\big)^2} f_{n-1,n} (w/z) \exp \left( \frac{1}{n} ( \tphi_- (z) - \tphi_- (w)) \right) \\
\hphantom{\tE(z) \tF(w) =}{}\times T_1 (z) T_{n-1}(w) \exp \left( \frac{1}{n} (\tphi_+ (z) - \tphi_+ (w)) \right) w^{-n_{tw}}.
 \end{gather*}
Thus,
\begin{gather*}
\big[ \tE(z) , \tF(w) \big] = \exp \left( \frac{1}{n} \left(\tphi_- (z) - \tphi_- (w) \right) \right)\\
\hphantom{\big[ \tE(z) , \tF(w) \big] =}{}\times \left( n \frac{w}{z} \delta' \left( \frac{w}{z} \right) w^{n_{tw}} + w \delta \left(\frac{w}{z} \right) \partial_w w^{n_{tw}} \right) \exp \left( \frac{1}{n} \left(\tphi_+ (z) - \tphi_+ (w) \right) \right) w^{-n_{tw}}.
\end{gather*}
Consequently,
\begin{gather*}
\big[\tE(z) , \tF(w)\big] = n \frac{w}{z} \delta' \left(\frac{w}{z} \right) + \left( \tH\big(q^{-\half}z\big) - \tH\big(q^{\half}z\big) + n_{tw} \right) \delta \left(\frac{w}{z} \right).\tag*{\qed}
\end{gather*}\renewcommand{\qed}{}
\end{proof}
Denote by $\tbE(z,w)=(z-qw)(z-q^{-1}w) \tE(z) \tE(w)$.

\begin{Lemma} \label{App:Lemma:E^2}
$\tE^2(z) = 2 \mu^2 \exp \big( \frac{2}{n} \tphi_-(z) \big) T_2(z) \exp \big( \frac{2}{n} \tphi_+(z) \big)$.
\end{Lemma}
\begin{proof}
\begin{gather*}
f_{1,n}(w/z) T_{1}(z) T_{1} (w) - f_{1,n}(z/w) T_1(w) T_1 (z) = \mu^{-2} \exp \left( - \frac{1}{n} \left( \tphi_-(z) + \tphi_- (w) \right) \right) \\
\qquad{}\times \tbE(z,w) \exp \left( - \frac{1}{n} \left( \tphi_+(z) + \tphi_+ (w) \right) \right) \partial_w \big( w^{-1} \delta(w/z) \big) \\
{} =- \mu^{-2}\big(q^{\half}-q^{-\half}\big)^2 \exp \left(- \frac{2}{n} \tphi_-(w) \right) \tE^{2}(w) \exp \left(- \frac{2}{n} \tphi_+(w) \right) \frac{w}{z} \delta'(w/z) + \cdots \delta(w/z).
\end{gather*}
On the other hand, l.h.s.\ can be found from relation \eqref{relation:currentTW1}. Comparing coefficients of $\delta'(w/z)$ completes the proof.
\end{proof}

\begin{Proposition}$\tE(z)$ and $\tF(z)$ satisfy Serre relation~\eqref{RqDEEE}.
\end{Proposition}

\begin{proof}Using Lemma \ref{App:Lemma:E^2}, we see that
\begin{gather*}
\tE (z) \tE^2(w) = 2\mu^3 \frac{ \big(1-\frac{w}{z}\big)^{2}}{ \big(1-q\frac{w}{z}\big) \big(1-q^{-1}\frac{w}{z}\big)}\exp \left( \frac{1}{n} ( \tphi_-(z) + 2 \tphi_- (w) ) \right) \\
\hphantom{\tE (z) \tE^2(w) =}{}\times f_{2,n}(w/z) T_{1}(z) T_{2} (w) \exp \left( \frac{1}{n} ( \tphi_+(z) + 2 \tphi_+ (w) ) \right).
\end{gather*}

Note that $(z-w)^2 f_{2,n}(w/z) T_{1}(z) T_{2} (w)$ is regular. Proposition~\ref{prop:SerreOPE} completes the proof of Serre relation for $\tE(z)$. Proof for $\tF(z)$ is analogous.
\end{proof}

\section{Whittaker vector} \label{Appendix:Whittaker}

\subsection{Uniqueness of Whittaker vector}
Recall that operators $E^k[d]$ are defined by $E^k(z) = \sum_d E^k[d] z^{-d}$.
\begin{Proposition}
Whittaker vector $W(z|u_1, \dots, u_n)$ is annihilated by $E^k[d]$ for $d>0$ and $k=1, \dots, n-1$.
\end{Proposition}

\begin{proof}
Actually we will prove that Whittaker vector is annihilated by $E^k[d]$ for $nd >k$. To do this we will need \cite[equation~(7.17)]{N16}. Let us rewrite this with respect to our notation
\begin{gather}
E^k[d] = \sum_{t \geqslant 1} (-1)^{k-t} \sum_{v= \{ \frac{d_1}{k_1} \leqslant \frac{d_2}{k_2} \leqslant \dots \leqslant \frac{d_t}{k_t} \} }^{\substack{k_i \in \mathbb{N}, \ d_i \in \mathbb{Z};\\k_1 + \cdots + k_t = k\\ d_1 + \cdots + d_t = d}} c_v E_{k_1,d_1} \cdots E_{k_t, d_t}. \label{Negut formula}
\end{gather}

Here $c_v$ denotes some combinatorially defined coefficient, which is not quite important for us. Inequality $nd >k$ implies $nd_t > k_t$; therefore $E_{k_t, d_t} W(z|u_1, \dots, u_n) =0$. So, any summand of r.h.s.\ annihilates Whittaker vector.
\end{proof}

Let us denote $\Wgl := \mathcal{W}_{q}(\mathfrak{gl}_n,0)= U(\qD)/J_{n,0}$. Denote Verma module for $\Wgl$ by $\VerglNT$ (cf. Definition \ref{definition: Verma}).

\begin{Definition}For each graded $\qD$ module $M$ let us define Shapovalov dual module $M^{\vee}$. As a vector space $M^{\vee}$ is graded dual to~$M$. Action is defined by requirement that canonical pairing $M^{\vee} \otimes M \rightarrow \mathbb{C}$ is Shapovalov.
\end{Definition}

Finally note that involution $E_{a,b} \mapsto E_{-a,-b}$ maps ideal $J_{n,0}$ to $J_{n,0}$ (maybe with different $\mu$). Hence if $M$ is a $\Wgl$-module then so is $M^{\vee}$.

\begin{Proposition} \label{cor:Whit uniquness}
Let $u_i/u_j \neq q^k$ for any $k \in \mathbb{Z}$ $($cf.\ Lemma {\rm \ref{lemma 3.1 FJMM})}. There is no more than one Whittaker vector $W(z|u_1, \dots, u_N) \in \mathcal{F}_{u_1} \otimes \cdots \otimes \mathcal{F}_{u_N}$.
\end{Proposition}
\begin{proof}Denote by $\mathfrak{n}$ a subalgebra of $\Wgl$ generated by $E^k[d]$ and $H_j$ for $k=1, \dots, n-1$, $d > 0$ and $j>0$. Analogously, let $\mathfrak{n}^{\vee}$ be a subalgebra of $\Wgl$ generated by $F^k[-d]$ and $H_{-j}$ for $k=1, \dots, n-1$, $d > 0$ and $j>0$. Note that involution $E_{a,b} \mapsto E_{-a,-b}$ induces an involoution on $\Wgl$ which swaps $\mathfrak{n}$ and $\mathfrak{n}^{\vee}$.

Consider $\mathcal{F}_{u_1} \otimes \cdots \otimes \mathcal{F}_{u_n}$ as a $\Wgl$-module. Whittaker vector is an eigenvector for $\mathfrak{n}$. Hence, it is enough to show that $\mathcal{F}_{u_1} \otimes \cdots \otimes \mathcal{F}_{u_n}$ is cocyclic for $\mathfrak{n}$. Equivalently, we need to prove that Shapovalov dual module $ ( \mathcal{F}_{u_1} \otimes \cdots \otimes \mathcal{F}_{u_n})^{\vee} \cong \mathcal{F}_{q/u_n} \otimes \cdots \otimes \mathcal{F}_{q/u_1}$ is cyclic for $\mathfrak{n}^{\vee}$.

Fock module $\mathcal{F}_{q/u_n} \otimes \cdots \otimes \mathcal{F}_{q/u_1}$ is isomorphic to Verma module $\VerglNT$ for corresponding~$\lambda_j$ by Theorem \ref{Th: Verma vs Fock}. Verma module $\VerglNT$ is cyclic for $\mathfrak{n}^{\vee}$ by an analogue of Proposition~\ref{prop: cyclic Verma} for~$F^k[d]$.
\end{proof}

\subsection{Construction of Whittaker vector} \label{section: construction of Whittaker}
Let $(n_1', n_1)$ and $(n_2', n_2)$ be a basis of $\mathbb{Z}^2$.
\begin{Theorem}[\cite{AFS}] \label{Th: AFS} There exist homomorphisms
\begin{gather*}
\Phi\colon \ \mathcal{F}_{u}^{(n'_1,n_1)} \otimes \mathcal{F}_{v}^{(n_2',n_2)} \rightarrow \mathcal{F}_{- q^{-1/2} uv}^{(n_1'+n_2', n_1+n_2)}, \\
\Phi^*\colon \ \mathcal{F}_{- q^{-1/2} uv}^{(n_1'+n_2', n_1+n_2)} \rightarrow \mathcal{F}_{u}^{(n'_1,n_1)} \otimes \mathcal{F}_{v}^{(n_2',n_2)}.
\end{gather*}
These homomorphisms are defined uniquely up to normalization.
\end{Theorem}

\begin{Remark} \label{remark: inf sum}
Actually operators $\Phi$ and $\Phi^*$ maps to graded completion of $\mathcal{F}_{-q^{-1/2} uv}^{(n_1'+n_2', n_1+n_2)}$ and $\mathcal{F}_{u}^{(n'_1,n_1)} \otimes \mathcal{F}_{v}^{(n_2',n_2)}$ correspondingly. Abusing notation, we will use the same symbol for a module and its completion. Moreover, we are going to consider a composition of such $\Phi^*$; there appear an infinite sum as a result of such composition (a priori this sum does not make sense). We will use a calculus approach to infinite sums; below we will provide a sufficient condition for convergence of the series.
\end{Remark}

Denote by $\Phi^*_{\mu}(u)$ component of $\Phi^*$ corresponding to $| \mu \rangle \in \mathcal{F}_{u}^{(n'_2,n_2)}$. More precisely, for any $x$ in $\mathcal{F}_{- q^{-1/2} uv}^{(n_1'+n_2', n_1+n_2)}$
\begin{gather*}
\Phi^* \cdot x = \sum_{\mu} | \mu \rangle \otimes ( \Phi_{\mu}^*(u) \cdot x ).
\end{gather*}
To simplify our notation we will consider particular case
\begin{gather*}
\Phi^* \colon \ \mathcal{F}_{- q^{-1/2} uv}^{(1, -k)} \rightarrow \mathcal{F}_{u}^{(0,1)} \otimes \mathcal{F}_{v}^{(1,-k-1)}.
\end{gather*}
Note that both $\mathcal{F}_{- q^{-1/2} uv}^{(1, -k)}$ and $\mathcal{F}_{v}^{(1,-k-1)}$ are Fock modules for Heisenberg algebra generated by $a_k = E_{k,0}$.
\begin{Proposition}[\cite{AFS}] Operator $\Phi_{\mu}^* (u)$ is defined by following explicit formulas
\begin{gather*}
\Phi_{\varnothing}^* (u) = {: \! \exp \left( \sum_{k \neq 0} \frac{u^{-k}}{k\big(1-q^{-k}\big)} a_{k} \right) \! :},\\
\Phi_{\mu}^* (u) \sim {: \! \Phi_{\varnothing}^* (u) \prod_{s \in \lambda} \! F\big(q^{c(s)-\frac12} u\big):}.
\end{gather*}
Here sign $\sim$ means up to multiplication by a number. Recall
\begin{gather*}
F(z) \sim {: \! \exp \left( \sum_k \frac{ q^{k/2}-q^{-k/2} }{k} a_k z^{-k} \right) \! :.}
\end{gather*}
\end{Proposition}

\begin{Corollary} \label{corol: Phi norm order}Heisenberg normal ordering is given by
\begin{gather*}
\Phi^*_{\lambda}(u) \Phi^*_{\mu}(v) = \frac{f_{\lambda, \mu}(v/u)}{(q v/u; q, q)_{\infty}} : \! \Phi^*_{\lambda}(u) \Phi^*_{\mu}(v) \! :
\end{gather*}
for some rational function $f_{\lambda, \mu}(v/u) = \frac{\prod_{i} (1-q^{k_i} v/u) }{\prod_{j} (1-q^{l_j} v/u)}$; here $l_j$ and $k_i$ are integer numbers.
\end{Corollary}

Consider a homomorphism
\begin{gather*}
\tilde{\Phi}_0 \colon \ \mathcal{F}^{(1,0)}_{q^{1/2} z^{-1}} \otimes \mathcal{F}^{(-1,n)}_{q^{1/2} z (-q^{-1/2})^n u_1 \cdots u_n} \\ \qquad{} \rightarrow\big( \mathcal{F}^{(0,1)}_{u_1} \otimes \cdots \otimes \mathcal{F}^{(0,1)}_{u_n} \big) \otimes \mathcal{F}^{(1,-n)}_{q^{1/2} (-q^{1/2})^{n} (z u_1 \dots u_{n})^{-1} } \otimes \mathcal{F}^{(-1,n)}_{q^{1/2} z (-q^{-1/2})^n u_1 \cdots u_n}
\end{gather*}
obtained as composition of
\begin{gather*}
\id^{\otimes k} \otimes \Phi^* \otimes \id \colon\ \big( \mathcal{F}^{(0,1)}_{u_{1}} \otimes \cdots \otimes \mathcal{F}^{(0,1)}_{u_{k}} \big) \otimes \mathcal{F}^{(1,-k)}_{q^{1/2} (-q^{1/2})^k (z u_1 \dots u_k)^{-1} } \otimes \mathcal{F}^{(-1,n)}_{q^{1/2} z (-q^{-1/2})^n u_1 \cdots u_n} \\
\qquad{} \rightarrow\big( \mathcal{F}^{(0,1)}_{u_{1}} \otimes \cdots \otimes \mathcal{F}^{(0,1)}_{u_{k+1}} \big) \otimes \mathcal{F}^{(1,-k-1)}_{q^{1/2} (-q^{1/2})^{k+1} (z u_1 \dots u_{k+1})^{-1} } \otimes \mathcal{F}^{(-1,n)}_{q^{1/2} z (-q^{-1/2})^n u_1 \cdots u_n} \rightarrow.
\end{gather*}

\begin{Lemma}There exists a unique invariant pairing $\mathcal{F}_{u} \otimes \mathcal{F}_{q u^{-1}} \rightarrow \mathbb{C}$ such that $\langle 0 | 0 \rangle =1$.
\end{Lemma}
\begin{proof}
This is equivalent to Proposition \ref{prop: Sapovalov p for Fock}.
\end{proof}

Composition of $\tilde{\Phi}_0$ and the pairing gives a homomorphism
\begin{gather*}
\tilde{\Phi}_1 \colon \ \mathcal{F}^{(1,0)}_{q^{\half}z^{-1}} \otimes \mathcal{F}^{(-1,n)}_{q^{\half} u_1 \cdots u_n z (-q^{-1/2})^n} \rightarrow \mathcal{F}^{(0,1)}_{u_1} \otimes \cdots \otimes \mathcal{F}^{(0,1)}_{u_n} . %\label{Whit map}
\end{gather*}
Let us reformulate above inductive procedure via an explicit formula
\begin{gather*}
\tilde{\Phi}_1 ( | \lambda_1 \rangle \otimes | \lambda_2 \rangle ) = \sum \langle \lambda_2 | \Phi_{\mu_n}^*(u_n) \cdots \Phi_{\mu_1}^*(u_1) | \lambda_1 \rangle \, | \mu_1 \rangle \otimes \dots \otimes | \mu_n \rangle .
\end{gather*}

As we warned in Remark~\ref{remark: inf sum}, operator $\tilde{\Phi}_1$ is not a priori well defined. However, the series (appearing from the composition) converges in a domain $|u_1| \ll |u_{2}| \ll \dots \ll |u_n|$. This assertion follows from a formula
\begin{gather*}
\langle \lambda_2 | \Phi_{\mu_n}^*(u_n) \dots \Phi_{\mu_1}^*(u_1) | \lambda_1 \rangle = \prod_{i<j} \frac{f_{\mu_i, \mu_j}(u_i/u_j)}{\left(q u_i/u_j; q, q \right)_{\infty}} \,\langle \lambda_1 | :\!\Phi_{\mu_n}^*(u_n) \dots \Phi_{\mu_1}^*(u_1)\!: | \lambda_2 \rangle.
\end{gather*}
Moreover, one can consider analytic continuation of obtained function given by r.h.s.\ of the formula. Corollary \ref{corol: Phi norm order} implies that we can extend the domain to $u_i/u_j \neq q^k$ for any $k \in \mathbb{Z}$. Evidently, analytic continuation also enjoys intertwiner property. Hence we obtained following proposition

\begin{Proposition}If $u_i/u_j \neq q^k$ for any $k \in \mathbb{Z}$, then there is an intertwiner
\begin{gather} \label{Whit map2}
\tilde{\Phi}\colon \ \mathcal{F}^{(1,0)}_{q^{1/2}z^{-1}} \otimes \mathcal{F}^{(-1,n)}_{q^{1/2} u_1 \cdots u_n z (-q^{-1/2})^n} \rightarrow \mathcal{F}^{(0,1)}_{u_1} \otimes \cdots \otimes \mathcal{F}^{(0,1)}_{u_n}.
\end{gather}
\end{Proposition}

Denote the highest vector of $\mathcal{F}_u^{(n_1,n_2)}$ by $|n_1,n_2 \rangle$.

\begin{Theorem} \label{th:Whit construction}
Whittaker vector $W(z|u_1, \dots, u_n) \in \mathcal{F}^{(0,1)}_{u_1} \otimes \cdots \otimes \mathcal{F}^{(0,1)}_{u_n}$ can be constructed via homomorphism $\tilde{\Phi}$ $($as in \eqref{Whit map2}$)$
\begin{gather*}
W(z|u_1, \dots, u_n) := \prod_{i<j} (q u_i/u_j; q, q )_{\infty} \tilde{\Phi} ( |1,0 \rangle \otimes |-1, n \rangle ).
\end{gather*}
\end{Theorem}

\begin{proof}Follows from \eqref{eq:eigen vac}.
\end{proof}

\begin{Remark}Recall that existence of Whittaker vector can be seen from geometric construction (see~\cite{N12} and \cite{Ts}).
\end{Remark}

\begin{proof}[Proof of Theorem \ref{th:Whittaker condition}]
Existence and uniqueness follows from Theorem~\ref{th:Whit construction} and Proposition~\ref{cor:Whit uniquness} correspondingly.
\end{proof}

\subsection[Whittaker vector for $\W$ algebra]{Whittaker vector for $\boldsymbol{\W}$ algebra}
Let us define coefficient $\bar{c}_{1/m}$ by (cf.\ \eqref{Negut formula})
\begin{equation} \label{eq: notation for c}
E^m[1]= \bar{c}_{1/m} E_{m,1} + \cdots.
\end{equation}
Also recall that for $\mathcal{F}_{u_1} \otimes \dots \otimes \mathcal{F}_{u_n}$
\begin{gather*}
\mu = \frac{1}{1-q} (u_1 \cdots u_n)^{\frac{1}{n}}.
\end{gather*}
\begin{Definition}
For $m=1, \dots, n-1$ Whittaker vector $W_m^{\mathfrak{sl}_n}(z| u_1, \dots, u_n) \in \nWFock$ with respect to $\W$ is an eigenvector for $T_k[r]$ (for $k= 1, \dots, n-1$ and $r \geqslant 0$) with eigenvalues given by
\begin{gather*}
T_{m}[1] W_m^{\mathfrak{sl}_n}(z| u_1, \dots, u_n) = \frac{\big( {-}q^{\half} \big)^nu_1 \cdots u_n z}{q^{-1/2}-q^{1/2}} \frac{\bar{c}_{1/m} \mu^{-m}}{m!} W_m^{\mathfrak{sl}_n}(z| u_1, \dots, u_n), \\
T_{k}[1] W_m^{\mathfrak{sl}_n}(z| u_1, \dots, u_n) =0 \qquad \text{for $k \neq m$},\\
T_{k}[r] W_m^{\mathfrak{sl}_n}(z| u_1, \dots, u_n) =0 \qquad \text{for $r \geqslant 2$}.
\end{gather*}
We require $W_m^{\mathfrak{sl}_n}(z| u_1, \dots, u_n) = \vacF + \cdots$ to fix normalization (by dots we mean lower vectors).
\end{Definition}
One can find notion of Whittaker vector for $\W$ in the literature (see \cite{T}). In this section we will explain connection between notion of Whittaker vector $W_m^{\mathfrak{sl}_n}(z| u_1, \dots, u_n)$ and Whittaker vector $W(z| u_1, \dots , u_n)$ for $\qD$ (see Definition~\ref{def: Whit eigen}). Our plan to explain this connection is as follows. First we define Whittaker vector with respect to $\Wgl$ (we denote it by $W_m^{\mathfrak{gl}_n}(z| u_1, \dots, u_n)$). Then we will see, that on the one hand, the vector $W_m^{\mathfrak{gl}_n}(z| u_1, \dots, u_n)$ is connected with $W_m^{\mathfrak{sl}_n}(z| u_1, \dots, u_n)$; on the other hand it is connected with $W(z| u_1, \dots , u_n)$.

Recall, that $\Wgl = \qD / J_{n,0}$; ideal $J_{n,0}$ annihilates $\mathcal{F}_{u_1} \otimes \dots \otimes \mathcal{F}_{u_n}$. Hence $\mathcal{F}_{u_1} \otimes \dots \otimes \mathcal{F}_{u_n}$ is a representation of $\Wgl$.
\begin{Definition}
For $m=1, \dots, n-1$ Whittaker vector $W_m^{\mathfrak{gl}_n}(z| u_1, \dots, u_n)$ is a vector belonging to $\mathcal{F}_{u_1} \otimes \dots \otimes \mathcal{F}_{u_n}$ and satisfying following conditions
\begin{gather*}
H_k W_m^{\mathfrak{gl}_n}(z| u_1, \dots, u_n) = 0 \qquad \text{for $k > 0$},\\
E^{m}[1] W_m^{\mathfrak{gl}_n}(z| u_1, \dots, u_n) = \frac{\big({-}q^{\half} \big)^nu_1 \cdots u_n z}{q^{-1/2}-q^{1/2}} \bar{c}_{1/m} W_m^{\mathfrak{gl}_n}(z| u_1, \dots, u_n),\\
E^{k}[1] W_m^{\mathfrak{gl}_n}(z| u_1, \dots, u_n) =0 \qquad \text{for $k < m$},\\
E^{k}[r] W_m^{\mathfrak{gl}_n}(z| u_1, \dots, u_n) =0 \qquad \text{for $r \geqslant 2$ and $k \leqslant m$}, \\
F^k[r] W_m^{\mathfrak{gl}_n}(z| u_1, \dots, u_n) =0 \qquad \text{for $r \geqslant 1$ and $k < n-m$}.
\end{gather*}
We require $W_m^{\mathfrak{gl}_n}(z| u_1, \dots, u_n) = \vacF + \cdots$ to fix normalization (by dots we mean lower vectors).
\end{Definition}

Recall that $\mathcal{F}_{u_1} \otimes \dots \otimes \mathcal{F}_{u_n} \cong \mathcal{F}_{u_1, \dots, u_n}^{\W} \otimes F^{H}$ with respect algebra identification $\Wgl \cong \W \UH$.

\begin{Lemma} \label{prop: Lemma sl vs gl}
Vector satisfies properties of $W_m^{\mathfrak{gl}_n}(z| u_1, \dots, u_n)$ iff it is $W_m^{\mathfrak{sl}_n}(z| u_1, \dots, u_n) \otimes \vac_H$.
\end{Lemma}

\begin{proof}Note that
\begin{gather*}
E^k (t) \big( W_m^{\mathfrak{sl}_n}(z| u_1, \dots, u_n) \otimes \vac_H \big) = \frac{k!}{\mu^{-k}} T_k (t)   W_m^{\mathfrak{sl}_n}(z| u_1, \dots, u_n) \otimes \exp \left(\frac{k}{n} \varphi_- (t) \right)\! \vac_H,\\
F^k (t) \big( W_m^{\mathfrak{sl}_n}(z| u_1, \dots, u_n) \otimes \vac_H \big)  = \frac{k!}{\mu^{k}} T_{n-k} (t)   W_m^{\mathfrak{sl}_n}(z| u_1, \dots, u_n) \otimes \exp \left(-\frac{k}{n} \varphi_- (t) \right) \!\vac_H.
\end{gather*}
Moreover $\varphi_-(t)$ has only terms of positive degree in~$t$. Hence we expressed action of $E^k[l]$ via $\T_k[s]$ for $s \geqslant l$. Therefore we have proven that $W_m^{\mathfrak{sl}_n}(z| u_1, \dots, u_n) \otimes \vac_H$ satisfies property of $W_m^{\mathfrak{gl}_n}(z| u_1, \dots, u_n)$.

The implication in opposite direction is analogous.
\end{proof}

\begin{Proposition} \label{prop: unique for m gl}
There exists at most one vector $W_m^{\mathfrak{gl}_n}(z| u_1, \dots, u_n) \in \mathcal{F}_{u_1} \otimes \dots \otimes \mathcal{F}_{u_n}$.
\end{Proposition}

\begin{proof}Analogous to proof of Proposition \ref{cor:Whit uniquness}. The only difference is that we consider a different character of subalgebra $\mathfrak{n}$.
\end{proof}

We need to generalize notion of Whittaker vector for $\qD$ to compare it with $W_m^{\mathfrak{sl}_n}(z| u_1,\allowbreak \dots, u_n)$.

\begin{Definition}For any $m \in \mathbb{Z}$, Whittaker vector $W_m (z|u_1, \dots, u_n) \in \mathcal{F}_{u_1} \otimes \cdots \otimes \mathcal{F}_{u_n}$ is an eigenvector of operators $E_{a,b}$ for $mb \geqslant a \geqslant -(n-m)b$ and $b > 0$. More precisely,
\begin{gather*}
E_{-(n-m)k,k}  W_m (z|u_1, \dots, u_n) = \frac{z^k}{q^{k/2}-q^{-k/2}} W_m (z|u_1, \dots, u_n), \\ %\label{eq: Whit m eigen1} \\
E_{mk,k}  W_m (z|u_1, \dots, u_n) = \frac{\big(\big( {-}q^{\half} \big)^nu_1 \cdots u_n z\big)^{k}}{q^{-k/2}-q^{k/2}} W_m(z|u_1, \dots, u_n) %\label{eq: Whit m eigen2}
\end{gather*}
for $k>0$,
\begin{gather*}
E_{k_1, k_2} W_m(z) =0
\end{gather*}
for $(n-m)k_2 > k_1 > -mk_2$ and $k_2 > 0$.We require $W_m(z|u_1, \dots, u_n) = \vac \otimes \cdots \otimes \vac + \cdots$ to fix normalization (by dots we mean lower vectors).
\end{Definition}

Recall that we have defined operator $\I_{\tau} \in \End(\mathcal{F}_{u})$ by~\eqref{eq: def I tau}. By Proposition~\ref{prop: I sigma} the operator enjoys intertwiner property $\I_{\tau} \rho(E_{a,b}) \I_{\tau}^{-1} = \rho(E_{a-b,b})$. Denote $\I_{\tau, n} = \I_{\tau} \otimes \dots \otimes \I_{\tau} \in \End(\mathcal{F}_{u_1} \otimes \dots \otimes \mathcal{F}_{u_n})$. Note that $I_{\tau, n}$ also enjoys intertwiner property $\I_{\tau} \rho_n(E_{a,b}) \I_{\tau}^{-1} = \rho_n(E_{a-b,b})$ (here $\rho_n$ denotes the homomorphism of the representation $\rho_n \colon \qD \rightarrow \End(\mathcal{F}_{u_1} \otimes \dots \otimes \mathcal{F}_{u_n}) $).

\begin{Proposition}$W_m (z|u_1, \dots, u_n) = \I^{n-m}_{\tau,n} \, W (z|u_1, \dots, u_n)$.
\end{Proposition}

\begin{Corollary} \label{corollary: uniq Whit sl}There exists unique $W_{m} (z|u_1, \dots, u_n)$ if $u_i/u_j \neq q^k$.
\end{Corollary}

\begin{Proposition} \label{prop: Whit gl vs qD}There exists unique vector $W_m^{\mathfrak{gl}_n}(z| u_1, \dots, u_n)$. Moreover, $W_m^{\mathfrak{gl}_n}(z| u_1, \dots,\allowbreak u_n) = W_{m} (z|u_1, \dots, u_n)$.
\end{Proposition}
\begin{proof}We already know uniqueness of $W_m^{\mathfrak{gl}_n}(z| u_1, \dots, u_n)$ and existence of $W_m (z|u_1, \dots, u_n)$ from Proposition~\ref{prop: unique for m gl} and Corollary~\ref{corollary: uniq Whit sl} correspondingly. So it is sufficient to show that $W_{m} (z|u_1, \dots, u_n)$ satisfies properties of $W_m^{\mathfrak{gl}_n}(z| u_1, \dots, u_n)$. Last assertion follows from formula~\eqref{Negut formula} (also see~\eqref{eq: notation for c}).
\end{proof}

\begin{Theorem}There exists unique vector $W_m^{\mathfrak{sl}_n}(z| u_1, \dots, u_n)$. Moreover, $W_m^{\mathfrak{sl}_n}(z| u_1, \dots, u_n) \otimes \vac_H = W_{m} (z|u_1, \dots, u_n)= \I^{n-m}_{\sigma,n}  W (z|u_1, \dots, u_n)$.
\end{Theorem}
\begin{proof}Follows from Lemma~\ref{prop: Lemma sl vs gl} and Proposition~\ref{prop: Whit gl vs qD}.
\end{proof}

\subsection*{Acknowledgments} We are grateful to B.~Feigin, P.~Gavrylenko, E.~Gorsky, A.~Negu\c{t}, J.~Shiraishi, for interest to our work and discussions. The work is partially supported by Russian Foundation of Basic Research under grant mol\_a\_ved 18-31-20062 and by the HSE University Basic Research Program jointly with Russian Academic Excellence Project `5-100’. R.G.\ was also supported in part by Young Russian Mathematics award. The results of Section \ref{section: conformal} are obtained under the support of the Russian Science Foundation under grant 19-11-00275.

\pdfbookmark[1]{References}{ref}
\LastPageEnding


\begin{thebibliography}{99}
\footnotesize\itemsep=0pt

\bibitem{AFS}
Awata H., Feigin B., Shiraishi J., Quantum algebraic approach to refined
 topological vertex, \href{https://doi.org/10.1007/JHEP03(2012)041}{\textit{J.~High Energy Phys.}} \textbf{2012} (2012),
 no.~3, 041, 35~pages, \href{https://arxiv.org/abs/1112.6074}{arXiv:1112.6074}.

\bibitem{AY10}
Awata H., Yamada Y., Five-dimensional {AGT} conjecture and the deformed
 {V}irasoro algebra, \href{https://doi.org/10.1007/JHEP01(2010)125}{\textit{J.~High Energy Phys.}} \textbf{2010} (2010),
 no.~1, 125, 11~pages, \href{https://arxiv.org/abs/0910.4431}{arXiv:0910.4431}.

\bibitem{BGHT}
Bergeron F., Garsia A.M., Haiman M., Tesler G., Identities and positivity
 conjectures for some remarkable operators in the theory of symmetric
 functions, \href{https://doi.org/10.4310/MAA.1999.v6.n3.a7}{\textit{Methods Appl. Anal.}} \textbf{16} (1999), 363--420, \href{https://arxiv.org/abs/alg-geom/9705013}{arXiv:alg-geom/9705013}.

\bibitem{BGMtw}
Bershtein M., Gavrylenko P., Marshakov A., Twist-field representations of
 {W}-algebras, exact conformal blocks and character identities,
 \href{https://doi.org/10.1007/jhep08(2018)108}{\textit{J.~High Energy Phys.}} \textbf{2018} (2018), no.~8, 108, 55~pages,
 \href{https://arxiv.org/abs/1705.00957}{arXiv:1705.00957}.

\bibitem{BGM}
Bershtein M., Gavrylenko P., Marshakov A., Cluster {T}oda lattices and
 {N}ekrasov functions, \href{https://doi.org/10.1134/S0040577919020016}{\textit{Theoret. and Math. Phys.}} \textbf{198} (2019),
 157--188, \href{https://arxiv.org/abs/1804.10145}{arXiv:1804.10145}.

\bibitem{BG}
Bershtein M., Gonin R., Twisted and non-twisted deformed {V}irasoro algebra via
 vertex operators of ${U}_q(\widehat{\mathfrak{sl}}_2)$, \href{https://arxiv.org/abs/2003.12472}{arXiv:2003.12472}.

\bibitem{BS:2014}
Bershtein M., Shchechkin A., Bilinear equations on {P}ainlev\'e {$\tau$}
 functions from {CFT}, \href{https://doi.org/10.1007/s00220-015-2427-4}{\textit{Comm. Math. Phys.}} \textbf{339} (2015),
 1021--1061, \href{https://arxiv.org/abs/1406.3008}{arXiv:1406.3008}.

\bibitem{BS:2016:2}
Bershtein M., Shchechkin A., B\"acklund transformation of {P}ainlev\'e {${\rm
 III}(D_8)$} {$\tau$} function, \href{https://doi.org/10.1088/1751-8121/aa59c9}{\textit{J.~Phys.~A: Math. Theor.}} \textbf{50}
 (2017), 115205, 31~pages, \href{https://arxiv.org/abs/1608.02568}{arXiv:1608.02568}.

\bibitem{BS:2016:1}
Bershtein M., Shchechkin A., {$q$}-deformed {P}ainlev\'e {$\tau$} function and
 {$q$}-deformed conformal blocks, \href{https://doi.org/10.1088/1751-8121/aa5572}{\textit{J.~Phys.~A: Math. Theor.}}
 \textbf{50} (2017), 085202, 22~pages, \href{https://arxiv.org/abs/1608.02566}{arXiv:1608.02566}.

\bibitem{BS:2018}
Bershtein M., Shchechkin A., Painlev\'{e} equations from {N}akajima-{Y}oshioka
 blowup relations, \href{https://doi.org/10.1007/s11005-019-01198-4}{\textit{Lett. Math. Phys.}} \textbf{109} (2019), 2359--2402,
 \href{https://arxiv.org/abs/1811.04050}{arXiv:1811.04050}.

\bibitem{BGT17}
Bonelli G., Grassi A., Tanzini A., Quantum curves and {$q$}-deformed
 {P}ainlev\'e equations, \href{https://doi.org/10.1007/s11005-019-01174-y}{\textit{Lett. Math. Phys.}} \textbf{109} (2019),
 1961--2001, \href{https://arxiv.org/abs/1710.11603}{arXiv:1710.11603}.

\bibitem{Burban}
Burban I., Schiffmann O., On the {H}all algebra of an elliptic curve,~{I},
 \href{https://doi.org/10.1215/00127094-1593263}{\textit{Duke Math.~J.}} \textbf{161} (2012), 1171--1231,
 \href{https://arxiv.org/abs/math.AG/0505148}{arXiv:math.AG/0505148}.

\bibitem{CM15}
Carlsson E., Mellit A., A proof of the shuffle conjecture, \href{https://doi.org/10.1090/jams/893}{\textit{J.~Amer.
 Math. Soc.}} \textbf{31} (2018), 661--697, \href{https://arxiv.org/abs/1508.06239}{arXiv:1508.06239}.

\bibitem{FFZ}
Fairlie D.B., Fletcher P., Zachos C.K., Trigonometric structure constants for
 new infinite-dimensional algebras, \href{https://doi.org/10.1016/0370-2693(89)91418-4}{\textit{Phys. Lett.~B}} \textbf{218}
 (1989), 203--206.

\bibitem{FFJMMa}
Feigin B., Feigin E., Jimbo M., Miwa T., Mukhin E., Quantum continuous
 {$\mathfrak{gl}_\infty$}: semiinfinite construction of representations,
 \href{https://doi.org/10.1215/21562261-1214375}{\textit{Kyoto~J. Math.}} \textbf{51} (2011), 337--364, \href{https://arxiv.org/abs/1002.3113}{arXiv:1002.3113}.

\bibitem{FFJMM}
Feigin B., Feigin E., Jimbo M., Miwa T., Mukhin E., Quantum continuous
 {$\mathfrak{gl}_\infty$}: tensor products of {F}ock modules and {${\mathcal
 W}_n$}-characters, \href{https://doi.org/10.1215/21562261-1214384}{\textit{Kyoto~J. Math.}} \textbf{51} (2011), 365--392,
 \href{https://arxiv.org/abs/1002.3113}{arXiv:1002.3113}.

\bibitem{FF}
Feigin B., Frenkel E., Quantum {$\mathcal W$}-algebras and elliptic algebras,
 \href{https://doi.org/10.1007/BF02108819}{\textit{Comm. Math. Phys.}} \textbf{178} (1996), 653--678,
 \href{https://arxiv.org/abs/q-alg/9508009}{arXiv:q-alg/9508009}.

\bibitem{FHHSY}
Feigin B., Hashizume K., Hoshino A., Shiraishi J., Yanagida S., A commutative
 algebra on degenerate {$\mathbb{CP}^1$} and {M}acdonald polynomials,
 \href{https://doi.org/10.1063/1.3192773}{\textit{J.~Math. Phys.}} \textbf{50} (2009), 095215, 42~pages,
 \href{https://arxiv.org/abs/0904.2291}{arXiv:0904.2291}.

\bibitem{FHSSY}
Feigin B., Hoshino A., Shibahara J., Shiraishi J., Yanagida S., Kernel function
 and quantum algebras, in Representation Theory and Combinatorics,
 \textit{RIMS K\^oky\^uroku.}, Vol.~1689, Res. Inst. Math. Sci. (RIMS), Kyoto,
 2010, 133--152, \href{https://arxiv.org/abs/1002.2485}{arXiv:1002.2485}.

\bibitem{FJMM}
Feigin B., Jimbo M., Miwa T., Mukhin E., Branching rules for quantum toroidal
 {$\mathfrak{gl}_n$}, \href{https://doi.org/10.1016/j.aim.2016.03.019}{\textit{Adv. Math.}} \textbf{300} (2016), 229--274,
 \href{https://arxiv.org/abs/1309.2147}{arXiv:1309.2147}.

\bibitem{FK}
Frenkel I.B., Kac V.G., Basic representations of affine {L}ie algebras and dual
 resonance models, \href{https://doi.org/10.1007/BF01391662}{\textit{Invent. Math.}} \textbf{62} (1980), 23--66.

\bibitem{FM}
Fujii S., Minabe S., A combinatorial study on quiver varieties, \href{https://doi.org/10.3842/SIGMA.2017.052}{\textit{SIGMA}}
 \textbf{13} (2017), 052, 28~pages, \href{https://arxiv.org/abs/math.AG/0510455}{arXiv:math.AG/0510455}.

\bibitem{GIL12}
Gamayun O., Iorgov N., Lisovyy O., Conformal field theory of {P}ainlev\'e~{VI},
 \href{https://doi.org/10.1007/JHEP10(2012)038}{\textit{J.~High Energy Phys.}} \textbf{2012} (2012), no.~10, 038, 25~pages,
 \href{https://arxiv.org/abs/1207.0787}{arXiv:1207.0787}.

\bibitem{GM16}
Gavrilenko P.G., Marshakov A.V., Free fermions, {$W$}-algebras, and
 isomonodromic deformations, \href{https://doi.org/10.1134/S0040577916050044}{\textit{Theoret. and Math. Phys.}} \textbf{187}
 (2016), 649--677, \href{https://arxiv.org/abs/1605.04554}{arXiv:1605.04554}.

\bibitem{GIL:2018}
Gavrylenko P., Iorgov N., Lisovyy O., Higher-rank isomonodromic deformations
 and {$W$}-algebras, \href{https://doi.org/10.1007/s11005-019-01207-6}{\textit{Lett. Math. Phys.}} \textbf{110} (2020), 327--364,
 \href{https://arxiv.org/abs/1801.09608}{arXiv:1801.09608}.

\bibitem{GL:2016}
Gavrylenko P., Lisovyy O., Fredholm determinant and {N}ekrasov sum
 representations of isomonodromic tau functions, \href{https://doi.org/10.1007/s00220-018-3224-7}{\textit{Comm. Math. Phys.}}
 \textbf{363} (2018), 1--58, \href{https://arxiv.org/abs/1608.00958}{arXiv:1608.00958}.

\bibitem{GL}
Golenishcheva-Kutuzova M., Lebedev D., Vertex operator representation of some
 quantum tori {L}ie algebras, \href{https://doi.org/10.1007/BF02100868}{\textit{Comm. Math. Phys.}} \textbf{148} (1992),
 403--416.

\bibitem{GN13}
Gorsky E., Negu\c{t} A., Refined knot invariants and {H}ilbert schemes,
 \href{https://doi.org/10.1016/j.matpur.2015.03.003}{\textit{J.~Math. Pures Appl.}} \textbf{104} (2015), 403--435,
 \href{https://arxiv.org/abs/1304.3328}{arXiv:1304.3328}.

\bibitem{GN15}
Gorsky E., Negu\c{t} A., Infinitesimal change of stable basis, \href{https://doi.org/10.1007/s00029-017-0327-5}{\textit{Selecta
 Math. (N.S.)}} \textbf{23} (2017), 1909--1930, \href{https://arxiv.org/abs/1510.07964}{arXiv:1510.07964}.

\bibitem{ILT:2015}
Iorgov N., Lisovyy O., Teschner J., Isomonodromic tau-functions from
 {L}iouville conformal blocks, \href{https://doi.org/10.1007/s00220-014-2245-0}{\textit{Comm. Math. Phys.}} \textbf{336} (2015),
 671--694, \href{https://arxiv.org/abs/1401.6104}{arXiv:1401.6104}.

\bibitem{JNS17}
Jimbo M., Nagoya H., Sakai H., C{FT} approach to the {$q$}-{P}ainlev\'e {VI}
 equation, \href{https://doi.org/10.1093/integr/xyx009}{\textit{J.~Integrable Syst.}} \textbf{2} (2017), xyx009, 27~pages,
 \href{https://arxiv.org/abs/1706.01940}{arXiv:1706.01940}.

\bibitem{KKLW}
Kac V.G., Kazhdan D.A., Lepowsky J., Wilson R.L., Realization of the basic
 representations of the {E}uclidean {L}ie algebras, \href{https://doi.org/10.1016/0001-8708(81)90053-0}{\textit{Adv. Math.}}
 \textbf{42} (1981), 83--112.

\bibitem{KR93}
Kac V.G., Radul A., Quasifinite highest weight modules over the {L}ie algebra
 of differential operators on the circle, \href{https://doi.org/10.1007/BF02096878}{\textit{Comm. Math. Phys.}}
 \textbf{157} (1993), 429--457, \href{https://arxiv.org/abs/hep-th/9308153}{arXiv:hep-th/9308153}.

\bibitem{KR}
Kac V.G., Raina A.K., Bombay lectures on highest weight representations of
 infinite-dimensional {L}ie algebras, 2nd ed., \textit{Advanced Series in Mathematical
 Physics}, Vol.~29, \href{https://doi.org/10.1142/8882}{World Sci. Publ. Co., Inc.}, Teaneck, NJ, 2013.

\bibitem{LW}
Lepowsky J., Wilson R.L., Construction of the affine {L}ie algebra
 {$A_{1}^{{}}(1)$}, \href{https://doi.org/10.1007/BF01940329}{\textit{Comm. Math. Phys.}} \textbf{62} (1978), 43--53.

\bibitem{M}
Macdonald I.G., Symmetric functions and {H}all polynomials, 2nd ed., \textit{Oxford
 Mathematical Monographs}, The Clarendon Press, Oxford University Press, New
 York, 1995.

\bibitem{MN:2018}
Matsuhira Y., Nagoya H., Combinatorial expressions for the tau functions of
 {$q$}-{P}ainlev\'e {V} and {III} equations, \href{https://doi.org/10.3842/SIGMA.2019.074}{\textit{SIGMA}} \textbf{15}
 (2019), 074, 17~pages, \href{https://arxiv.org/abs/1811.03285}{arXiv:1811.03285}.

\bibitem{M07}
Miki K., A {$(q,\gamma)$} analog of the {$W_{1+\infty}$} algebra,
 \href{https://doi.org/10.1063/1.2823979}{\textit{J.~Math. Phys.}} \textbf{48} (2007), 123520, 35~pages.

\bibitem{Na15}
Nagoya H., Irregular conformal blocks, with an application to the fifth and
 fourth {P}ainlev\'{e} equations, \href{https://doi.org/10.1063/1.4937760}{\textit{J.~Math. Phys.}} \textbf{56} (2015),
 123505, 24~pages, \href{https://arxiv.org/abs/1505.02398}{arXiv:1505.02398}.

\bibitem{Negut15}
Negu\c{t} A., Quantum algebras and cyclic quiver varieties, Ph.D.~Thesis,
 {C}olumbia University, 2015, \href{https://arxiv.org/abs/1504.06525}{arXiv:1504.06525}.

\bibitem{N12}
Negu\c{t} A., Moduli of flags of sheaves and their {$K$}-theory,
 \href{https://doi.org/10.14231/AG-2015-002}{\textit{Algebr. Geom.}} \textbf{2} (2015), 19--43, \href{https://arxiv.org/abs/1209.4242}{arXiv:1209.4242}.

\bibitem{N17}
Negu\c{t} A., {W}-algebras associated to surfaces, \href{https://arxiv.org/abs/1710.03217}{arXiv:1710.03217}.

\bibitem{N16}
Negu\c{t} A., The {$q$}-{AGT}-{W} relations via shuffle algebras, \href{https://doi.org/10.1007/s00220-018-3102-3}{\textit{Comm.
 Math. Phys.}} \textbf{358} (2018), 101--170, \href{https://arxiv.org/abs/1608.08613}{arXiv:1608.08613}.

\bibitem{Sh}
Shiraishi J., Free field constructions for the elliptic algebra
 {$\mathcal{A}_{q,p}(\widehat{\rm sl}_2)$} and {B}axter's eight-vertex model,
 \href{https://doi.org/10.1142/S0217751X0402052X}{\textit{Internat.~J. Modern Phys.~A}} \textbf{19} (2004), suppl., 363--380,
 \href{https://arxiv.org/abs/math.QA/0302097}{arXiv:math.QA/0302097}.

\bibitem{SKAO}
Shiraishi J., Kubo H., Awata H., Odake S., A quantum deformation of the
 {V}irasoro algebra and the {M}acdonald symmetric functions, \href{https://doi.org/10.1007/BF00398297}{\textit{Lett.
 Math. Phys.}} \textbf{38} (1996), 33--51, \href{https://arxiv.org/abs/q-alg/9507034}{arXiv:q-alg/9507034}.

\bibitem{T}
Taki M., On {AGT}-{W} conjecture and $q$-deformed ${W}$-algebra,
 \href{https://arxiv.org/abs/1403.7016}{arXiv:1403.7016}.

\bibitem{Ts}
Tsymbaliuk A., The affine {Y}angian of {$\mathfrak{gl}_1$} revisited,
 \href{https://doi.org/10.1016/j.aim.2016.08.041}{\textit{Adv. Math.}} \textbf{304} (2017), 583--645, \href{https://arxiv.org/abs/1404.5240}{arXiv:1404.5240}.

\bibitem{Z}
Zamolodchikov Al.B., Conformal scalar field on the hyperelliptic curve and
 critical {A}shkin--{T}eller multipoint correlation functions, \href{https://doi.org/10.1016/0550-3213(87)90350-6}{\textit{Nuclear
 Phys.~B}} \textbf{285} (1987), 481--503.

\end{thebibliography}
\end{document}